\documentclass[review]{elsarticle}

\usepackage{lineno,hyperref}
\modulolinenumbers[5]

\journal{Journal of Computational Physics}

\usepackage{pstricks}
\usepackage{pstricks-add}
\usepackage{pst-plot}
\usepackage{amssymb}
\usepackage{graphicx}
\usepackage{caption}
\usepackage{subcaption}
\usepackage{verbatim}
\usepackage[english]{biblio}
\usepackage{enumerate}
\usepackage{amsmath}
\usepackage{psfrag}
\usepackage{tikz}
\usepackage{verbatim}
\newtheorem{thm}{Theorem}[section]

\newtheorem{rem}[thm]{Remark}

\newcommand{\Rr}{\mathbb R}

\newcommand{\bu}{{\bf u}}
\newcommand{\bv}{{\bf v}}

\newcommand{\be}{\begin{equation}}
\newcommand{\ee}{\end{equation}}%\graphicspath{{figs/}}

\def\Xint#1{\mathchoice
{\XXint\displaystyle\textstyle{#1}}%
{\XXint\textstyle\scriptstyle{#1}}%
{\XXint\scriptstyle\scriptscriptstyle{#1}}%
{\XXint\scriptscriptstyle\scriptscriptstyle{#1}}%
\!\int}
\def\XXint#1#2#3{{\setbox0=\hbox{$#1{#2#3}{\int}$}
\vcenter{\hbox{$#2#3$}}\kern-.5\wd0}}

\def\dashint{\Xint-}

\usepackage{hyperref}
\hypersetup{
  colorlinks   = true, %Colours links instead of ugly boxes
  urlcolor     = blue, %Colour for external hyperlinks
  linkcolor    = blue, %Colour of internal links
  citecolor   = purple %Colour of citations
}

\numberwithin{equation}{section}

\begin{document}
%\maketitle

\begin{frontmatter}

\title{In-cell Discontinuous Reconstruction path-conservative methods for non conservative hyperbolic systems - Second-order extension}

%\tnotetext[mytitlenote]{Fully documented templates are available in the elsarticle package on \href{http://www.ctan.org/tex-archive/macros/latex/contrib/elsarticle}{CTAN}.}

\author[uma]{Ernesto Pimentel-Garc\'ia\corref{mycorrespondingauthor}}
\cortext[mycorrespondingauthor]{Corresponding author}
\ead{erpigar@uma.es}

\author[uma]{Manuel J. Castro}
\ead{mjcastro@uma.es}

\author[versailles]{Christophe Chalons}
\ead{christophe.chalons@uvsq.fr}

\author[uco]{Tom\'as Morales de Luna}
\ead{tomas.morales@uco.es}

\author[uma]{Carlos Par\'es}
\ead{pares@uma.es}

\address[uma]{Departamento  de  An\'alisis  Matem\'atico,  Estad\'istica  e  Investigaci\'on  Operativa,  y  Matem\'atica  aplicada,  Universidad  de  M\'alaga,Bulevar Louis Pasteur, 31, 29010, M\'alaga, Spain. }
\address[versailles]{Laboratoire de Math\'ematiques de Versailles, UVSQ, CNRS, Universit\'e Paris-Saclay, 78035 Versailles, France.}
\address[uco]{Departamento de Matem\'aticas, Universidad de C\'ordoba, Campus de Rabanales, 14071,C\'ordoba, Spain.}

\begin{abstract}
We are interested in the numerical approximation of discontinuous solutions in non conservative hyperbolic systems. An extension to second-order of a new strategy based on in-cell discontinuous reconstructions to deal with this challenging topic is presented. This extension is based on the combination of the first-order in-cell reconstruction and the MUSCL-Hancock reconstruction.  The first-order strategy allowed in particular to capture exactly the isolated shocks and this new second-order extension keep this property. We also set the basis of an extension to high-order methods following this in-cell methodology and a way for capturing exactly more than one shock.  Several numerical tests are proposed to validate the methods for the Coupled-Burgers system, Gas dynamics equations in Lagrangian coordinates and the modified shallow water system.
\end{abstract}


\end{frontmatter}

\linenumbers

\section{Introduction}\label{S:intro}

  We consider  first
order quasi-linear PDE systems
\begin{equation}\label{sys:nonconservative}
\partial_t {\bf u} + {\mathcal{A}}({\bf u}) \partial_x {\bf u} = 0, \quad x \in \Rr, \quad t \in \Rr^{+},
\end{equation}
in which the unknown $\bu(x,t)$ takes values in an open convex set
$\Omega$ of $\mathbb{R}^N$, and
$\mathcal{A}(\bf u)$ is a smooth locally bounded map from $\Omega$ to $\mathcal{M}_{N\times
N}(\mathbb{R})$. The system is
supposed to be strictly hyperbolic and the characteristic fields
$R_i(\bu)$, $i=1,\ldots ,N$, are supposed to be either genuinely
nonlinear:
$$
\nabla \lambda _i(\bu) \cdot R_i(\bu) \neq 0, \quad \forall \,
\bu\in\Omega,
$$
or linearly degenerate:
$$
\nabla \lambda _i(\bu) \cdot R_i(\bu) = 0, \quad \forall \, \bu \in \Omega.
$$
Here, $\lambda_1(\bu),\ldots,\lambda_N(\bu)$ represent the eigenvalues
of $\mathcal{A}(\bu)$ (in increasing order) and $R_1(\bu), \ldots,
R_N(\bu)$ a set of associated eigenvectors.

Under some hypotheses of regularity for $\mathcal{A}(\bu)$, the
theory introduced by Dal Maso, LeFloch, and Murat
\cite{Maso1995} allows one to define  the nonconservative product
$\mathcal{A}(\bu)\,\partial_x \bu$  as a bounded measure for functions $\bu$ with bounded variation.
To do this,  a family of Lipschitz continuous paths $\Phi: [0,1]
\times \Omega \times \Omega \to \Omega$ has to be prescribed, which must
satisfy certain regularity and  compatibility conditions, in particular
\begin{equation} \label{cond1}
\Phi(0;\bu_l, \bu_r) = \bu_l, \qquad  \Phi(1;\bu_l, \bu_r) = \bu_r,\quad \forall \bu_l, \bu_r \in \Omega,
\end{equation}
and
\begin{equation} \label{cond2}
\Phi(s; \bu, \bu) = \bu, \quad \forall \bu \in \Omega.
\end{equation}
The interested reader
is addressed to \cite{Maso1995}  for a rigorous and complete
presentation of this theory. 
 The family of paths can be understood as a tool to
give a sense to integrals of the form
$$
\int_a^b \mathcal{A}(\bu(x))\,\partial_x \bu(x) \, dx,
$$
for functions $\bu$ with jump discontinuities. More precisely,
given a bounded variation function $\bu:[a,b] \mapsto \Omega$, we define:
\begin{equation}\label{wint}
\dashint_a^b \mathcal{A}(\bu(x))\,\partial_x \bu(x) \, dx =   \int_a^b \mathcal{A}(\bu(x))\,\partial_x \bu(x) \,dx  
+ \sum_m \int_0^1\mathcal{A}(\Phi(s;\bu_m^-,\bu_m^+))\frac{\partial\Phi}{\partial
s}(s;\bu_m^-,\bu_m^+)\,ds.
\end{equation}
In this definition, $\bu_m^-$ and $\bu_m^+$ represent, respectively,  the
limits of $\bu$ to the left and right of its $m$th discontinuity. Observe that, in \eqref{wint}, the family
of paths has been used to determine the Dirac measures placed at
the discontinuities of $\bu$.

If such a mathematical definition of the nonconservative products is assumed to define the concept of weak solution,  the generalized Rankine-Hugoniot condition:
\begin{equation} \label{gen_R-H}
 \int_0^1 \mathcal{A}(\Phi(s;\bu^-\bu^+))
\frac{\partial\Phi}{\partial s}(s;\bu^-,\bu^+)\, ds  = \sigma (\bu^+ - \bu^-)
\end{equation}
has to be satisfied across an admissible discontinuity. Here,
$\sigma $ is the speed of propagation of the discontinuity, and
$\bu^-$ and $\bu^+$ are the left and right limits of the solution at
the discontinuity. 

Once the family of paths has been prescribed, a concept of entropy
is required, as it happens for systems of conservation laws, that may be given by an entropy pair  or by Lax's entropy criterion. 

Since the concept of weak solution depends on the family
of paths, which is a priori arbitrary, the crucial question is
how to choose the 'good' family of paths. In fact, when the
hyperbolic system is the vanishing-viscosity limit of the
parabolic problems
\begin{equation}
\partial_t \bu^\eps + \mathcal{A}(\bu^\eps)\,\partial_x \bu^\eps = \eps
(\mathcal{R}(\bu^\eps)\partial_x \bu^\eps)_x, \label{NC-p}
\end{equation}
where $\mathcal{R}(\bu)$ is any positive-definite matrix, 
the adequate family of paths should be related to the \textit{viscous
profiles}: a function $\bv$ is said to be
a viscous profile for \eqref{NC-p} linking the states $\bu^-$ and  $\bu^+$
if it satisfies
\begin{equation} \label{vpbc}
\lim_{\chi \to -\infty} \bv(\chi ) = \bu^-, \quad  \lim_{\chi \to
+\infty} \bv(\chi ) = \bu^+,  \quad \lim_{\chi \to \pm \infty} \bv'(\chi) = 0
\end{equation}
and there exists $\sigma \in \mathbb{R}$ such that the travelling wave
\begin{equation}
\label{vprofile} \bu^\eps(x,t) = \bv\left( \frac{x - \sigma
t}{\eps}\right),
\end{equation}
is a solution of \eqref{NC-p} for every $\epsilon$.
It can be easily
verified that, in order to be a viscous profile, $\bv$ has to solve the equation
\begin{equation} \label{NC-vpec}
-\xi \bv' + \mathcal{A}(\bv)\,\bv' = (\mathcal{R}(\bv)\,\bv')',
\end{equation}
with boundary conditions \eqref{vpbc}.
If there exists a viscous profile linking the states $\bu^-$ and $\bu^+$, the good choice for the path connecting the
states would be, after a reparameterization, the
viscous profile $\bv$. 

 The main difference with the conservative case is
that now every choice of viscous term $\mathcal{R}$ leads to
different  jump conditions,
while for standard conservative systems the usual Rankine-Hugoniot conditions
are always recovered independently of the choice of
the viscous term.

The definition of numerical methods for system \eqref{sys:nonconservative} that converge to the correct weak solutions is not a simple task. It is well known that,  although Lax's equivalence Theorem ensures that consistency and stability implies convergence for linear systems and methods, this is not the case in general for nonlinear problems. For instance, in the case of systems of conservation laws,  stable conservative methods may converge to solutions that are not admissible weak solutions: this is the case for Roe method that may converge to weak solutions that are not entropy solutions. In order to ensure the convergence to the right weak solutions, besides consistency and stability, entropy has to be well controlled: for instance, entropy-fix techniques have to be added to Roe method (see for example \cite{Harte83}). In the case of nonconservative systems, consistency, stability, and control of the entropy are not enough: the numerical viscosity and, in general, the numerical dissipation effects, have to be well-controlled as well (see \cite{leflochmishra} for a review on this topic).

The design of  finite-difference or finite-volume methods satisfying these four properties is difficult in general. Nevertheless,  different techniques have been introduced to overcome, at least partially, this convergence issue: \cite{Berthon},
\cite{BLFMP}, \cite{audebert2006hybrid}, \cite{berthon2002nonlinear}, \cite{castro2013entropy}, \cite{chalons2005navier}, \cite{chalons2017new}, \cite{fjordholm2012accurate}, \cite{pares2006numerical}, \cite{chalons2019path}.  In particular the path-conservative entropy stable methods introduced in \cite{castro2013entropy} and extended to DG high-order methods in \cite{hiltebrand} significantly reduce the convergence error: to do this, entropy-conservative numerical methods are first introduced that are stabilized by means of a discretization of the viscous term of the regularized equation \eqref{NC-p}. 

More recently, in \cite{chalons2019path}, an in-cell discontinuous reconstruction technique has been added to first-order path-conservative methods that allows one to  capture correctly  weak solutions with isolated shock waves. 

The goal of this article is to extend the in-cell discontinuous reconstruction methods introduced in \cite{chalons2019path} to second-order accuracy. To do this, these numerical methods will be first  written as high-order path-conservative schemes (see for example \cite{castro2006,castro2017well}) and then, depending of the smoothness of the numerical solution, a standard MUSCL-Hancock reconstruction (see \cite{van1974towards} and \cite{van1984relation}) or a discontinuous one is used in the cell to update the numerical solution. 
 
 The paper is organized as follows: In Section 2 a brief introduction to path-conservative methods is given, then in Section 3  the new family of second-order in-cell discontinuous reconstruction methods is presented. First the semi-discrete method is introduced including the description of the reconstructions in the cells; then a temporal discretization based on a second order Taylor development is introduced. The shock-capturing property of the method is then enunciated and proved. Section 4 is devoted to show numerically the efficiency of the proposed numerical scheme. More precisely, first the Coupled-Burgers nonconservative system introduced as a toy problem in
 \cite{CMP} is considered: the application of the method to this system is described and several numerical tests are shown to validate the methods.
 Next, we focus on the Gas dynamics equations in Lagrangian coordinates and the modified shallow water  system introduced in \cite{Castro2008}. These systems were used in \cite{Abgrall2010} and \cite{Castro2008} respectively to  illustrate the convergence issue of path-conservative methods when small-scale effects are not controlled: the method proposed in this paper is applied to these system to put on evidence that the convergence issue is corrected.

\section{Path-conservative methods}
According to \cite{pares2006numerical}, a numerical method for solving \eqref{sys:nonconservative} is said to be path-conservative if  it can be written in the form
\begin{equation}\label{scheme}
\bu_j^{n+1}=\bu_j^n-\frac{\Delta t}{\Delta x}\big(\mathcal{D}_{j-1/2}^{+} +
\mathcal{D}_{j+1/2}^{-}\big),
\end{equation}
where the following notation is used: 
\begin{itemize}

\item $\Delta x$ and $\Delta t$ are the space and time steps respectively. They are supposed to be constant for simplicity. 

\item  $I_j=\left[x_{j-\frac{1}{2}},x_{j+\frac{1}{2}}\right]$ are the computational cells, whose length  is $\Delta x$.

\item $t_n = n \Delta t $, $n = 0, 1 \dots$.

\item $\bu_j^n$ is the approximation of the average of the exact solution at the $j$th cell at time $t_n$, that is, 
\be
\bu_j^n \approx \frac{1}{\Delta x} \int_{x_{j-\frac{1}{2}}}^{x_{j+\frac{1}{2}}}  \bu(x,t_n) \, dr. 
\ee
\item Finally,
$$\mathcal{D}_{j+1/2}^{\pm}=\mathcal{D}^{\pm}\big(\bu_{j}^n,\bu_{j+1}^n \big), $$
where
$\mathcal{D}^-$ and $\mathcal{D}^+$ two Lipschitz continuous functions from $\Omega \times \Omega$ to $\Omega$ that satisfy
\begin{equation}\label{pc1}
\mathcal{D}^{\pm}(\bu,\bu)=0, \quad \forall \bu\in \Omega,
\end{equation}
and
\begin{equation}\label{pc2}
\mathcal{D}^-(\bu_l,\bu_r) + \mathcal{D}^+(\bu_l, \bu_r) = \int _0^1 \mathcal{A}\big(\Phi (s;\bu_l,\bu_r)\big)\frac{\partial
\Phi }{\partial s} (s;\bu_l,\bu_r)\,ds,
\end{equation}
for every  set $\bu_{l}, \bu_r \in \Omega$. 

\end{itemize}

The definition of path-conservative methods is a \textit{formal concept of consistency} for weak solutions defined on the basis of the
family of paths $\Phi$. In fact, this is a natural extension of the definition of conservative methods for systems of conservation
laws: it can be  easily shown that, if \eqref{sys:nonconservative} is a system of conservation laws, i.e. if $\mathcal{A}(\bu)$ is the Jacobian
of a flux function $F(\bu)$, then every method that is path-conservative for any family of paths can be rewritten as a conservative method (see \cite{castro2017well} for a recent review). 

This framework makes it easy to extend many well-known conservative schemes to nonconservative systems . Let us show two examples:

\begin{itemize}

% \item  Lax-Friedrichs methods:
% \begin{equation} \label{NCFLF}
% D_{LF}^\pm (\bu_l, \bu_r) = \frac{1}{2}
% \int_0^1\mathcal{A}(\Phi(s;\bu_l,\bu_r)\frac{\partial\Phi}{\partial
% s}(s;\bu_l,\bu_r)\,ds \pm  \frac{\Delta x}{2 \Delta t} (\bu_r - \bu_l).
% \end{equation}

\item  Godunov method: 
\begin{eqnarray} \label{NCFG1}
\mathcal{D}^-_{G} (\bu_l, \bu_r) =  \int_0^1\mathcal{A}(\Phi(s;\bu_l,\bu_0))\frac{\partial\Phi}{\partial
s}(s;\bu_l,\bu_0)\,ds,\\
\label{NCFG2}
\mathcal{D}^+_{G} (\bu_l, \bu_r) =  \int_0^1\mathcal{A}(\Phi(s;\bu_0,\bu_r))\frac{\partial\Phi}{\partial
s}(s;\bu_0,\bu_r)\,ds,
\end{eqnarray}
where $\bu_0$ is the value at $x=0$ of the
self-similar solution of the Riemann problem 
\begin{equation}\label{RiemannPb}
\left\{
\begin{array}{l}
\displaystyle \partial_t {\bf u} + {\mathcal{A}}({\bf u}) \partial_x {\bf u} = 0, \\
\bu(x, 0) = \begin{cases}
\bu_l &\text{if $x < 0$,}\\
\bu_r & \text{otherwise.}
\end{cases}
\end{array}
\right.
\end{equation}
If the family of paths satisfies some conditions of
compatibility with the solutions of the Riemann problems, the method can be interpreted in terms of the averages of the exact solutions of local Riemann problems in the cells,
as it happens for system of conservation laws:  see
\cite{munoz2007godunov}. 

\item Roe methods:
\begin{equation}\label{NCFR} \mathcal{D}_R^\pm (\bu_l, \bu_r) =
\mathcal{A}^\pm_\Phi(\bu_l, \bu_r) \, (\bu_r - \bu_l),
\end{equation}
where  $\mathcal{A}_\Phi(\bu_l, \bu_r)$ is a  Roe linearization of $\mathcal{A}(\bu)$  in
the sense defined by Toumi in \cite{toumi1992weak}, i.e. a function
$\mathcal{A}_\Phi\colon\Omega \times\Omega \mapsto\mathcal{M}_{N\times
N}(\mathbb{R})$ satisfying the following properties:
\begin{itemize}
\item  for each $\bu_l,\bu_r\in\Omega$,
$\mathcal{A}_\Phi(\bu_l,\bu_r)$ has $N$ distinct real eigenvalues
$\lambda_1(\bu_l, \bu_r)$, \dots, $\lambda_N(\bu_l, \bu_r)$;
\item $\mathcal{A}_\Phi(\bu,\bu)=\mathcal{A}(\bu)$, for every
$\bu\in\Omega$;
\item  for any $\bu_l,\bu_r\in\Omega$,
\begin{equation}\label{Roe-gen}
\mathcal{A}_\Phi(\bu_l,\bu_r)\,
(\bu_r-\bu_l)=\int_0^1\mathcal{A}(\Phi(s;\bu_l,\bu_r))\frac{\partial\Phi}{\partial
s}(s;\bu_l,\bu_r)\,ds.
\end{equation}
\end{itemize}
As usual $\mathcal{A}^\pm_\Phi(\bu_l, \bu_r)$ represent  the matrices whose
eigenvalues are the positive/negative parts of $\lambda_1(\bu_l, \bu_r)$, \dots, $\lambda_N(\bu_l, \bu_r)$
with same eigenvectors. 

\end{itemize}

First order path-conservative numerical schemes can be extended to high-order by using reconstruction operators:
\begin{equation}\label{eq:hopc}
    {\bf u}_{j}'(t) = -\frac{1}{\Delta x}\left(\mathcal{D}_{j+\frac{1}{2}}^{-}(t) + \mathcal{D}_{j-\frac{1}{2}}^{+}(t) + \int_{x_{j-\frac{1}{2}}}^{x_{j+\frac{1}{2}}}\mathcal{A}(P^t_{j}(x))\frac{\partial}{\partial x}P^t_{j}(x)\, dx\right),
\end{equation}
where $P^t_j(x)$ is the smooth approximation of the solution at the $j$th-cell provided by a high-order reconstruction operator from the sequence
of cell values $\{\bu_j(t) \}$ and  
$$\mathcal{D}^\pm_{j+1/2}(t) =\mathcal{D}^\pm_{j+1/2}(\bu^-_{j+1/2}(t),\bu^+_{j+1/2}(t)),$$
where $\bu^-_{j + 1/2}(t) = P^t_{j}(x_{j+\frac{1}{2}})$ and $\bu^-_{j+1/2}(t) = P_{j+1}^{t}(x_{j+\frac{1}{2}})$ 
(see \cite{castro2017well} for details).

In \cite{Castro2008} it was shown that, if the numerical solutions provided by a path-conservative method converge uniformly in the sense of graphs as $\Delta x \to 0$, the limit is a weak solution according to the chosen family of  paths. Nevertheless, this notion of  convergence is too strong and the  numerical solutions provided by finite-difference or finite-volume methods do not converge usually in this sense. This is not to say that path-conservative methods do not converge: in practice, it can be observed that numerical methods like the extensions of Godunov or Roe schemes described in the previous section converge in $L^1$-norm under the usual CFL condition. What happens is that the limit may be a weak solution according  to a different family of paths, i.e. it is a classical solution in the smoothness regions but its discontinuities satisfy a jump condition \eqref{gen_R-H} different of the expected one: see \cite{Castro2008}, \cite{Abgrall2010}.
In fact, the family of paths that controls the jump conditions satisfied by the limits of the numerical solutions is related to the viscous profiles of the equivalent equation of the method: see \cite{Castro2008}. If, for instance, the family of paths is based on the viscous profiles related to a regularization \eqref{NC-p}, the leading terms in the equivalent equation that represent the numerical viscosity of the scheme may not match the viscous term in \eqref{NC-p}.

Observe that, as  it has been mentioned before, the definition of path-conservative method is a formal notion of consistency. Nevertheless, as pointed out before, this consistency, together with stability and control of the entropy, is not enough to ensure the convergence towards the correct weak solution:  the numerical viscosity and, in general, the numerical dissipation effects, have to be well-controlled. Let us stress, before finishing this section, that:

\begin{itemize}
    \item The convergence to wrong weak solutions is not due to the consistency property,  but to the lack of control of the small-scale effects in the numerical solutions.
    \item This convergence issue affects to every methods in which the small-scale effects are not controlled, whether its consistency is based on the notion of path-conservative method or not. 
    \item It is possible to design path-conservative methods that overcome, at least partially, this difficulty, as shown in \cite{chalons2019path} or \cite{castro2013entropy}.
\end{itemize}

\section{Second-order in-cell discontinuous reconstruction path-conservative methods}\label{sec:methods}
In this section, a new numerical method of the form \eqref{eq:hopc} is described.  The scheme is based on a first-order path-conservative numerical method with fluctuation functions $\mathcal{D}^\pm$, which is combined with a particular novel reconstruction operator.  A standard second-order reconstruction operator in smoothness regions is used, while a discontinuous reconstruction operator close to discontinuities is performed, so that numerical viscosity is removed in the non-smooth regions. 
\subsection{Semi-discrete method}
Once the numerical approximations $\bu_j^n$
of the averages of the solutions have been computed at time $t_n = n \Delta t$, the first step is to mark the cells $I_j$ where a discontinuity is present. More explicitly, the cells such that the solution of the Riemann problem consisting of
\eqref{sys:nonconservative} with initial conditions
\begin{equation} \label{RPj}
\bu(x,0) = \begin{cases} 
\bu^n_{j-1} & \text{if $x<0$,}\\
\bu^n_{j+1} & \text{if $x >0$,}
\end{cases}
\end{equation}
involves a shock wave.  Let us denote by $\mathcal{M}_n$ the set of indices of the marked cells, i.e.
\begin{equation}\label{markedcells}
\mathcal{M}_n = \{ j \text{ s.t. the solution of the Riemann problem \eqref{sys:nonconservative}, \eqref{RPj} involves a shock wave} \}.
\end{equation}

To advance in time the following semi-discrete numerical method is considered:
\begin{equation}\label{eq:semi-discrete}
    {\bf u}_{j}'(t) = -\frac{1}{\Delta x}\left(\mathcal{D}_{j+\frac{1}{2}}^{-}(t) + \mathcal{D}_{j-\frac{1}{2}}^{+}(t) + \int_{x_{j-\frac{1}{2}}}^{x_{j+\frac{1}{2}}}\mathcal{A}(P_{j}^{n}(x,t))\frac{\partial}{\partial x}P_{j}^{n}(x,t)dx\right),
    \quad t\geq t_n,
\end{equation}
where
$${\bf u}_{j}(t) \approx  \frac{1}{\Delta x}\int_{x_{j-\frac{1}{2}}}^{x_{j+\frac{1}{2}}} {\bf u}(x,t)dx, $$
$$\mathcal{D}^\pm_{j+1/2}(t) =\mathcal{D}^\pm_{j+1/2}(\bu^-_{j+ 1/2}(t),\bu^+_{j+1/2}(t)), $$
with
$$\bu^-_{j+1/2}(t) = P_{j}^{n}(x_{j+\frac{1}{2}},t), \quad \bu^+_{j+1/2}(t) = P_{j+1}^{n}(x_{j+\frac{1}{2}},t),$$
and $P_{j}^{n}(x,t)$ is defined as follows:
    \begin{itemize}
        \item If $j-1,j,j+1 \notin \mathcal{M}_n$ then $P_j^n$ is the approximation of the first degree Taylor polynomial of the solution  given by: 
        $$P_{j}^{n}(x,t)= \bu_{j}^{n} + \widetilde{\partial_x\bu}_{j}^{n}(x-x_{j})- \mathcal{A}(\bu_{j}^{n})\widetilde{\partial_x\bu}_{j}^{n}(t-t^{n}).$$
        Here, $\widetilde{\partial_x\bu}_{j}^{n}$ is the \textit{minmod} approximation of the first order spacial derivative of 
        $\bu$ at $x_j$ at time $t_n$, whose $k$th component is given by
        $$
        \left(\widetilde{\partial_x\bu}_{j}^{n}\right)_k = \rm{minmod}\left(\alpha \frac{u_{j+1,k}^{n}-u_{j,k}^{n}}{\Delta x}, \frac{u_{j+1,k}^{n}- u_{j-1,k}^{n}}{2\Delta x}, \alpha\frac{u_{j,k}^{n}-u_{j-1,k}^{n}}{\Delta x}\right),$$
        where $u_{j,k}^n$ represents the $k$th component of $\bu_j^n$, $\alpha$ is a parameter with $1\leq \alpha<2$ and 
$$\rm{minmod}(a,b,c) = \begin{cases}
     \min\{a,b,c\} & \text{if $a,b,c>0$,}  \\
     \max\{a,b,c\} & \text{if $a,b,c<0$,}  \\
     0 & \text{otherwise.}
     \end{cases}$$
     
        Observe that for the Taylor polynomial we have used the approximation:
        $$
        \partial_t \bu(x_i, t_n) = - \mathcal{A}(\bu(x_i, t_n))\partial_x \bu(x_i, t_n) \approx - \mathcal{A}(\bu_{j}^{n})\widetilde{\partial_x\bu}_{j}^{n}.
        $$

\item If $j \in \mathcal{M}_n$ then
        $$P_{j}^{n}(x,t) = 
\begin{cases}
\bu^n_{j,l} & \text{ if $x \leq x_{j - 1/2} + d_{j}^{n} \Delta x + \sigma_j^n (t -t_n) $,}\\
\bu^n_{j,r} & \text{otherwise}.
\end{cases}.$$
where
 $d_j^n$ is chosen so that
\begin{equation}\label{conservation}
d_j^n u_{j,l, k}^n + (1 -  d_j^n) u_{j,r, k}^n = u_{j, k}^n, 
\end{equation}
for some index $k \in \{ 1, \dots, N \}$; and $\sigma_j^n$, $\bu^n_{j,l}$, and $\bu^n_{j,r}$ are chosen so that if $\bu^n_{j-1}$ and
$\bu^n_{j+1}$ may be  linked by an admissible discontinuity with speed $\sigma$, then
\begin{equation}\label{speedstates}
\bu^n_{j,l} = \bu^n_{j-1}, \quad \bu^n_{j,r} = \bu^n_{j+1}, \quad \sigma_j^n = \sigma.
\end{equation}
 Observe that this in-cell discontinuous reconstruction can only be done
if $0 \leq d_j^n \leq 1$, i.e. if
$$
0 \leq \frac{u_{j,r,k}^n - u_{j,k}^n}{u_{j,r,k}^n - u_{j,l,k}^n} \leq 1,
$$
otherwise the index $j$ is removed from the set $\mathcal{M}_n$ and the MUSCL-Hancock reconstruction is applied in the cell.
Moreover, if $d_j^n = 1$ and $\sigma_j^n > 0$ (resp. $d_j^n = 0$ and $\sigma_j^n < 0$) the  cell is unmarked and the cell $I_{j+1}$ (resp. $I_{j-1}$) is marked if necessary: note that in these cases, the discontinuity leaves the cell $I_j$ for any $t > t_n$.
       
        \item Otherwise (i.e. if $j \notin \mathcal{M}_n$ but $j-1 \in \mathcal{M}_n$ or $j+1 \in \mathcal{M}_n$) then
        $$P_{j}^{n}(x,t)= \bu_{j}^{n}.$$
\end{itemize}

\begin{rem} In the case $j \in \mathcal{M}_n$, if one of the equations of system \eqref{sys:nonconservative}, say the $k$th one,  is a conservation law,   the index $k$ is selected in  \eqref{conservation},  so that the corresponding variable is conserved. Moreover, if there is a linear combination of the unknowns $\sum_{k=1}^N \alpha_k u_k$ that is conserved, \eqref{conservation} may be replaced by:
\begin{equation}\label{conservation2}
d_j^n \sum_{k = 1}^N \alpha_k u_{j,l, k}^n + (1 -  d_j^n) \sum_{k = 1}^N \alpha_k u_{j,r, k}^n = \sum_{k = 1}^N \alpha_k u_{j, k}^n. 
\end{equation}
If there are more than one conservation laws, the index $k$ corresponding to one of them is selected in \eqref{conservation}.

\end{rem}

\subsection{Choice of $\sigma_j^n$, $\bu^n_{j,l}$, $\bu^n_{j,r}$}\label{ss:strategy}
Two different strategies are considered here: the first one is based on the exact solutions of the Riemann problems and the second one on a Roe linearization:

\begin{itemize}

\item \textbf{First strategy:}
 Assume that the solutions for Riemann problems are explicitly known. Then, for any given marked cell $j$, we may choose  $\sigma_j^n$, $\bu^n_{j,l}$, $\bu^n_{j,r}$  as the speed,  the left, and the right states of (one of the) discontinuous waves appearing in the solution of the Riemann problem with initial data $\bu^n_{j-1}$, $\bu^n_{j+1}$.

\item \textbf{Second strategy:} If a Roe matrix is available, for any given  marked cell $j$, we may choose $\sigma_j^n$, $\bu^n_{j,l}$, $\bu^n_{j,r}$ as the speed,  the left, and the right states of one of the discontinuities appearing in the solution of the linearized Riemann problem with initial data $\bu^n_{j-1}$, $\bu^n_{j+1}$. More explicitly, an index $k^*$ is selected and then
$$
\sigma_j^n = \lambda_{k^*}(\bu^n_{j-1}, \bu^n_{j+1}), \quad \bu^n_{j,l} = \bu^n_{j-1} +  \sum_{k=1}^{k^*-1} \alpha_k R_k(\bu^n_{j-1}, \bu^n_{j+1}),
\quad \bu^n_{j,r} = \bu^n_{j,l} + \alpha_{k^*} R_{k^*}(\bu^n_{j-1}, \bu^n_{j+1}),
$$
where $\alpha_k$, $k = 1, \dots,N$ represent the coordinates of $\bu^n_{j+1} - \bu^n_{j-1}$ on the basis of eigenvectors of $\mathcal{A}_\Phi(\bu^n_{j-1}, \bu^n_{j+1})$, i.e.
$$
\bu^n_{j+1} - \bu^n_{j-1} = \sum_{k=1}^N \alpha_k R_k(\bu^n_{j-1}, \bu^n_{j+1}).
$$

\end{itemize}

It can be easily checked  that both strategies satisfy \eqref{speedstates} if the solution of the Riemann problem with initial data
$\bu^n_{j-1}$, $\bu^n_{j+1}$ consist of only one discontinuous wave. These two strategies can be easily extended  to any approximate Riemann solver.

\subsection{Time step}

The time step $\Delta t_n$ is chosen as follows:
\begin{equation}
    \Delta t_n = \min(\Delta t^c_n, \Delta t^r_n).
\end{equation}
Here
\begin{equation}
    \Delta t_n^c = CFL \min\left( \frac{ \Delta x}{ \max_{j,l}|\lambda_{j,l}|} \right)
\end{equation}
where $CFL \in (0,1)$ is the stability parameter and $\lambda_{j,l}, \dots, \lambda_{j,N}$ represent  the eigenvalues of
$\mathcal{A}(\bu_j^n)$; and
\begin{equation}
    \Delta t_n^r = \min_{j \in \mathcal{M}_n} \begin{cases}
    \displaystyle \frac{1 - d_j^n}{|\sigma_j^n|}\Delta x, & \quad if \quad \sigma_j^n>0, \\ \\
    \displaystyle \frac{d_j^n}{|\sigma_j^n|}\Delta x, & \quad if \quad \sigma_j^n<0.
    \end{cases} 
\end{equation}
Observe that, besides the stability requirement, this choice of time step ensures that no discontinuous reconstruction leaves a marked cell. 

\subsection{Fully discrete method}
Once the time step is chosen, (\ref{eq:semi-discrete}) is integrated in the interval $[t^{n}, t^{n+1}]$, with $t^{n+1}= t^n + \Delta t_n$, to
obtain: 
$$ \bu_{j}^{n+1} = \bu_j^n -\frac{1}{\Delta x}\int_{t^{n}}^{t^{n+1}}\left(\mathcal{D}_{j+\frac{1}{2}}^{-}(t) + \mathcal{D}_{j-\frac{1}{2}}^{+}(t) + \dashint_{x_{j-\frac{1}{2}}}^{x_{j+\frac{1}{2}}}\mathcal{A}(P_{j}^{n}(x,t)) \partial_x P_{j}^{n}(x,t)dx\right)dt,$$
and the mid-point rule is used to approximate the integrals in time: 
\begin{equation}
    \bu_{j}^{n+1} = \bu_{j}^{n} - \frac{\Delta t_n}{\Delta x}\left(\mathcal{D}_{j+\frac{1}{2}}^{-}(t^{n+\frac{1}{2}}) + \mathcal{D}_{j-\frac{1}{2}}^{+}(t^{n+\frac{1}{2}}) +  \dashint_{x_{j-\frac{1}{2}}}^{x_{j+\frac{1}{2}}}\mathcal{A}(P_{j}^{n}(x,t^{n+1/2}))
    \partial_x P_{j}^{n}(x,t^{n+1/2})\, dx\right).
\end{equation}

The computation of the dashed integral in this expression depends  on the cell:
\begin{enumerate}
    \item If $j-1, j, j+1 \notin \mathcal{M}_n$ the mid-point rule is  used again to approximate the integral:
 \begin{equation}   
  \int_{x_{j-\frac{1}{2}}}^{x_{j+\frac{1}{2}}}\mathcal{A}(P_{j}^{n}(x,t^{n+1/2}))  \partial_x P_{j}^{n}(x,t^{n+1/2}) \,dx \approx \Delta x \mathcal{A}(\bu_{j}^{n+\frac{1}{2}})\widetilde{\partial_x\bu}_{j}^{n},
    \end{equation} 
    where 
    $$\bu_{j}^{n+\frac{1}{2}} = P_{j}^{n}(x_{j}, t^{n+\frac{1}{2}}) = \bu_{j}^{n} - \frac{\Delta t}{2}\mathcal{A}(\bu_{j}^{n})\widetilde{\partial_x\bu}_{j}^{n}.$$
    
    \item If $j \in \mathcal{M}_n$, taking into account the definition of the  dashed integrals \eqref{wint}, one has:
     \begin{equation}   
  \dashint_{x_{j-\frac{1}{2}}}^{x_{j+\frac{1}{2}}}\mathcal{A}(P_{j}^{n}(x,t^{n+1/2}))  \partial_x P_{j}^{n}(x,t^{n+1/2}) \, dx=
  \int_0^1 \mathcal{A}(\Phi(s; \bu^n_{j,l}, \bu^n_{j,r}))\partial_s \Phi(s; \bu^n_{j,l}, \bu^n_{j,r})\, ds. 
    \end{equation}    
    Observe that, if $\bu^n_{j,l}$ and $\bu^n_{j,r}$ can be linked by a shock whose speed is $\sigma_j^n$, then  
    the generalized Rankine-Hugoniot condition \eqref{gen_R-H} leads to
     \begin{equation}   
  \dashint_{x_{j-\frac{1}{2}}}^{x_{j+\frac{1}{2}}}\mathcal{A}(P_{j}^{n}(x,t^{n+1/2}))  \partial_x P_{j}^{n}(x,t^{n+1/2}) \, dx= \sigma_j^n \left(\bu^n_{j,r}- \bu^n_{j,l}\right). 
    \end{equation}

    \item If $j\notin \mathcal{M}_n$ but $j-1 \in \mathcal{M}_n$ or $j + 1 \in \mathcal{M}_n$ then
    \begin{equation}
          \int_{x_{j-\frac{1}{2}}}^{x_{j+\frac{1}{2}}}\mathcal{A}(P_{j}^{n}(x,t^{n+1/2}))  \partial_x P_{j}^{n}(x,t^{n+1/2}) \,dx  = 0.
    \end{equation}
\end{enumerate}

The final expression of the fully discrete numerical method is then as follows:
\begin{equation}\label{eq:discrete1}
    \bu_{j}^{n+1} = \bu_{j}^{n} - \frac{\Delta t_n}{\Delta x}\left(\mathcal{D}_{j+\frac{1}{2}}^{-}(t^{n+\frac{1}{2}}) + \mathcal{D}_{j-\frac{1}{2}}^{+}(t^{n+\frac{1}{2}}) + \mathcal{D}_j\right),
\end{equation}
where
\begin{equation}\label{eq:discrete2}
 \mathcal{D}_j = 
 \begin{cases}
 \Delta x \mathcal{A}(\bu_{j}^{n+\frac{1}{2}})\widetilde{\partial_x\bu}_{j}^{n} & \text{if $j-1, j, j+1 \notin \mathcal{M}_n$;}\\
\displaystyle  \int_0^1 \mathcal{A}(\Phi(s; \bu^n_{j,l}, \bu^n_{j,r}))\partial_s \Phi(s; \bu^n_{j,l}, \bu^n_{j,r})\, ds & \text{if $j \in \mathcal{M}_n$;}\\
 0 &\text{otherwise.}
 \end{cases}
\end{equation}

Observe that the numerical method satisfies the following properties: 
\begin{itemize}
\item Far from discontinuities it coincides with the standard MUSCL-Hancock
\item Close to discontinuities it corresponds to the numerical method introduced in \cite{chalons2019path}. Nevertheless, there is a slight variation when compared with \cite{chalons2019path}: in that reference, the discontinuities are allowed to leave the marked cells and the contribution to the neighbor cells are then taken into account. While this technique allows one to avoid additional restrictions to the time step, it makes more difficult the implementation of the numerical method. Nevertheless, the technique proposed here may as well be implemented as in \cite{chalons2019path}. 
\end{itemize}

\subsection{Shock-capturing property}

Let us prove that isolated shock waves are
exactly captured by the scheme and contain no spurious numerical diffusion. Although the proof is essentially the same as in \cite{chalons2019path}, it is included for the sake of completeness. 

\begin{thm}\label{th}
Assume that $\bu_l$ and $\bu_r$ can be joined by an entropy shock of speed $\sigma$. Then, the numerical method  provides an exact numerical solution
of the Riemann problem with initial conditions
$$
\bu(x,0) = \begin{cases}
\bu_l & \text{if $x < 0$,} \\
\bu_r & \text{otherwise,}
\end{cases} 
$$
in the sense that
\begin{equation}\label{presave}
\bu_j^n = \frac{1}{\Delta x}\int_{x_{j-1/2}}^{x_{j+1/2}} \bu(x,t^n) \, dx, \quad \forall j, n
\end{equation}
where $\bu(x,t)$ is the exact solution.

\end{thm}

\begin{proof} Let us suppose that $ 0 \in I_{j^*}$ and $ 0 = x_{j^*- 1/2} + d \Delta x$, with $0 \leq d \leq 1$.  Then the initial cell averages are:
$$
\bu^0_j = \begin{cases}
\bu_l & \text{if $j< j^*$;}\\
d \bu_l + (1- d) \bu_r  & \text{if $j =  j^*$;}\\
\bu_r & \text{otherwise.}
\end{cases}
$$
If $0 < d < 1$ the only marked cell at time $t^0 = 0$  is $I_{j^*}$, i.e. $\mathcal{M}_0 = \{ j^* \}.$ The only non-constant reconstruction is then $P_0^0$ and the equalities
$$
\bu_j^1 = \bu_j^0 = \frac{1}{\Delta x}\int_{x_{j-1/2}}^{x_{j+1/2}} \bu(x,t^1) \, dx, \quad \forall j \not= j^*
$$
can be easily deduced from the definition of the numerical method.

Let us compute $\bu_{j^*}^1$. Observe that, in order to 
have \eqref{conservation}, necessarily 
$d_0^0 = d$. Therefore, since  $\bu_l$ and $\bu_r$ can be linked by an admissible discontinuity of speed $\sigma$, using 
\eqref{speedstates} one has:
$$
P_0^0(x,t) = 
\begin{cases}
\bu_{l} & \text{ if $x \leq  \sigma t $,}\\
\bu_{r} & \text{otherwise}.
\end{cases}
$$
Observe  that $P_0^0$ coincides with the exact solution. We have now:
\begin{eqnarray*}
\bu^1_{j^*} & = & \bu_{j^*}^{0} - \frac{\Delta t_0}{\Delta x}\left(\mathcal{D}_{\frac{1}{2}}^{-}(t^{\frac{1}{2}}) + \mathcal{D}_{-\frac{1}{2}}^{+}(t^{\frac{1}{2}}) + \mathcal{D}_0\right)\\
& = & \bu_{j^*}^0  - \frac{\Delta t_0}{\Delta x} \mathcal{D}_0\\
& = & \bu_{j^*}^0 - \frac{\Delta t_0}{\Delta x} \sigma (\bu_r - \bu_l)\\
& = & \left(d + \frac{\sigma \Delta t}{\Delta x} \right) \bu_l+  \left( 1 - d - \frac{\sigma \Delta t}{\Delta x} \right)\bu_r .
\end{eqnarray*}
where it has been used that
$$
\bu^-_{j-1+2}(t^{1/2}) = \bu^+_{j-1+2}(t^{1/2}) = \bu_l, 
$$
$$
\bu^-_{j+1+2}(t^{1/2}) = \bu^+_{j+1+2}(t^{1/2}) = \bu_r, 
$$
so that 
$$
\mathcal{D}_{-\frac{1}{2}}^{+}(t^{\frac{1}{2}}) = \mathcal{D}_{\frac{1}{2}}^{-}(t^{\frac{1}{2}}) = 0.
$$
On the other hand,  due to the time step restrictions one has
$$
x_{j^*-1/2} \leq x_{j^*-1/2} + d\Delta x + \sigma \Delta t  = \sigma \Delta t \leq x_{j^* + 1/2}.
$$
Thus, it can be easily checked that:
$$
\frac{1}{\Delta x} \int_{x_{j^*-1/2}}^{x_{j^* + 1/2}}\bu(x,t^1) \, dx= \bu^1_{j^*},
$$
and \eqref{presave} has been proved for $n = 1$.

If $d = 1$ (resp. $d = 0$) the only marked cell is $I_{j+1}$ (resp. $I_{j-1}$) and the proof  is similar. 

The proof of the equality \eqref{presave} for $n \geq 2$ is similar to the case $n = 1$.

\end{proof}

\section{Numerical tests}
The following numerical methods will be applied here to three nonconservative systems:
\begin{itemize}
\item O1\_noDisRec: standard first-order path-conservative Roe or Godunov (it will be indicated between parentheses) methods;
    \item O1\_DisRec: first-order path-conservative method with discontinuous reconstruction;
 \item O2\_noDisRec: second-order extension standard of the first order path-conservative  method based on the MUSCL-Hancock reconstruction;
    \item O2\_DisRec: second-order path-conservative method that combines MUSCL-Hancock and discontinuous reconstruction;
\end{itemize}

\subsection{Coupled Burgers system}
Let us first consider the toy system
\begin{equation} \label{syst_CB}
\left\{
\begin{array}{rcl}
\displaystyle \partial_t u + \partial_x \left( \frac{u^2}{2} \right) + u \partial_x v &=& 0, \\
\displaystyle \partial_t v + \partial_x \left( \frac{v^2}{2} \right) + v \partial_x u &=& 0, \\
\end{array}
\right.
\quad (x,t) \in \mathbb{R} \times \mathbb{R}^+,
\end{equation}
introduced in \cite{CMP}, 
where ${\bf u} = (u,v)^T$ belongs to the state space $\Omega=\{{\bf u} \in \mathbb{R}^2, u+v > 0\}$. This system can be written in 
the form \eqref{sys:nonconservative} with 
$$
\mathcal{A}(\bu) = \left[ \begin{array}{cc} u  & u \\ v & v \end{array} \right].
$$
The system is strictly hyperbolic in $\Omega$ with eigenvalues
$$
\lambda_1 (\bu) = 0, \quad \lambda_2 (\bu) = u + v.
$$
whose characteristic fields, given by the eigenvectors
$$
R_1(\bu) = [1 , -1]^T, \quad R_2(\bu) = [u, v]^T,
$$
are respectively linearly degenerate and genuinely nonlinear. 

The sum $u + v$ satisfies the standard Burgers equation
$$
\partial_t (u+ v) + \partial_x\left( \frac{1}{2}(u+ v)^2 \right) = 0,
$$
and thus the variable $u+v$ is conserved. 

Once the family of paths has been chosen, the simple waves of this system are:
\begin{itemize}
    \item  Stationary contact discontinuities linking states $\bu_l$, $\bu_r$ such that
    $$ u_l + v_l = u_r + v_r.$$
    
    \item Rarefactions waves joining states $\bu_l$, $\bu_r$ such that
    $$ u_l + v_l < u_r + v_r, \quad \frac{u_l}{v_l} = \frac{u_r}{v_r}.$$
    
    \item Shock waves joining states $\bu_l$ and $\bu_r$ such that
    $$ u_l + v_l > u_r + v_r$$
    that satisfy the jump condition:
    \begin{eqnarray*}
    \sigma[u] & = & \left[\frac{u^2}{2} \right] + \int_0^1 \phi_u(s; \bu_l, \bu_r) 
    \partial_s \phi_v(s; \bu_l, \bu_r) \,d s,\\
    \sigma[v] & = & \left[\frac{v^2}{2} \right] +  \int_0^1 \phi_v(s; \bu_l, \bu_r) 
    \partial_s \phi_u(s; \bu_l, \bu_r) \,ds.
    \end{eqnarray*}
    As usual, for any variable $\phi$, $[\phi]$ stands for the jump on the variable $\phi_r -\phi_l$. 
    Remark that this leads, independently of the choice of the family of paths, to    
    $$
    \sigma = \frac{u_l + v_l + u_r + v_r}{2},
    $$    
\end{itemize}

If, for instance, the family of straight segments is chosen
\begin{equation}\label{CB:segments}
\phi_u(s; \bu_l, \bu_r) = u_l + s(u_r - u_l); \quad \phi_v(s; \bu_l, \bu_r) = v_l + s(v_r - v_l),
\end{equation}
the jump conditions reduce to:
    \begin{eqnarray*}
    \sigma[u] & = & \left(\frac{u_l + u_r}{2} \right) (u_r - u_l + v_r - v_l),\\
    \sigma[v] & = & \left(\frac{v_l + v_r}{2} \right) (u_r - u_l + v_r - v_l),
    \end{eqnarray*}
   and two states can be joined by an admissible shock if
$$
 u_l + v_l > u_r + v_r, \quad \frac{u_l}{v_l} = \frac{u_r}{v_r}.
 $$
A Roe matrix is given in this case by:
\begin{equation}\label{CB:Roe}
\mathcal{A}(\bu_l, \bu_r) = \left[ \begin{array}{cc} 0.5(u_l + u_r) & 0.5(u_l + u_r) \\ 0.5 (v_l + v_r) & 0.5 (v_l + v_r)
\end{array} \right].
\end{equation}

As it will be seen in Test \hyperref[CBtest1]{1}, the corresponding Roe method captures correctly the discontinuities of the weak solutions, what puts on evidence that being path-conservative is not in itself a barrier to the convergence to the right solutions. Nevertheless this is not true for other choices of family of paths. Let us consider, for instance, the family of paths given by the viscous profiles of the regularized system:
\begin{equation} \label{syst_CB_visc}
\left\{
\begin{array}{rcl}
\displaystyle \partial_t u_ + \partial_x \left( \frac{u^2}{2} \right)+ u \partial_x v &=& \epsilon u_{xx}, \\
\displaystyle \partial_t v + \partial_x \left( \frac{v^2}{2} \right) + v \partial_x u &=& \epsilon v_{xx}, \\
\end{array}
\right.
\quad (x,t) \in \mathbb{R} \times \mathbb{R}^+,
\end{equation}
introduced in \cite{Berthon}: see this reference for the expression of the corresponding family of paths.

It will be seen in Test \hyperref[CB:test2]{2} that Godunov's method does not converge to the right weak solutions. In
\cite{chalons2019path} the in-cell discontinuous reconstruction technique has been used to correct this issue with good results. To apply this technique, a cell is marked if
$$
u^n_{j-1} + v^n_{j-1} > u^n_{j+1} + v^n_{j+1}.
$$
Strategy 1 (based on the exact solutions of the Riemann problems) is followed here to select the discontinuous reconstruction (see Subsection \ref{ss:strategy}). More precisely, in a marked cell 
the left and right states are chosen as follows:
$$
\sigma_j^n = \frac{1}{2} (u^n_{j-1} + v^n_{j-1} + u^n_{j+1} + v^n_{j+1}), \quad  \bu^n_{j,l} = \bu^*(\bu^n_{j-1}, \bu^n_{j+1} ) , \quad \bu^n_{j,r} = \bu^n_{j+1},
$$
where $ \bu^*(\bu^n_{j-1}, \bu^n_{j+1})$ represents the state at the left of the shock wave appearing in the solution of the Riemann problem. Finally, 
the conserved variable $u + v$ is chosen to determine $d_j^n$, i.e.
$$
d_j^n (u^n_{j,l} + v^n_{j,l}) + (1 -  d_j^n) (u^n_{j,r} + v^n_{j,r})= (u^n_j + v^n_j).
$$
This method is extended here to second order following Section \ref{sec:methods}.

\subsubsection*{Test 1: Coupled Burgers' equations with straight segment paths}  \label{CBtest1}
In this test case we consider the definition of weak solution related to the  family of straight segments \eqref{CB:segments} and the corresponding Roe matrix \eqref{CB:Roe}. Let us consider the following initial condition
$$\bu_{0}(x) = (u,v)_{0}(x) = \begin{cases}
     (2.0,2.0) & \text{if $x<0.5$,}  \\
     (1.0,1.0) & \text{otherwise.}
\end{cases}$$
The solution of the Riemann problem in this case consists of a shock wave joining the left and right states.

Figures \ref{fig:Burgers_Test1_O1_vs_O2_DisRec_vs_noDisRec_Roe_u} and  \ref{fig:Burgers_Test1_O1_vs_O2_DisRec_vs_noDisRec_Roe_v} compare the exact solution with the numerical approximations at time $t = 0.1$ obtained with 
O$p$\_noDisRec and O$p$\_DisRec, $p = 1,2$ using a 1000-cell mesh and CFL=0.5:  as it can be seen, in spite of the numerical diffusion added by the CFL parameter's choice,  the four  methods capture correctly the exact solution. The same comparison has been  done for a number of different Riemann problems and, in all cases, the numerical solutions converge to the weak solution. As we said before this test puts on evidence that being path-conservative is not in itself a barrier to the convergence to the right solutions.

\begin{figure}[h]
		\begin{subfigure}{0.5\textwidth}
			\includegraphics[width=1.1\linewidth]{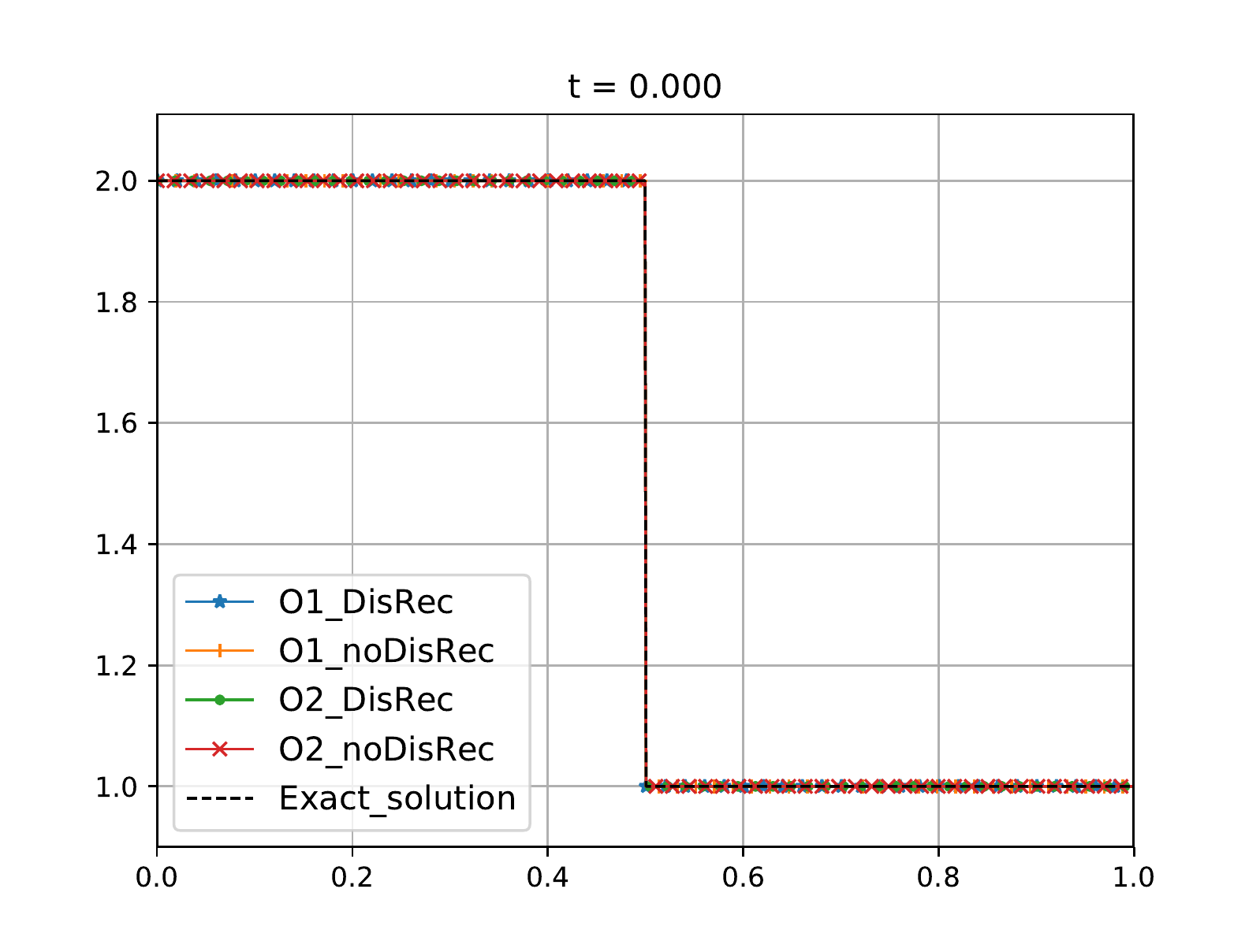}
		\end{subfigure}
		\begin{subfigure}{0.5\textwidth}
				\includegraphics[width=1.1\linewidth]{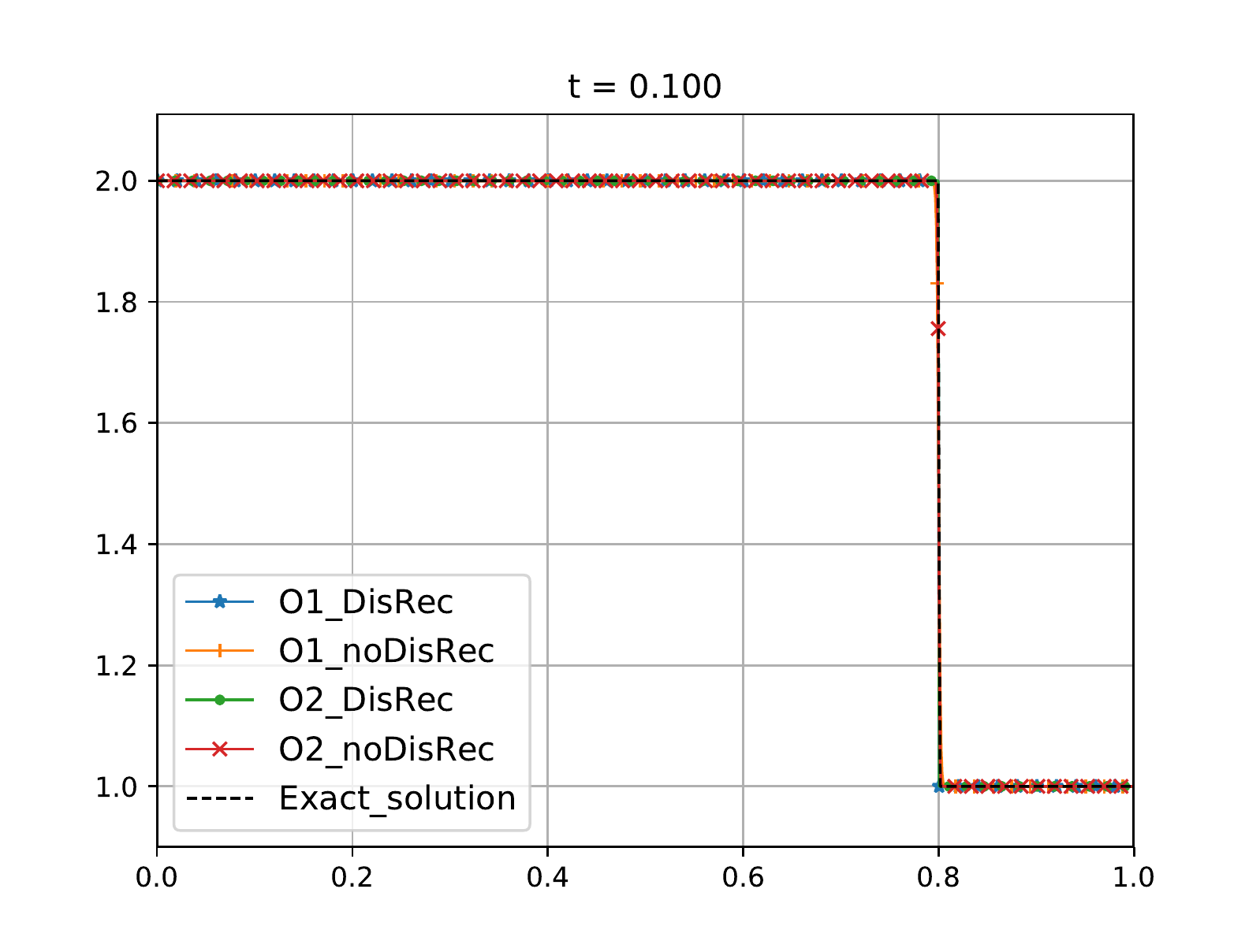}
		\end{subfigure}
		\caption{Coupled Burgers system. Test 1: variable $u$. Left: initial condition. Right: exact solution and numerical solutions obtained at time $t = 0.1$ with 1000 cells.}
		\label{fig:Burgers_Test1_O1_vs_O2_DisRec_vs_noDisRec_Roe_u}
	\end{figure}
	
	\begin{figure}[h]
		\begin{subfigure}{0.5\textwidth}
			\includegraphics[width=1.1\linewidth]{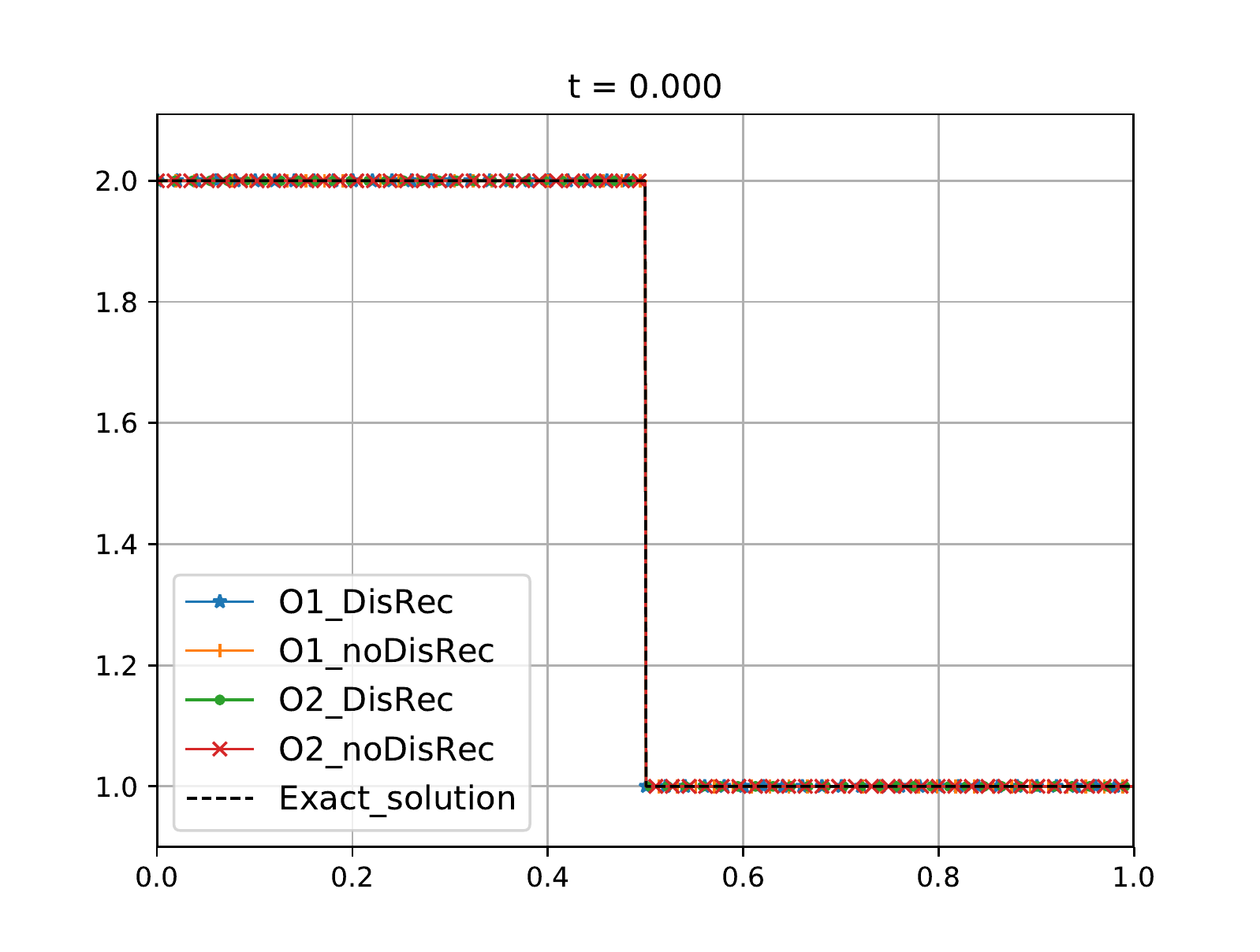}
		\end{subfigure}
		\begin{subfigure}{0.5\textwidth}
				\includegraphics[width=1.1\linewidth]{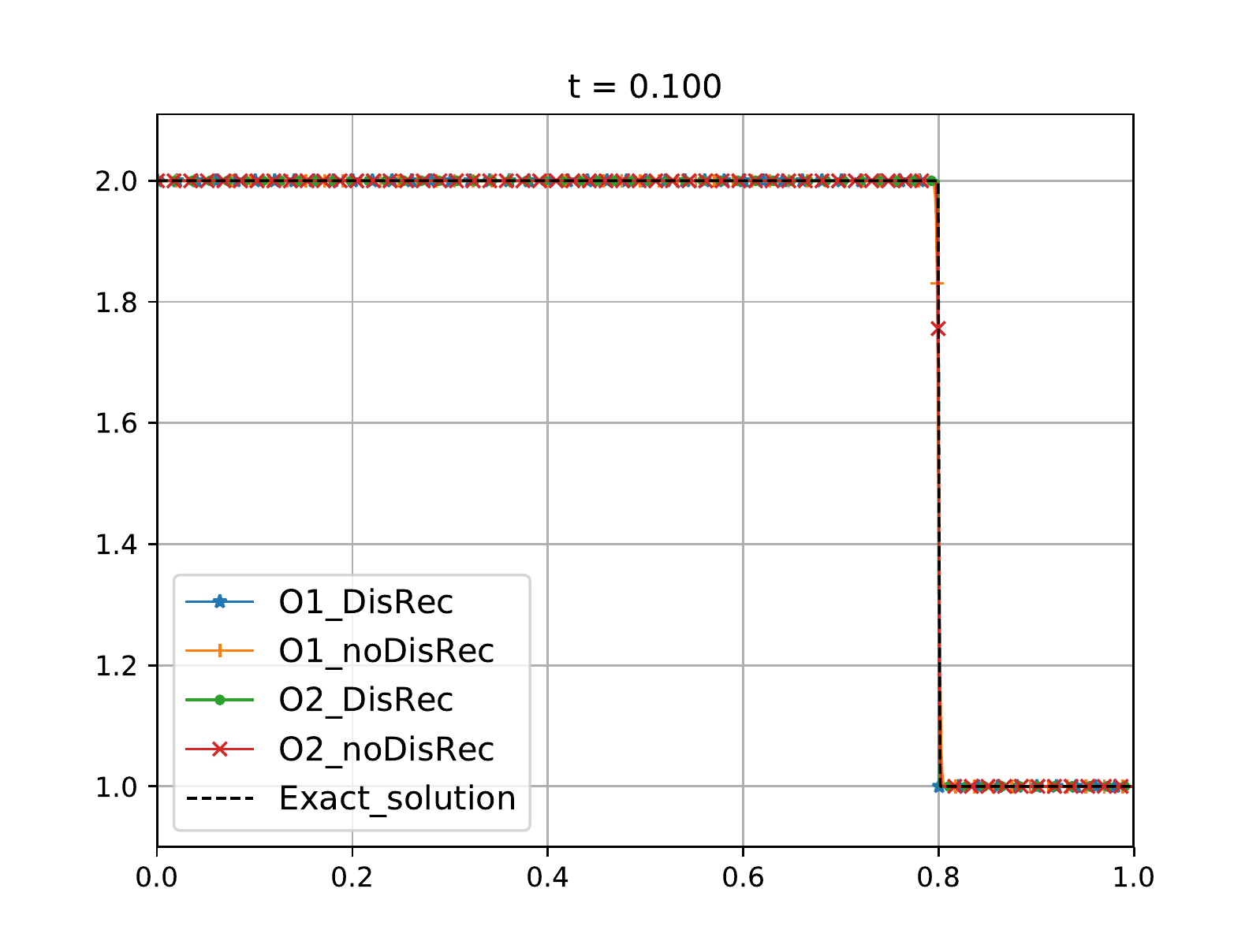}
		\end{subfigure}
		\caption{Coupled Burgers system. Test 1: variable $v$. Left: initial condition. Right: exact solution and numerical solutions obtained at time $t = 0.1$ with 1000 cells.}
		\label{fig:Burgers_Test1_O1_vs_O2_DisRec_vs_noDisRec_Roe_v}
	\end{figure}

\subsubsection*{Test 2: Isolated shock wave} \label{CB:test2}

From now on, the family of paths given by the viscous profiles of the regularized equation  \eqref{syst_CB_visc} is considered. Let us consider the following initial condition taken from \cite{castro2013entropy}
$$\bu_{0}(x) = (u,v)_{0}(x) = \begin{cases}
     (7.99,11.01) & \text{ if $x<0.5$,}  \\
     (0.25,0.75) & \text{othewise}.
\end{cases}$$
The solution of the Riemann problem consists of a shock wave joining the left and right states.

Figures \ref{fig:Burgers_Test1_O1_vs_O2_DisRec_vs_noDisRec_Godunov_u_100} and  \ref{fig:Burgers_Test1_O1_vs_O2_DisRec_vs_noDisRec_Godunov_v_100} compare the exact solution with the numerical approximations at time $t = 0.03$ obtained with O$p$\_noDisRec(Godunov) and O$p$\_DisRec(Godunov), $p = 1,2$  using a 100-cell mesh: as it can be seen Godunov's method and its second order extension do not capture the discontinuity properly what is not the case for the methods based on the discontinuous reconstruction. In Figures \ref{fig:Burgers_Test1_O1_vs_O2_DisRec_vs_noDisRec_Godunov_u} and  \ref{fig:Burgers_Test1_O1_vs_O2_DisRec_vs_noDisRec_Godunov_v} we compare the exact solution with the numerical approximations obtained with the same methods but using a 1000-cell mesh and again Godunov's method and its second order extension are not able to capture the right solution. CFL = 0.5 has been considered.

\begin{figure}[h]
		\begin{subfigure}{0.5\textwidth}
			\includegraphics[width=1.1\linewidth]{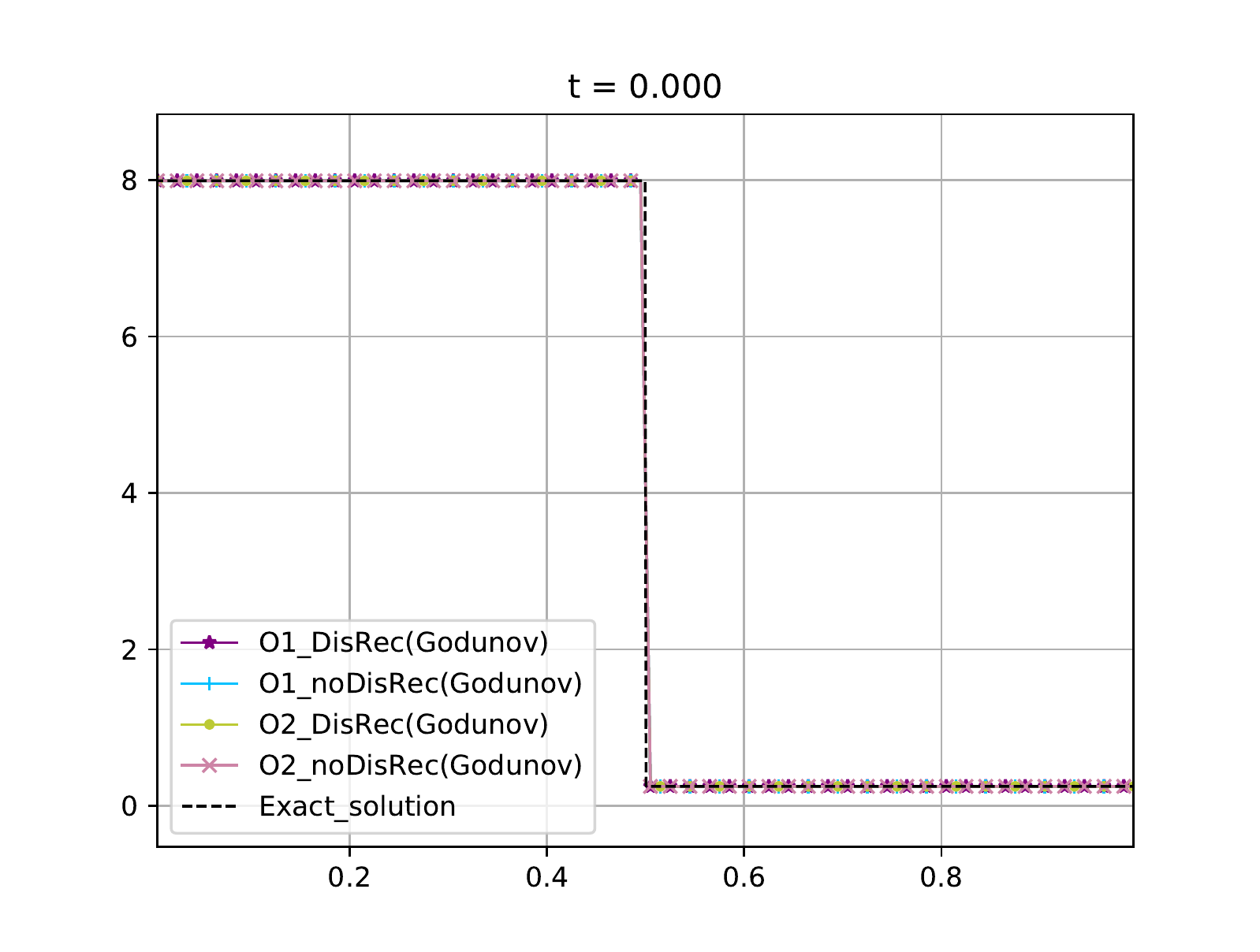}
		\end{subfigure}
		\begin{subfigure}{0.5\textwidth}
				\includegraphics[width=1.1\linewidth]{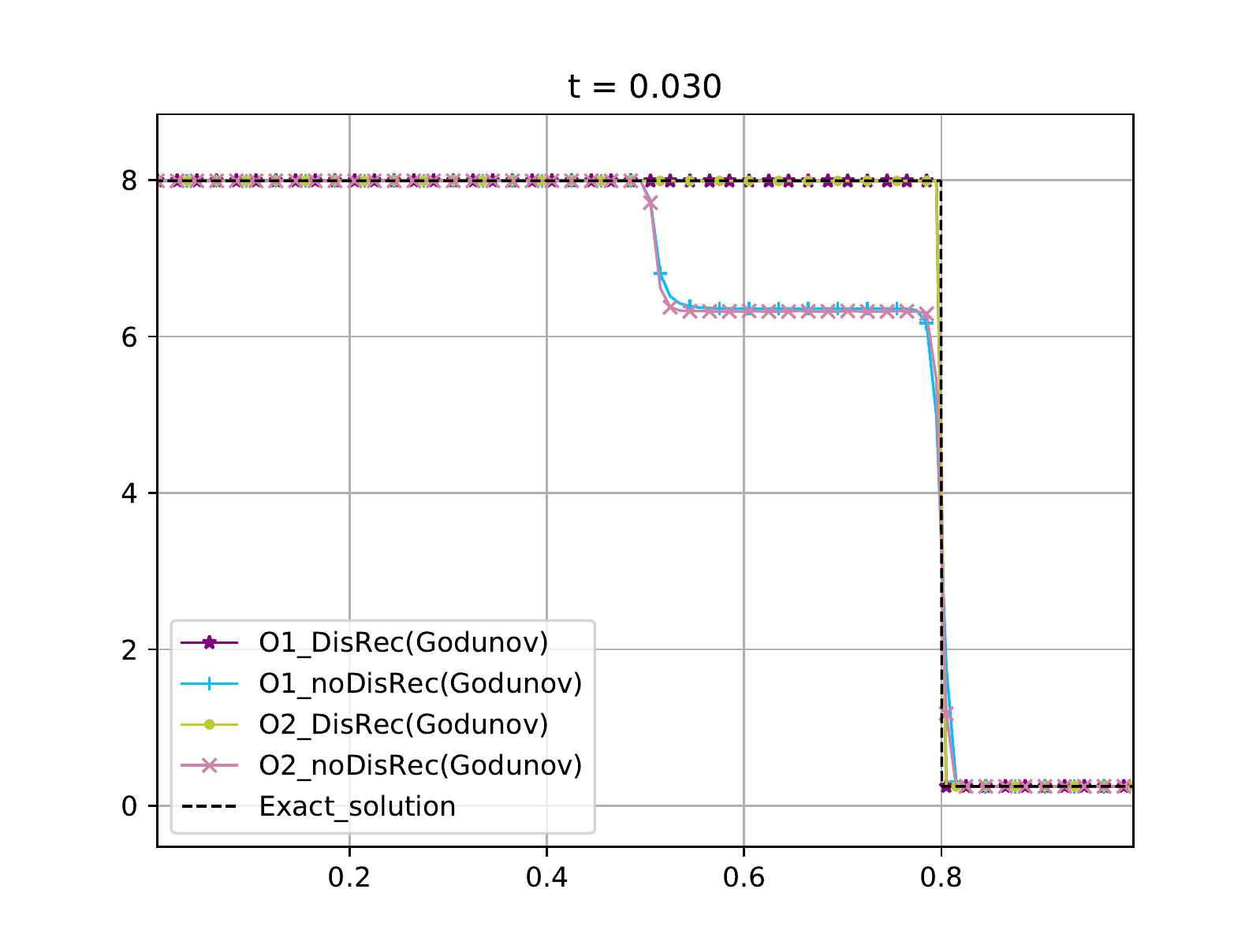}
		\end{subfigure}
		\caption{Coupled Burgers system. Test 2: variable $u$. Left: initial condition. Right: exact solution and numerical solutions obtained at time $t = 0.03$ with 100 cells.}
		\label{fig:Burgers_Test1_O1_vs_O2_DisRec_vs_noDisRec_Godunov_u_100}
	\end{figure}
	
	\begin{figure}[h]
		\begin{subfigure}{0.5\textwidth}
			\includegraphics[width=1.1\linewidth]{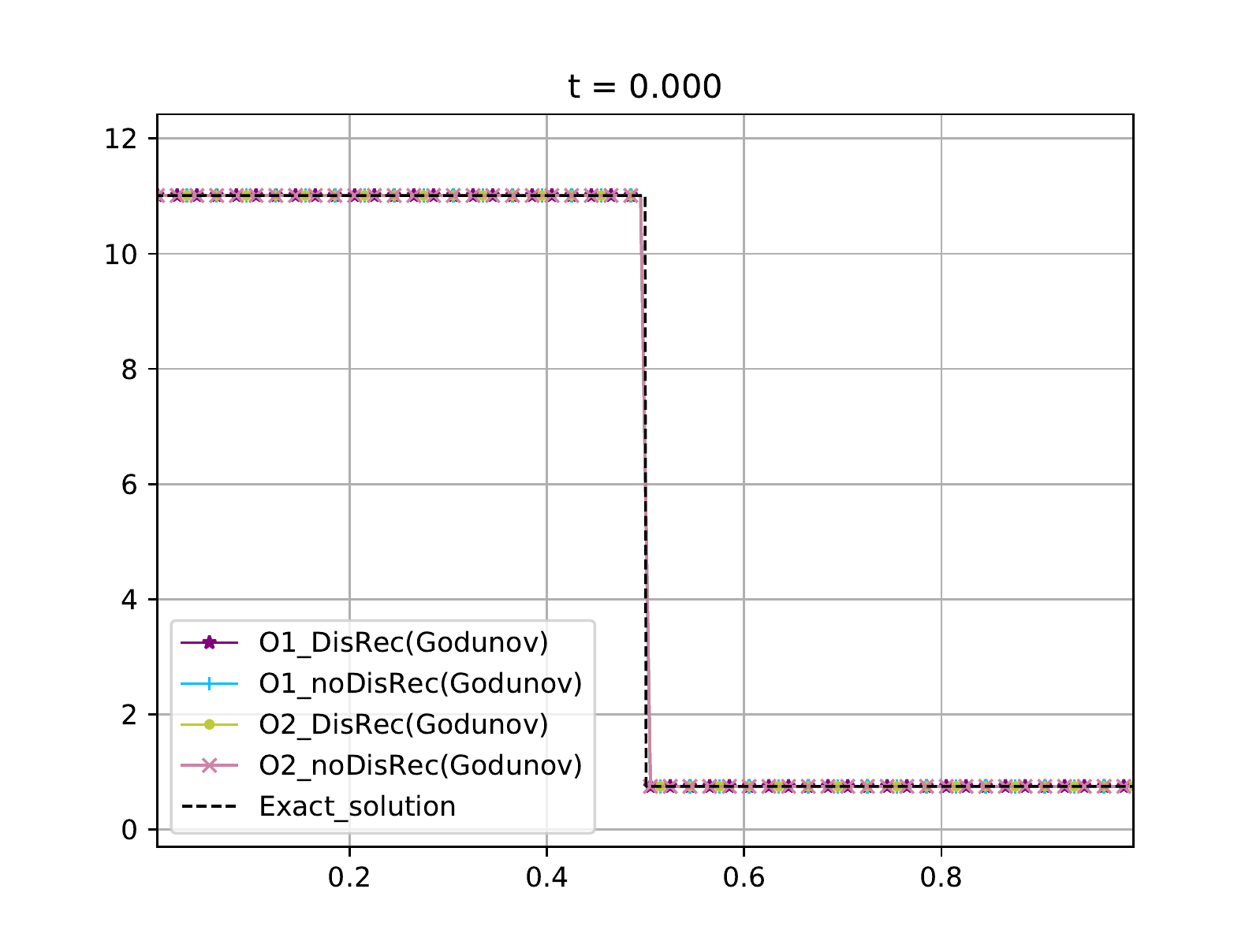}
		\end{subfigure}
		\begin{subfigure}{0.5\textwidth}
				\includegraphics[width=1.1\linewidth]{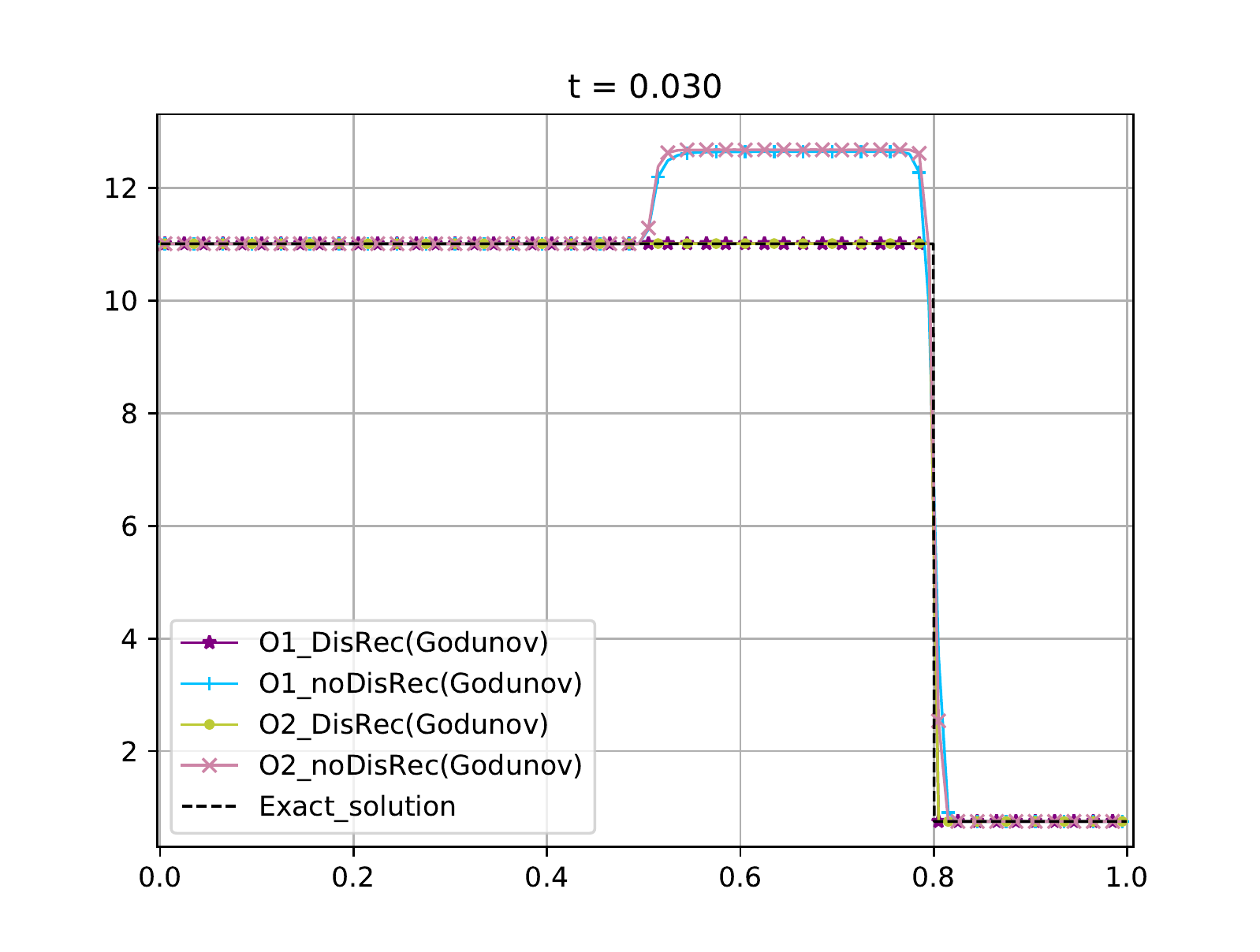}
		\end{subfigure}
		\caption{Coupled Burgers system. Test 2: variable $v$. Left: initial condition. Right: exact solution and numerical solutions obtained at time $t = 0.03$ with 100 cells.}
		\label{fig:Burgers_Test1_O1_vs_O2_DisRec_vs_noDisRec_Godunov_v_100}
	\end{figure}

\begin{figure}[h]
		\begin{subfigure}{0.5\textwidth}
			\includegraphics[width=1.1\linewidth]{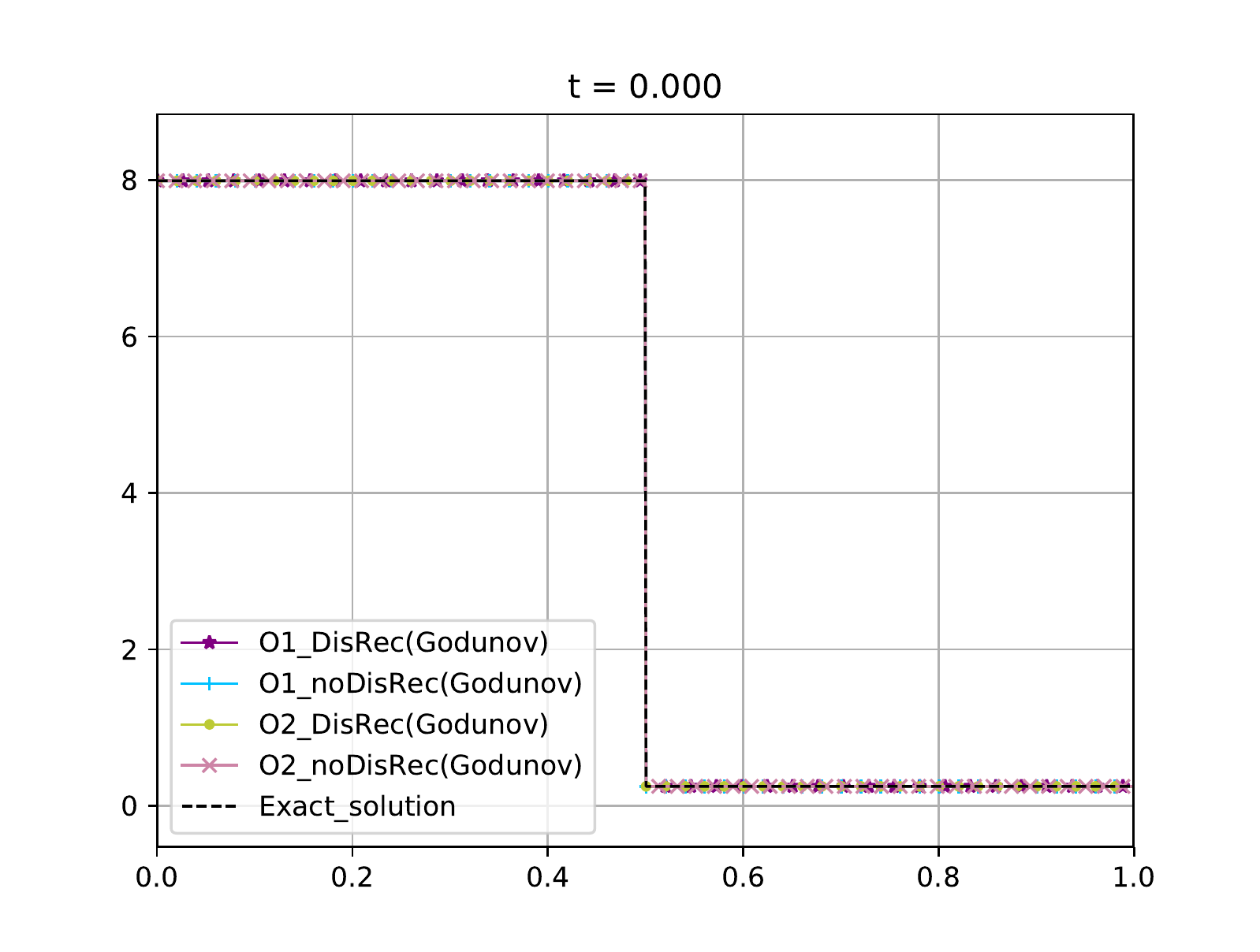}
		\end{subfigure}
		\begin{subfigure}{0.5\textwidth}
				\includegraphics[width=1.1\linewidth]{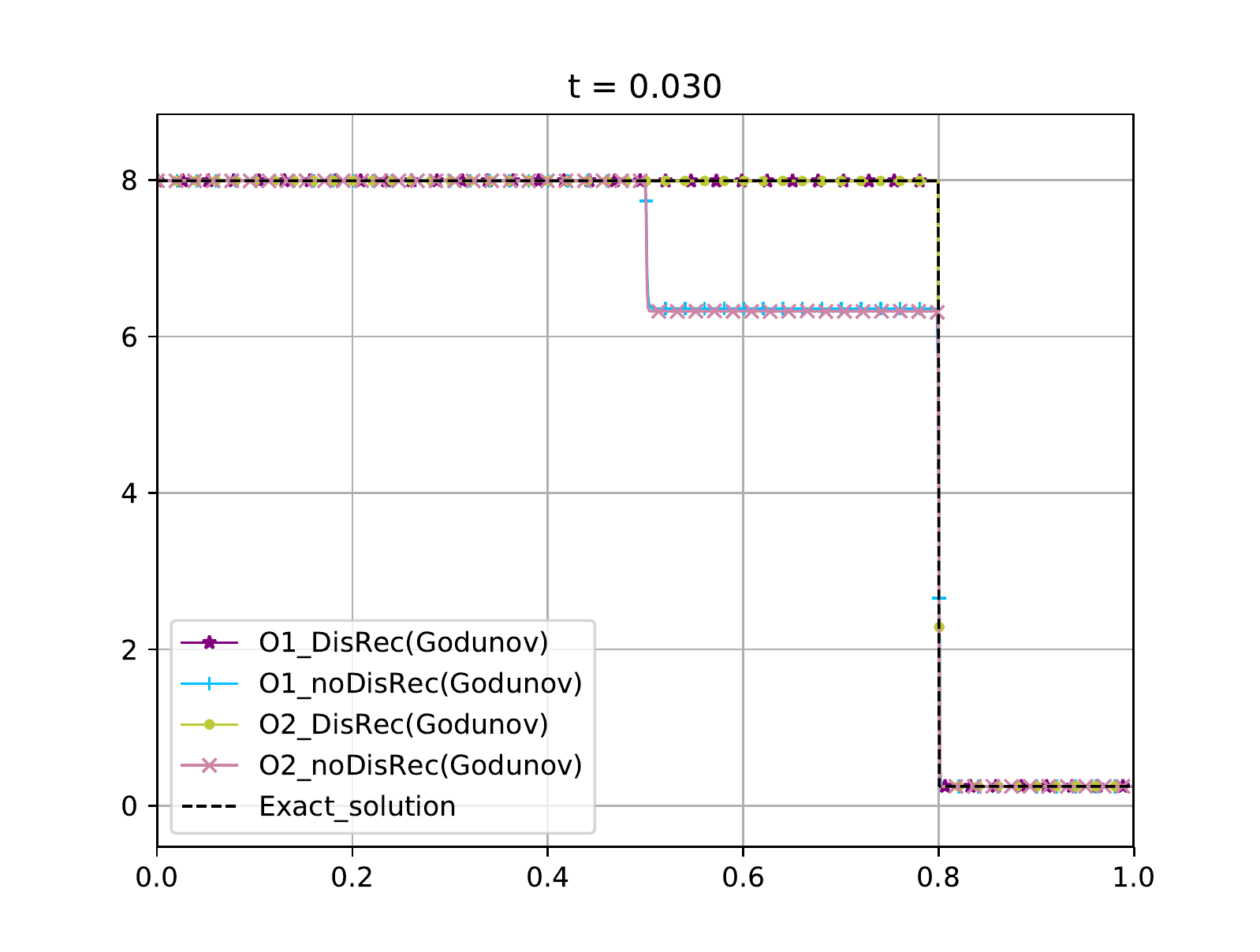}
		\end{subfigure}
		\caption{Coupled Burgers system. Test 2: variable $u$. Left: initial condition. Right: exact solution and numerical solutions obtained at time $t = 0.03$ with 1000 cells.}
		\label{fig:Burgers_Test1_O1_vs_O2_DisRec_vs_noDisRec_Godunov_u}
	\end{figure}
	
	\begin{figure}[h]
		\begin{subfigure}{0.5\textwidth}
			\includegraphics[width=1.1\linewidth]{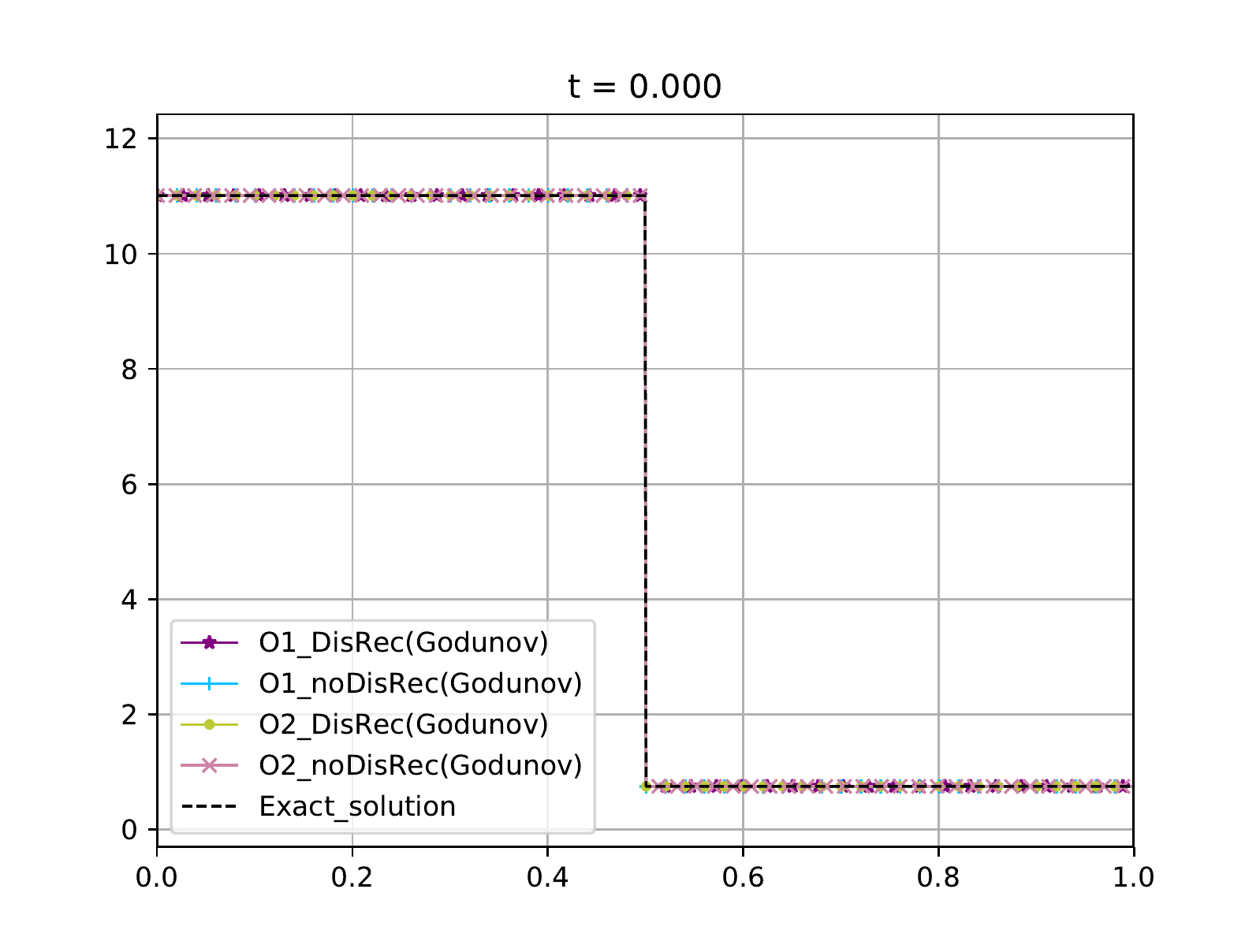}
		\end{subfigure}
		\begin{subfigure}{0.5\textwidth}
				\includegraphics[width=1.1\linewidth]{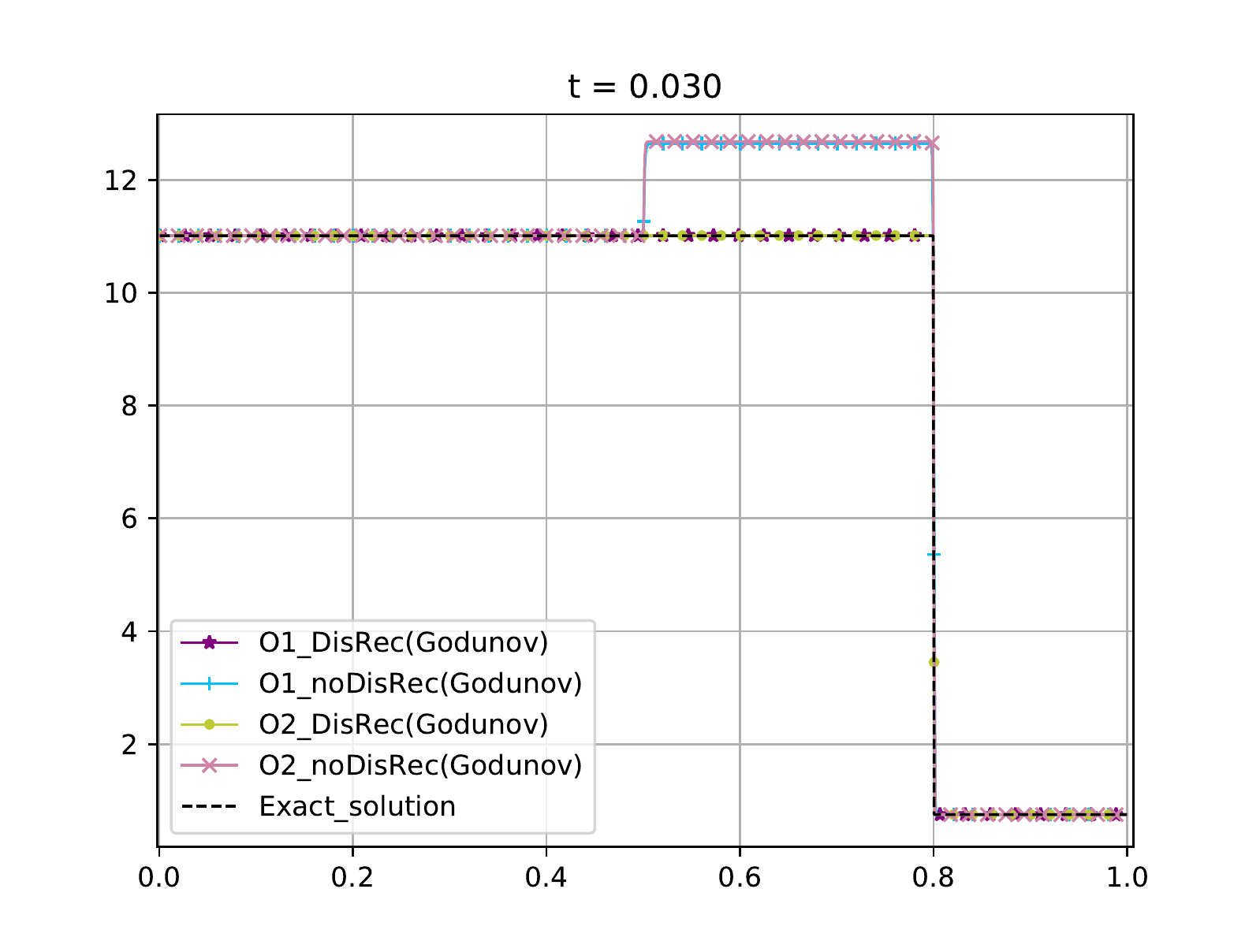}
		\end{subfigure}
		\caption{Coupled Burgers system. Test 2: variable $v$. Left: initial condition. Right: exact solution and numerical solutions obtained at time $t = 0.03$ with 1000 cells.}
		\label{fig:Burgers_Test1_O1_vs_O2_DisRec_vs_noDisRec_Godunov_v}
	\end{figure}

\subsubsection*{Test 3: Contact discontinuity + shock wave}
We  consider now the initial condition
$$\bu_{0}(x) = (u,v)_{0}(x) = \begin{cases}
     (5,1) & \text{if $x<0.5$,}  \\
     (1,2) & \text{otherwise.}
\end{cases}$$
The solution of the corresponding Riemann problems consists of a stationary contact discontinuity followed by a shock.
Figures \ref{fig:Burgers_Test2_O1_vs_O2_DisRec_vs_noDisRec_Godunov_u} and  \ref{fig:Burgers_Test2_O1_vs_O2_DisRec_vs_noDisRec_Godunov_v} show the exact and the numerical solutions at time
$t = 0.05$ using a 1000-cell mesh  and CFL = 0.5. The conclusions are the same: the in-cell discontinuous reconstruction methods of order 1 and 2 get the exact solution while the standard Godunov methods do not.

\begin{figure}[h]
		\begin{subfigure}{0.33\textwidth}
			\includegraphics[width=1.1\linewidth]{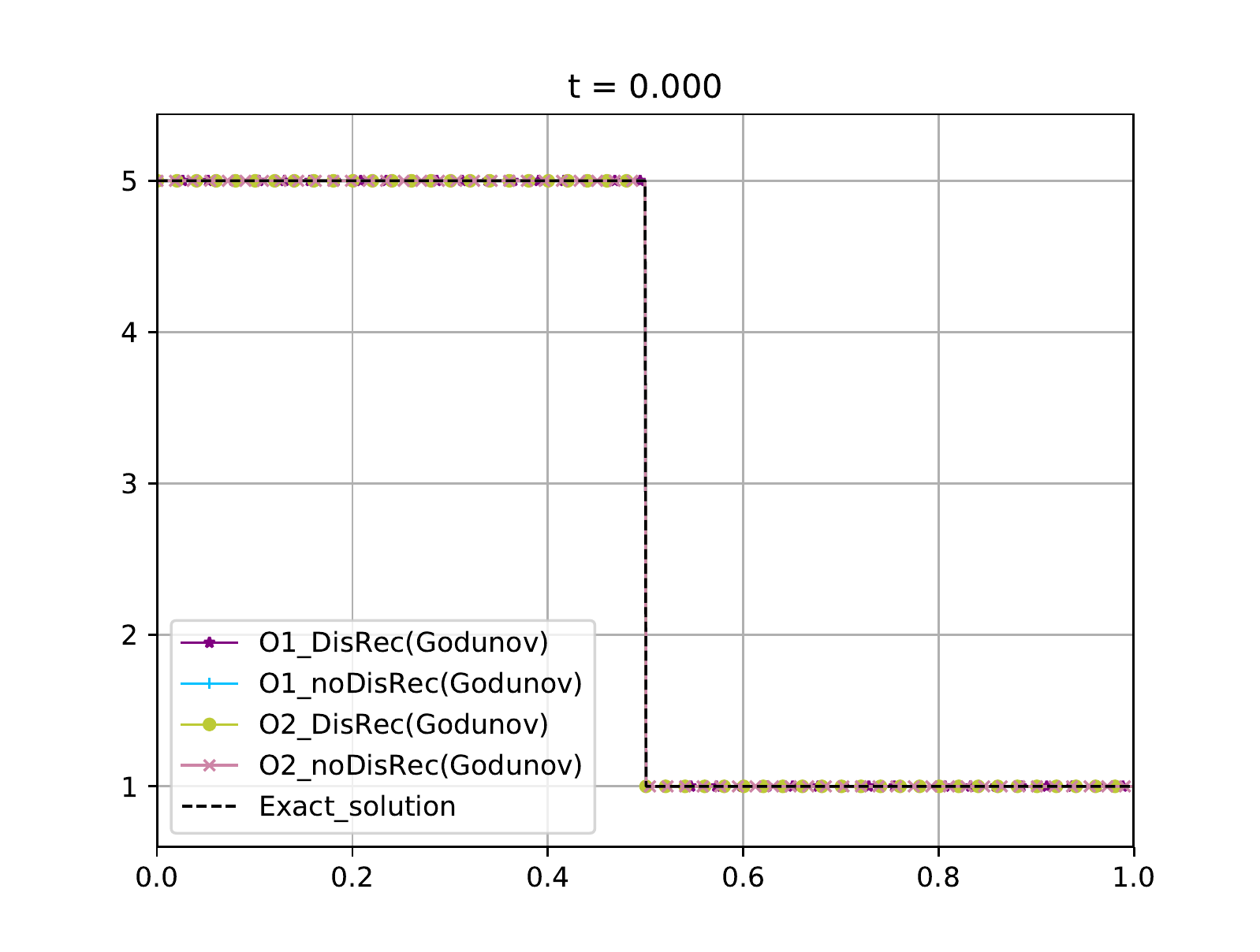}
		\end{subfigure}
		\begin{subfigure}{0.33\textwidth}
				\includegraphics[width=1.1\linewidth]{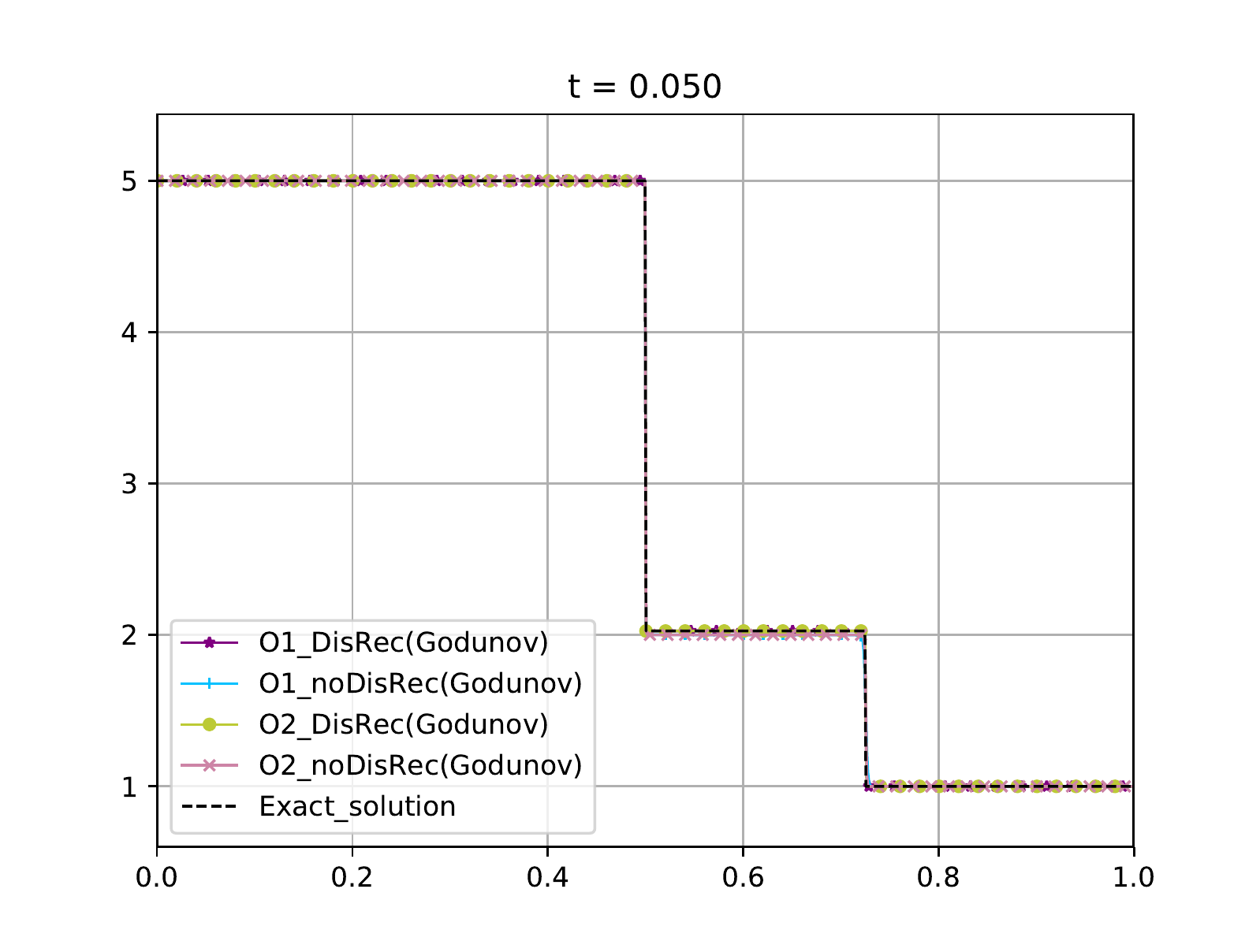}
		\end{subfigure}
		\begin{subfigure}{0.33\textwidth}
				\includegraphics[width=1.1\linewidth]{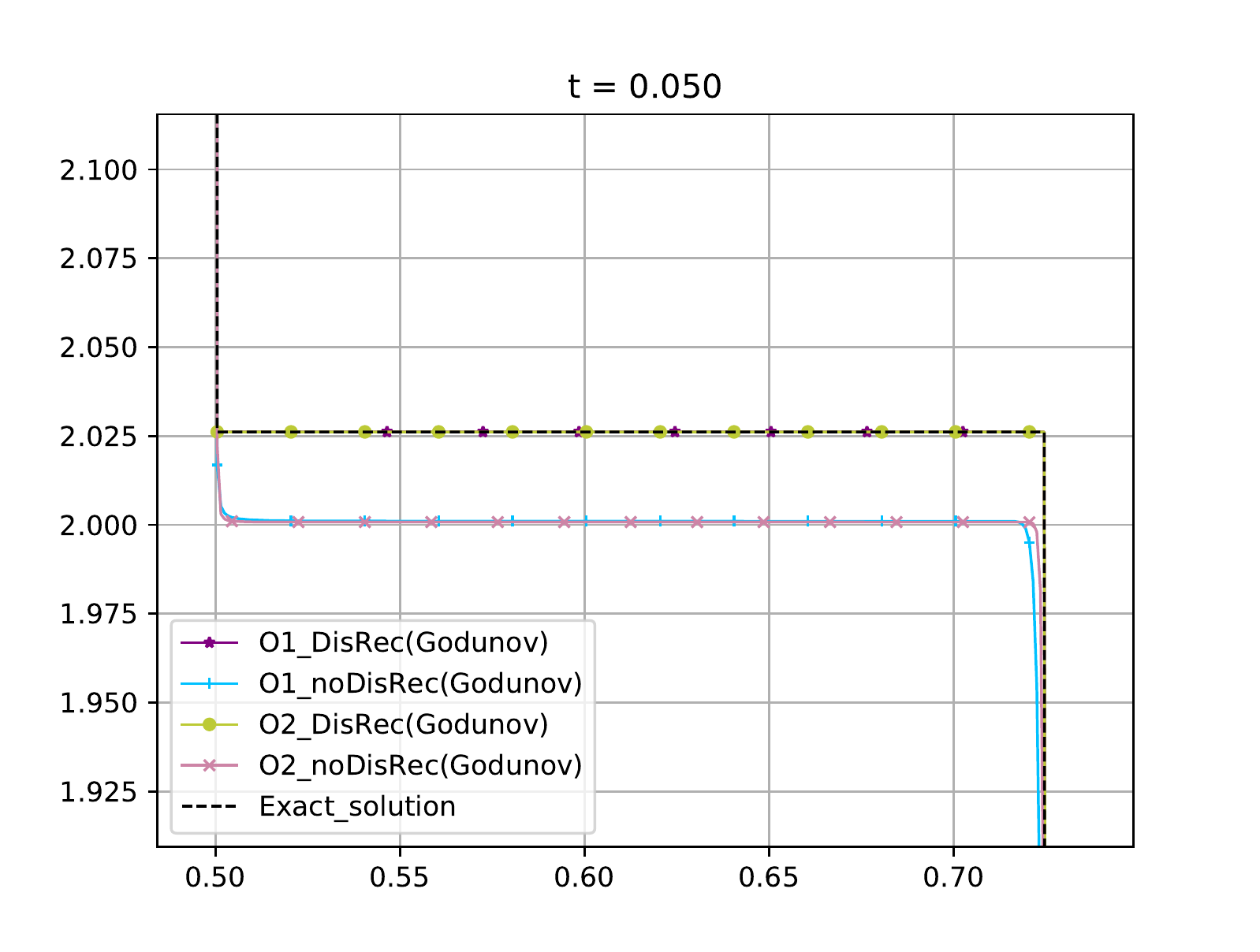}
				\caption*{Zoom}
		\end{subfigure}
		\caption{Coupled Burgers system. Test 3: variable $u$. Left: initial condition. Center: exact solution and numerical solutions obtained at time $t = 0.05$ with 1000 cells. Right: zoom.}
		\label{fig:Burgers_Test2_O1_vs_O2_DisRec_vs_noDisRec_Godunov_u}
	\end{figure}
	
	\begin{figure}[h]
		\begin{subfigure}{0.33\textwidth}
			\includegraphics[width=1.1\linewidth]{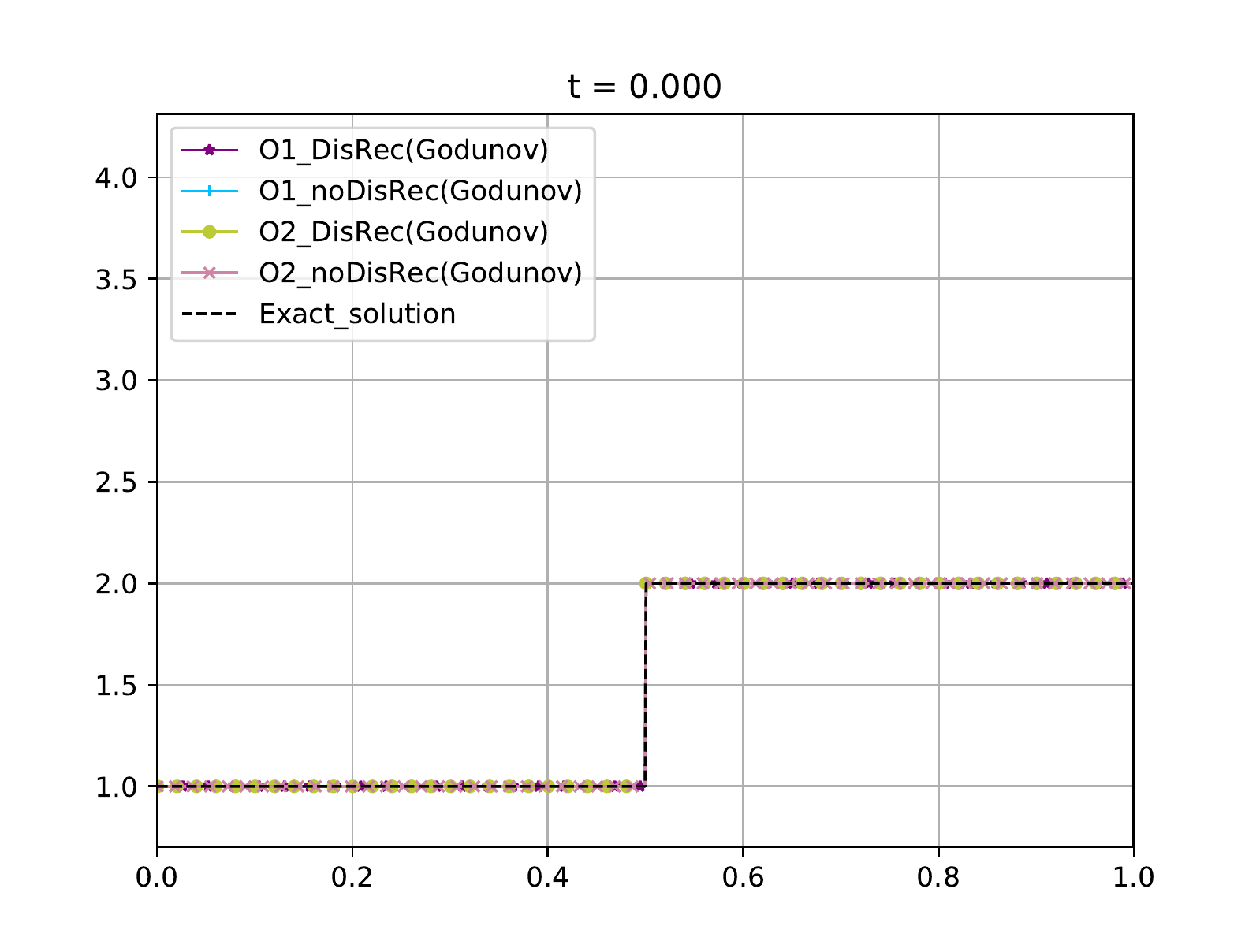}
		\end{subfigure}
		\begin{subfigure}{0.33\textwidth}
				\includegraphics[width=1.1\linewidth]{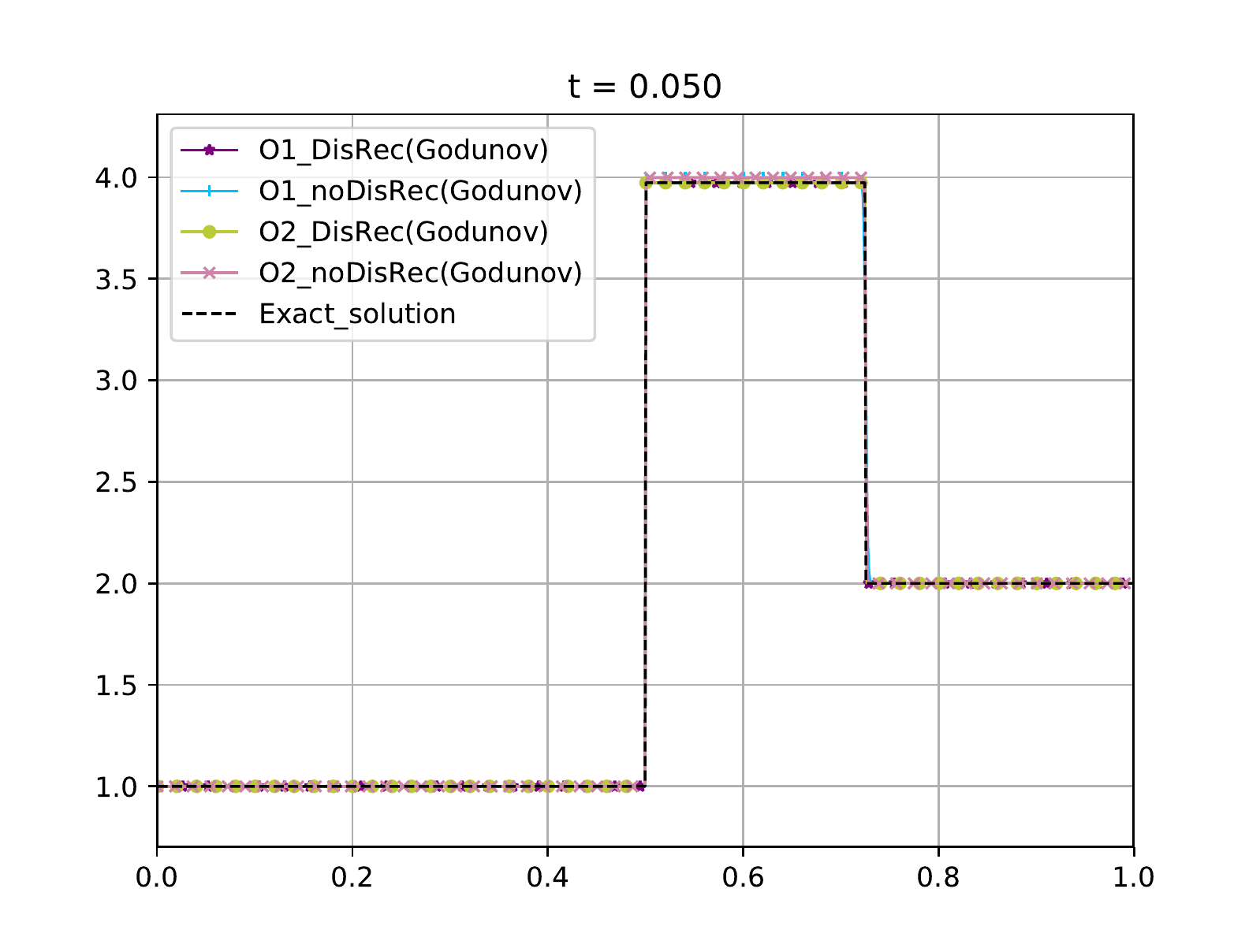}
		\end{subfigure}
		\begin{subfigure}{0.33\textwidth}
				\includegraphics[width=1.1\linewidth]{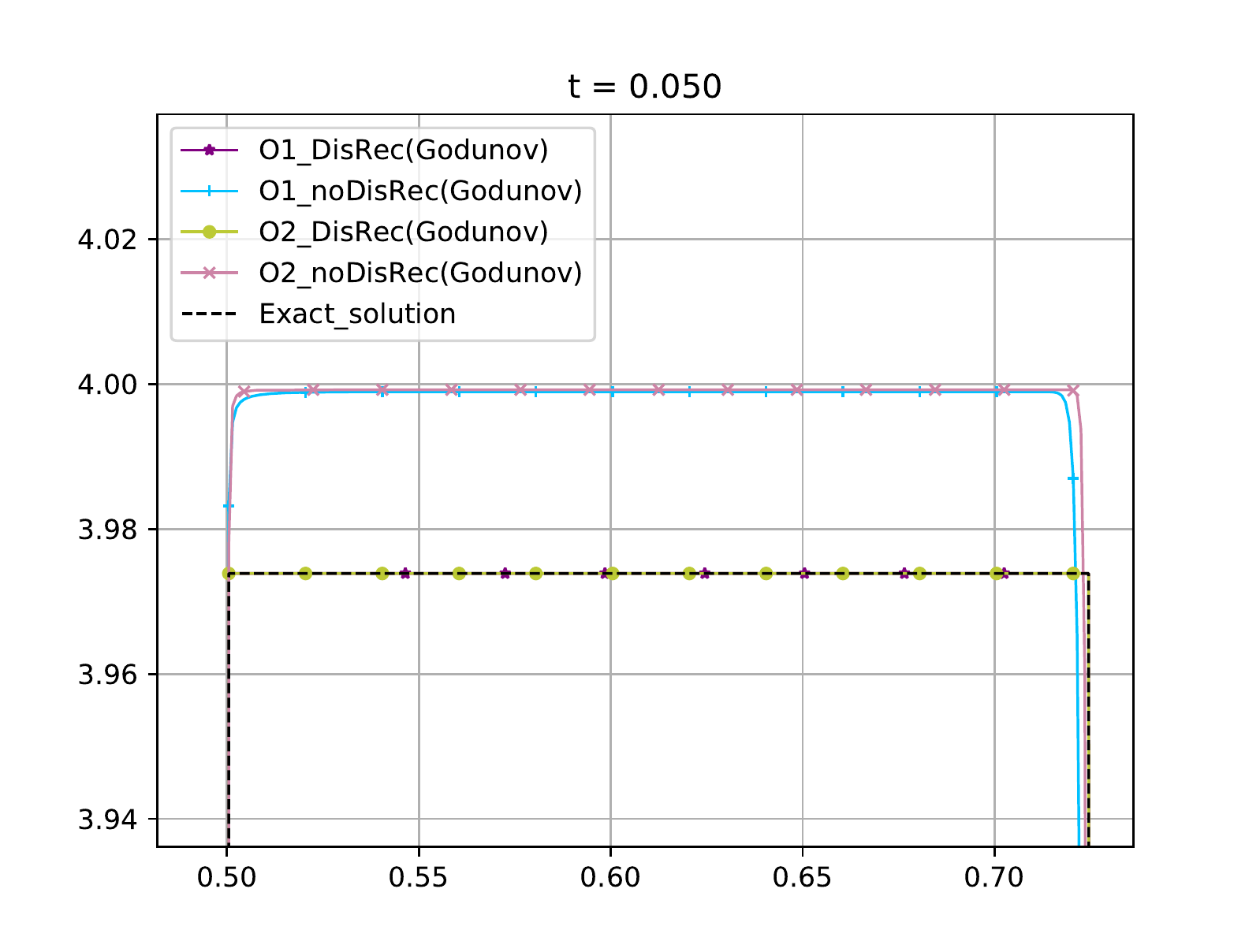}
				\caption*{Zoom}
		\end{subfigure}
		\caption{Coupled Burgers system. Test 3: variable $u$. Left: initial condition. Center: exact solution and numerical solutions obtained at time $t = 0.05$ with 1000 cells. Right: zoom.}
		\label{fig:Burgers_Test2_O1_vs_O2_DisRec_vs_noDisRec_Godunov_v}
	\end{figure}
	
\subsubsection*{Test 4: Contact discontinuity + rarefaction}
We consider the initial condition
$$\bu_{0}(x) = (u,v)_{0}(x) =\begin{cases}
     (1,2) & \text{ if $x<0.5$,}  \\
     (5,1) & \text{otherwise.}
\end{cases}$$
The solution of the corresponding Riemann problem consists  of a stationary contact discontinuity followed by a rarefaction.

Figures \ref{fig:Burgers_Test3_O1_vs_O2_DisRec_vs_noDisRec_Godunov_u} and  \ref{fig:Burgers_Test3_O1_vs_O2_DisRec_vs_noDisRec_Godunov_v} show the exact and the numerical solutions at time
$t = 0.05$ using a 1000-cell mesh. In this case all the methods converge to the exact solution but the second order one captures better the solution, as expected. 

\begin{figure}[h]
		\begin{subfigure}{0.33\textwidth}
			\includegraphics[width=1.1\linewidth]{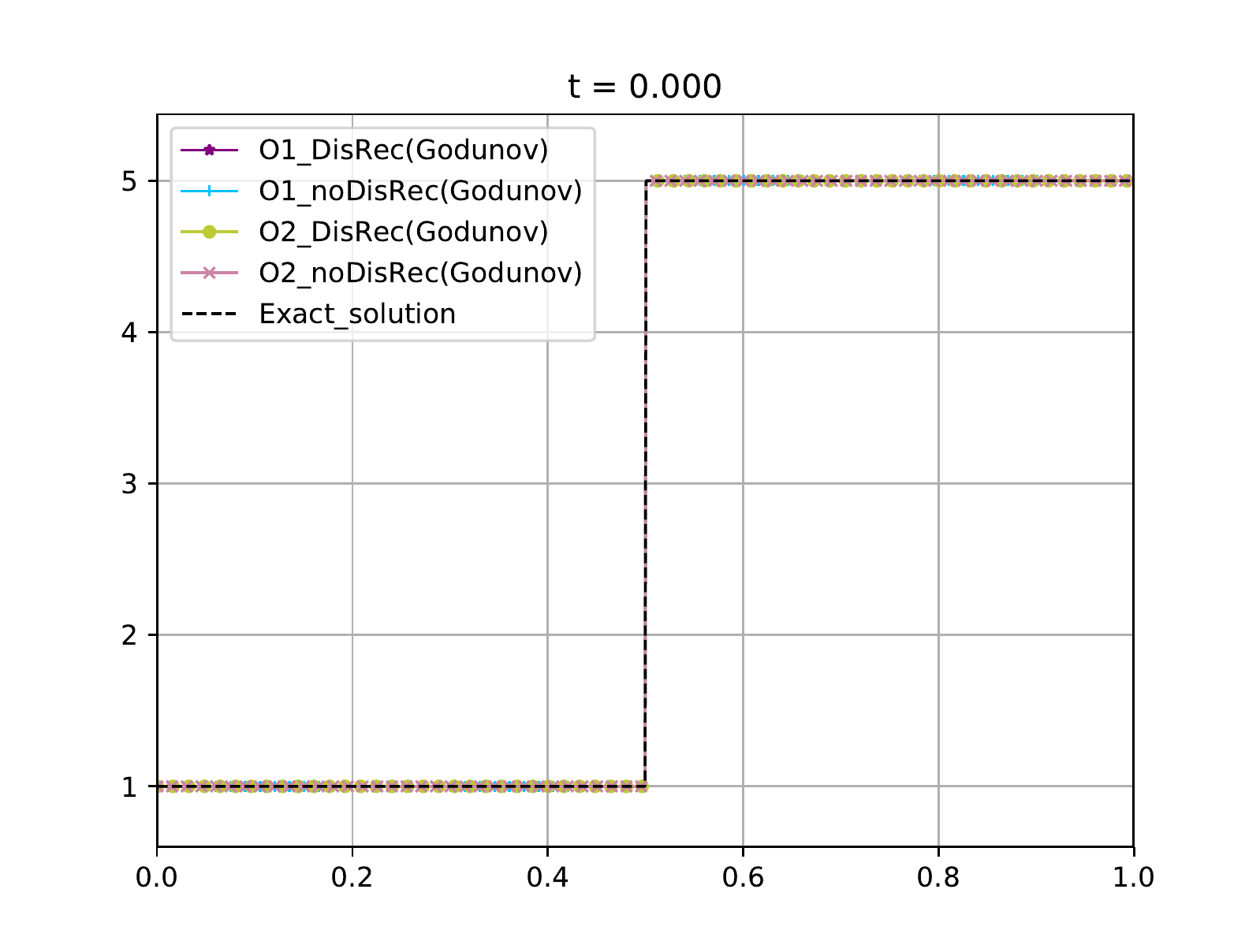}
		\end{subfigure}
		\begin{subfigure}{0.33\textwidth}
				\includegraphics[width=1.1\linewidth]{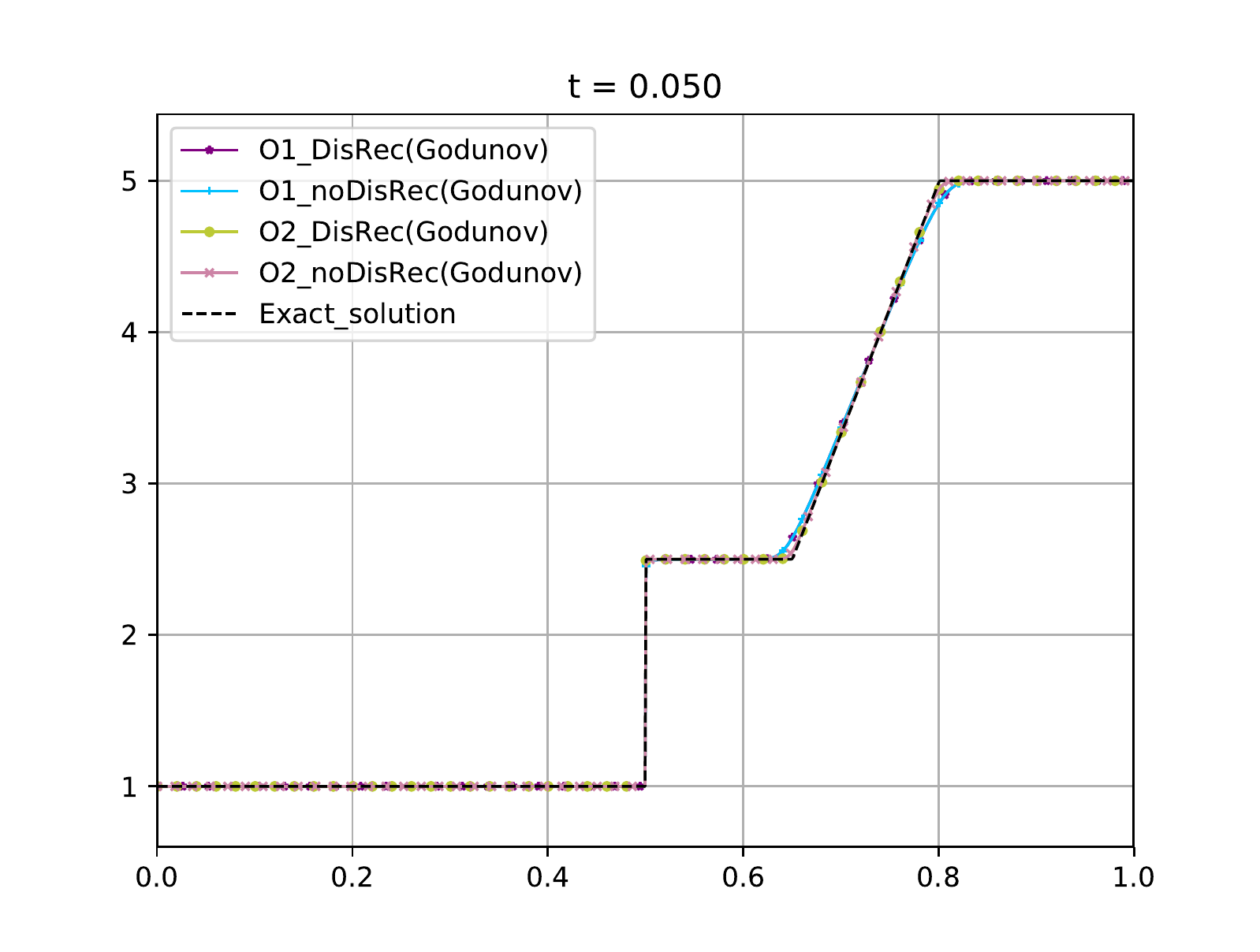}
		\end{subfigure}
		\begin{subfigure}{0.33\textwidth}
				\includegraphics[width=1.1\linewidth]{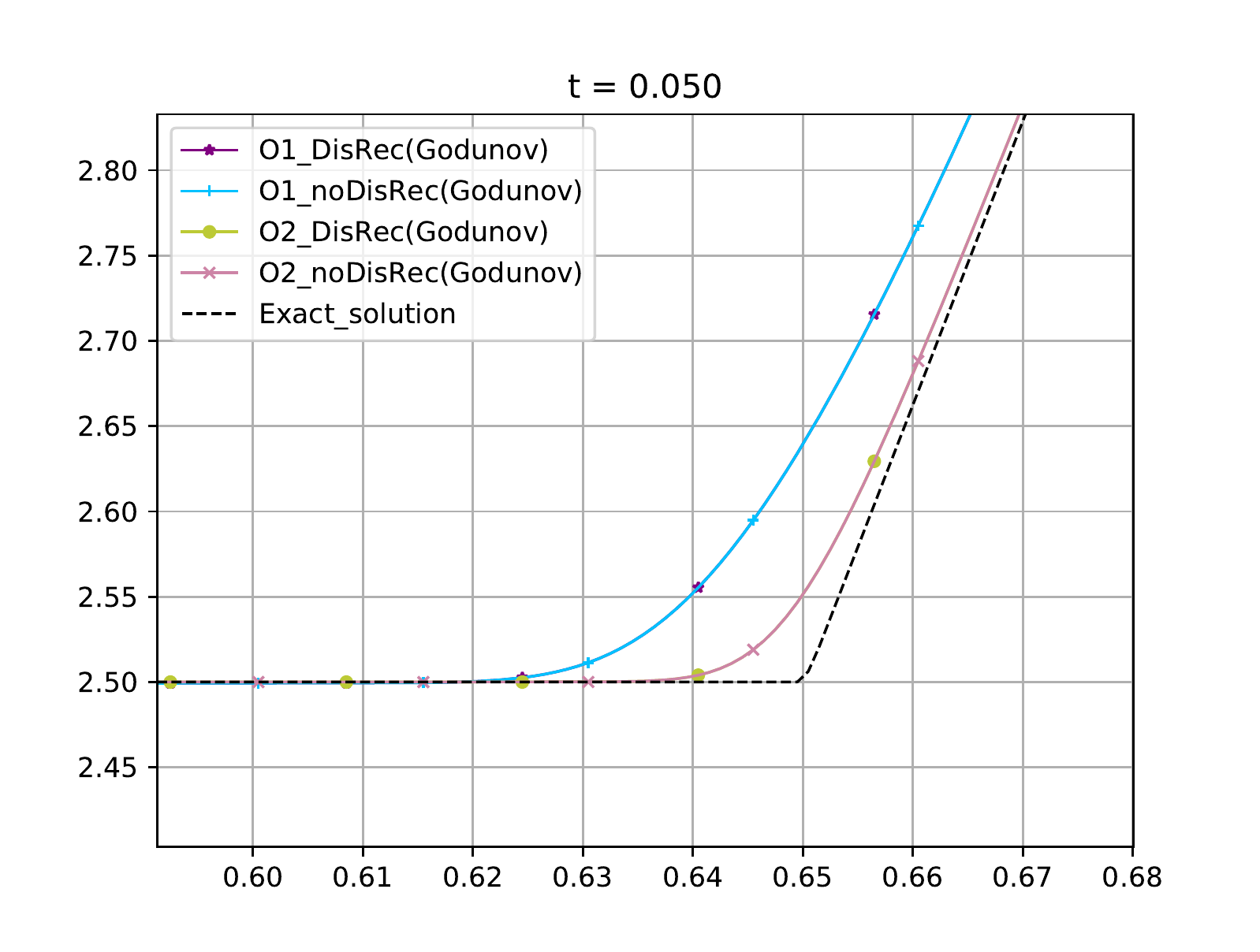}
				\caption*{Zoom}
		\end{subfigure}
		\caption{Coupled Burgers system. Test 4: variable $u$. Left: initial condition. Center: exact solution and numerical solutions obtained at time $t = 0.05$ with 1000 cells. Right: zoom.}
		\label{fig:Burgers_Test3_O1_vs_O2_DisRec_vs_noDisRec_Godunov_u}
	\end{figure}
	
	\begin{figure}[h]
		\begin{subfigure}{0.33\textwidth}
			\includegraphics[width=1.1\linewidth]{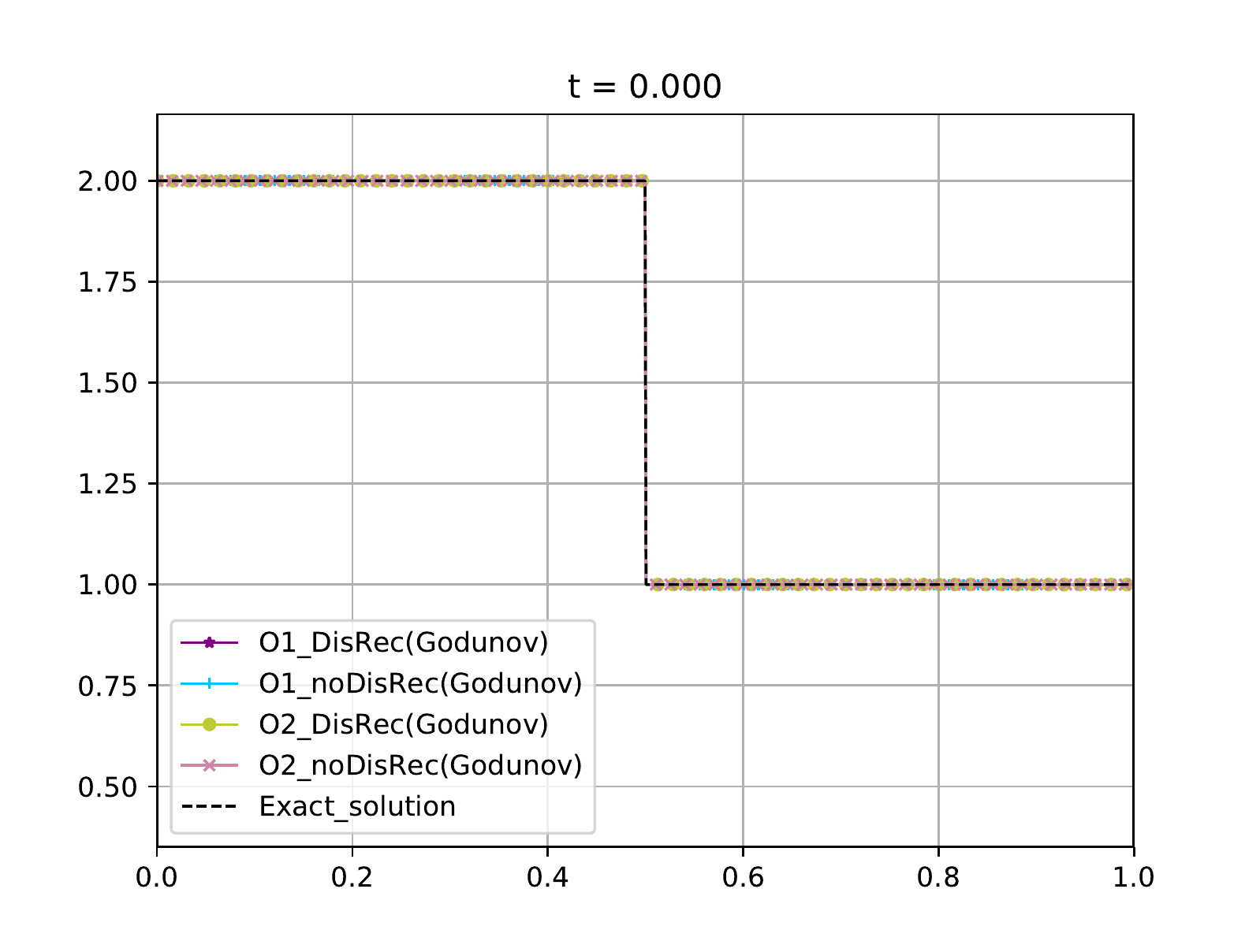}
		\end{subfigure}
		\begin{subfigure}{0.33\textwidth}
				\includegraphics[width=1.1\linewidth]{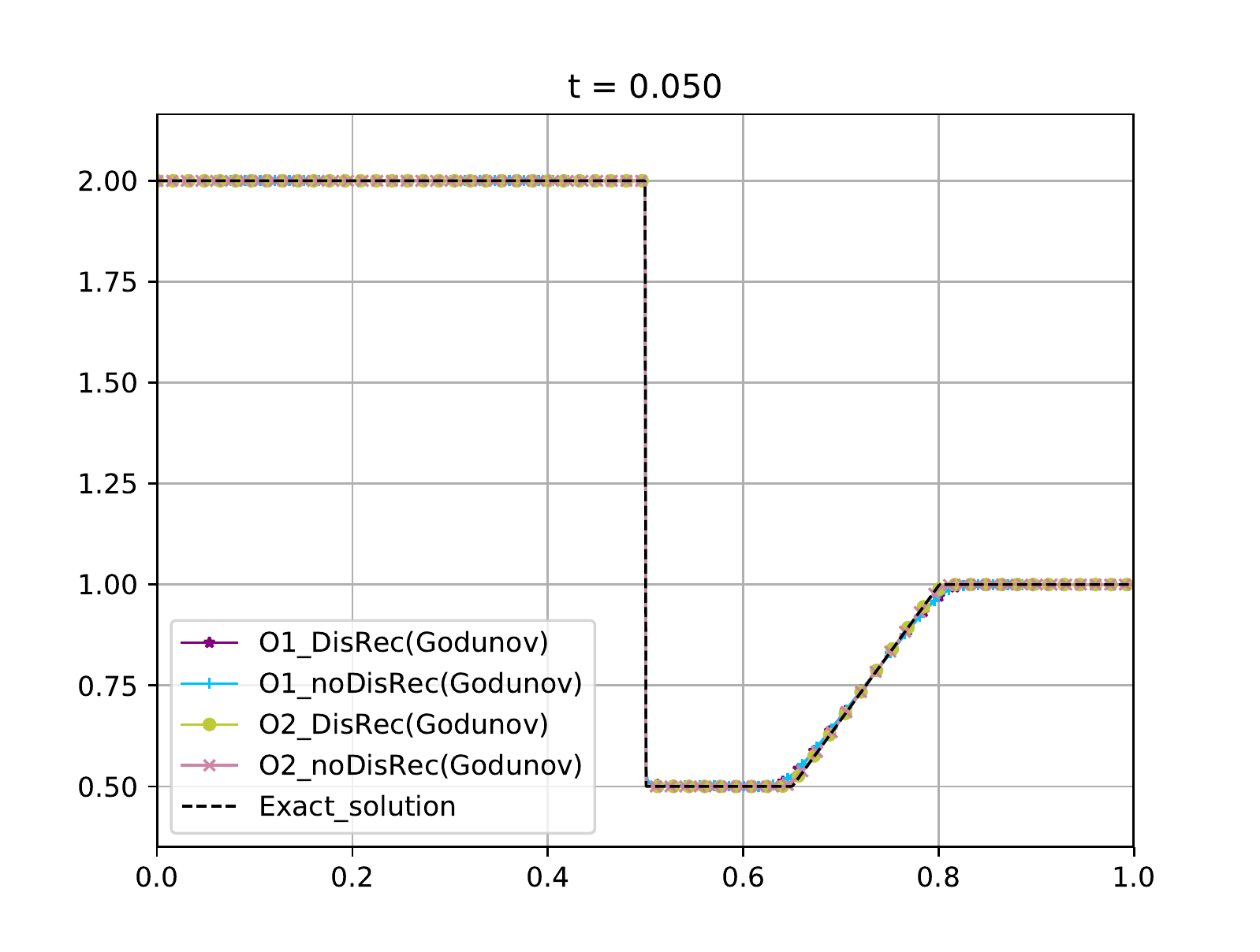}
		\end{subfigure}
		\begin{subfigure}{0.33\textwidth}
				\includegraphics[width=1.1\linewidth]{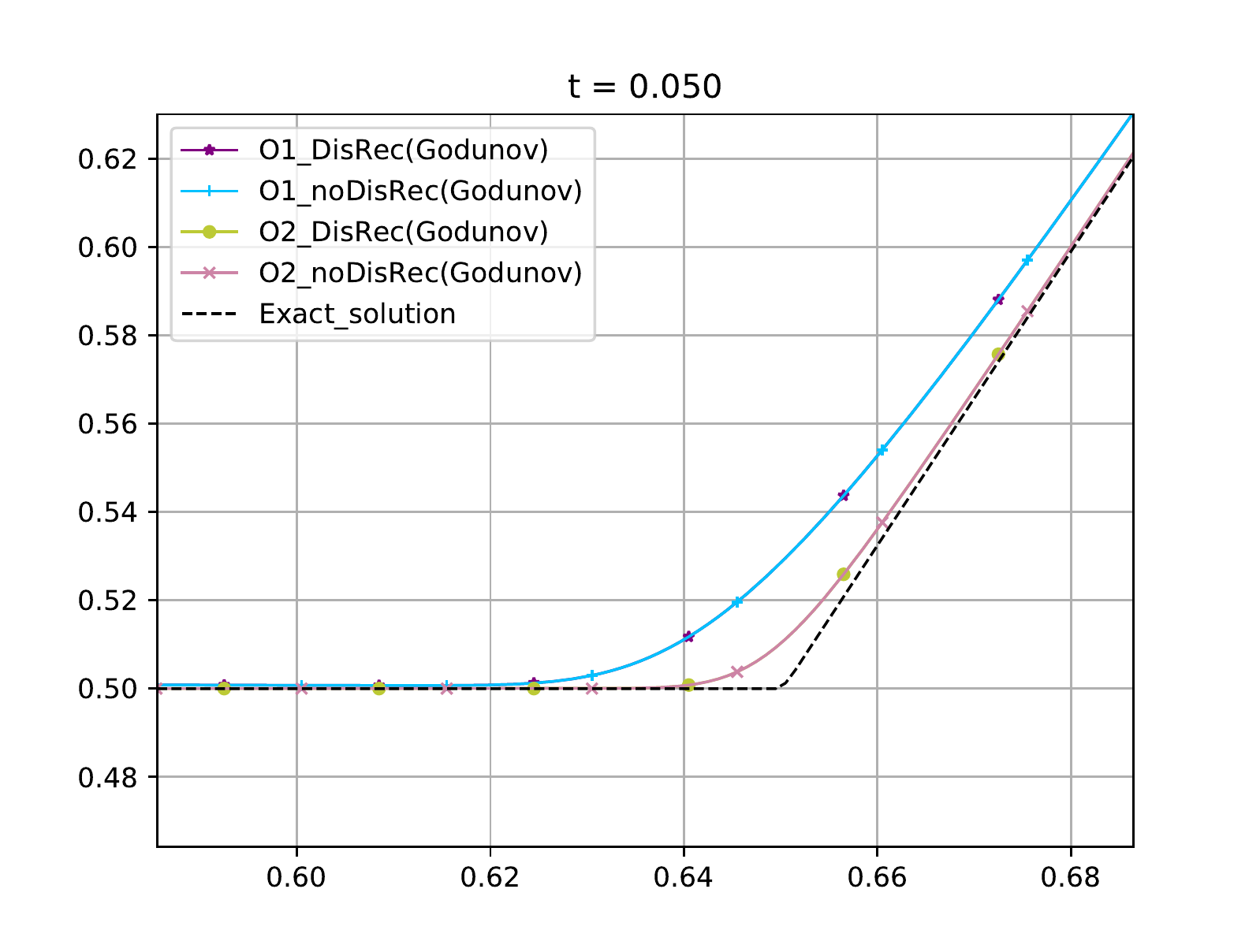}
				\caption*{Zoom}
		\end{subfigure}
		\caption{Coupled Burgers system. Test 4: variable $v$. Left: initial condition. Center: exact solution and numerical solutions obtained at time $t = 0.05$ with 1000 cells. Right: zoom.}
		\label{fig:Burgers_Test3_O1_vs_O2_DisRec_vs_noDisRec_Godunov_v}
	\end{figure}

%{\color{red} Habria que añadir algun test para que se vea que el de segundo orden mejora las soluciones: puedes considerar  un problema de  Riemann con una rarefacción. Otra posibilidad sería un  test tipo w.b.: tomas una solución estacionaria, p.ej u(x) = sen(x), v(x) = 1 - sen(x). Pones una  perturbacion y estudias la evolucion de la perturbacion. Se tiene que ver que el de orden 2 la  captura mejor -y de paso se ve que los metodos son wb-}%

\subsubsection*{Test 5: Stationary solution}

We consider the initial condition
\begin{equation}\label{TestWB}
    \bu_{0}(x) = (u,v)_{0}(x) =(\sin(x),1-\sin(x)),
\end{equation}
that is a stationary solution of the system (\ref{syst_CB}). We show in Figure \ref{fig:Burgers_TestWB_O1_O2_DisRec_Godunov_1000} the numerical solution obtained with the first and second order discontinuous in-cell reconstruction using a 1000-mesh. The results in Figure \ref{fig:Burgers_TestWB_O1_O2_DisRec_Godunov_error_comparison} and Table \ref{tab:Error_TestWB} show that the both schemes are well-balanced.

\begin{figure}[h]
		\begin{subfigure}{0.5\textwidth}
			\includegraphics[width=1.1\linewidth]{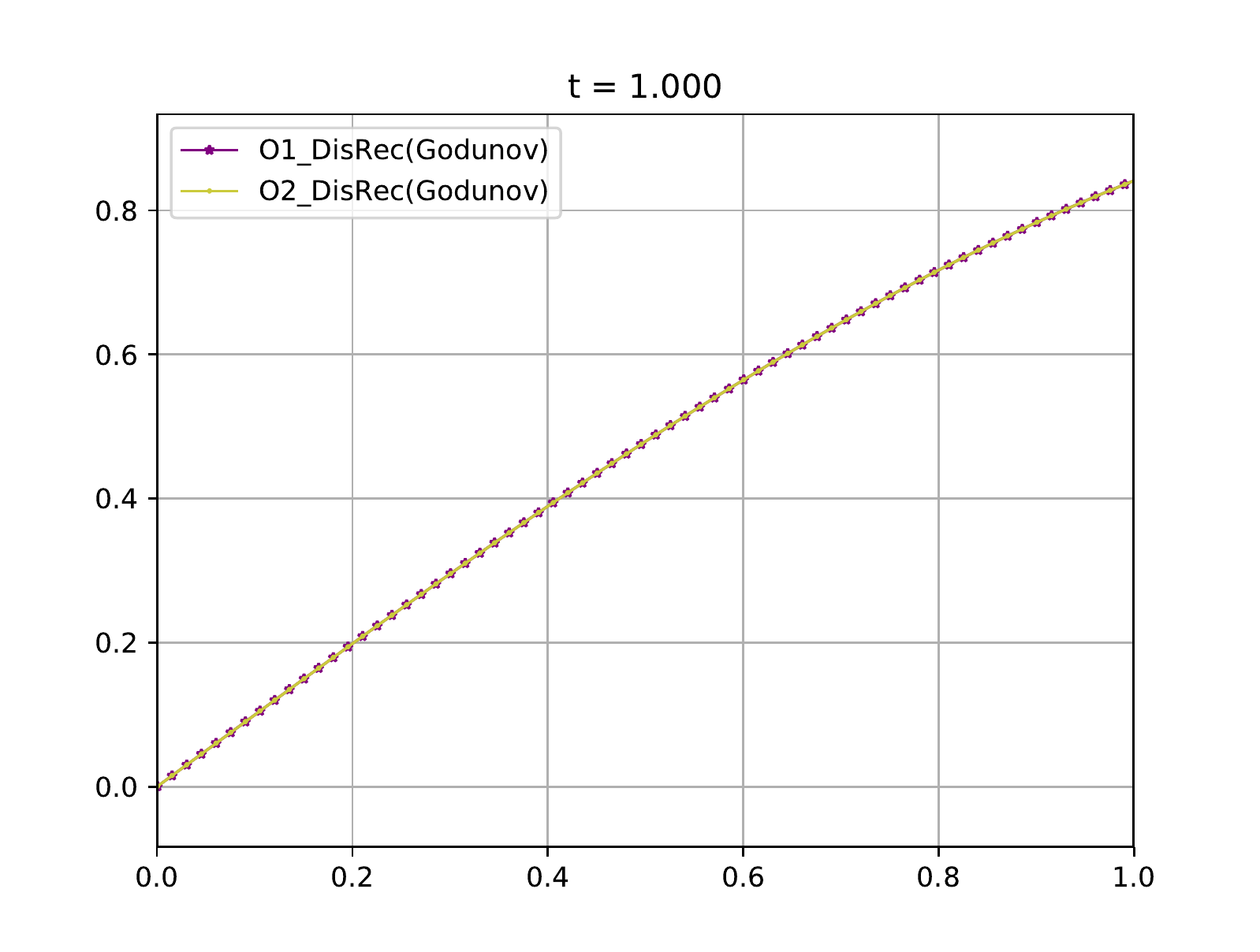}
		\end{subfigure}
		\begin{subfigure}{0.5\textwidth}
				\includegraphics[width=1.1\linewidth]{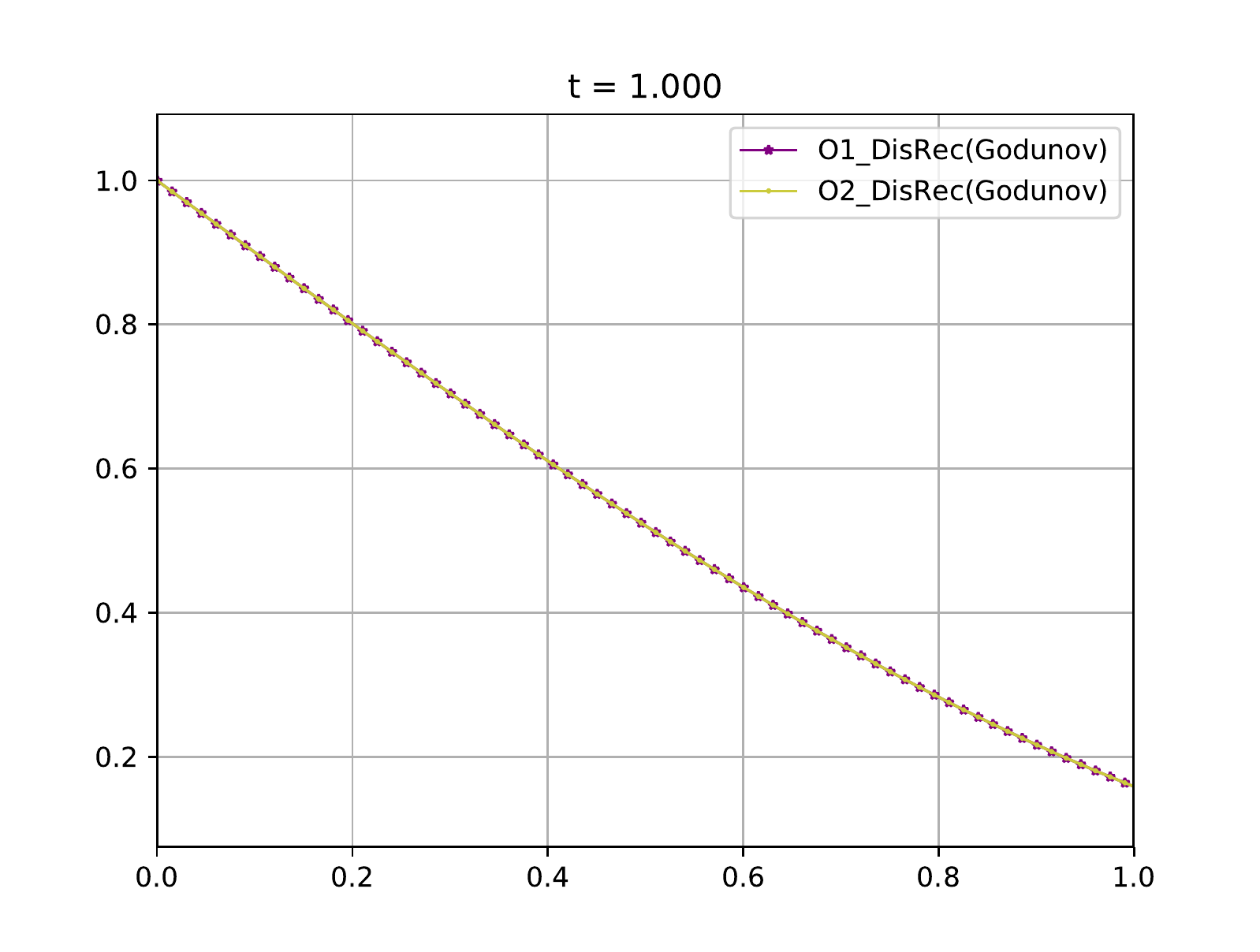}
		\end{subfigure}
		\caption{Coupled Burgers system. Test 5: numerical solution of (\ref{TestWB}) at time $t=1.00$ with 1000 cells. Left: variable $u$. Right: variable $v$.}
		\label{fig:Burgers_TestWB_O1_O2_DisRec_Godunov_1000}
	\end{figure}
	
\begin{figure}[h]
		\begin{subfigure}{0.5\textwidth}
			\includegraphics[width=1.1\linewidth]{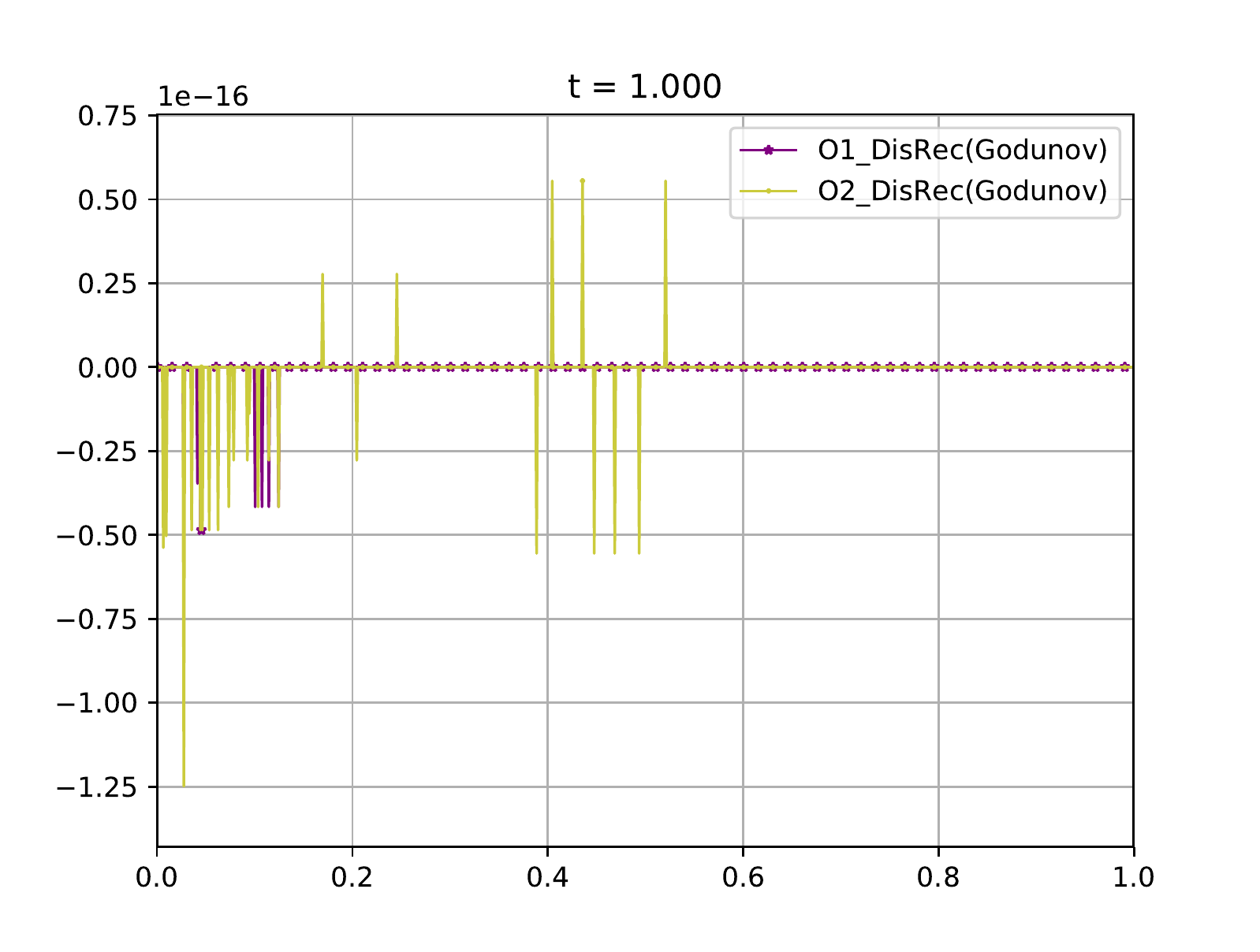}
		\end{subfigure}
		\begin{subfigure}{0.5\textwidth}
				\includegraphics[width=1.1\linewidth]{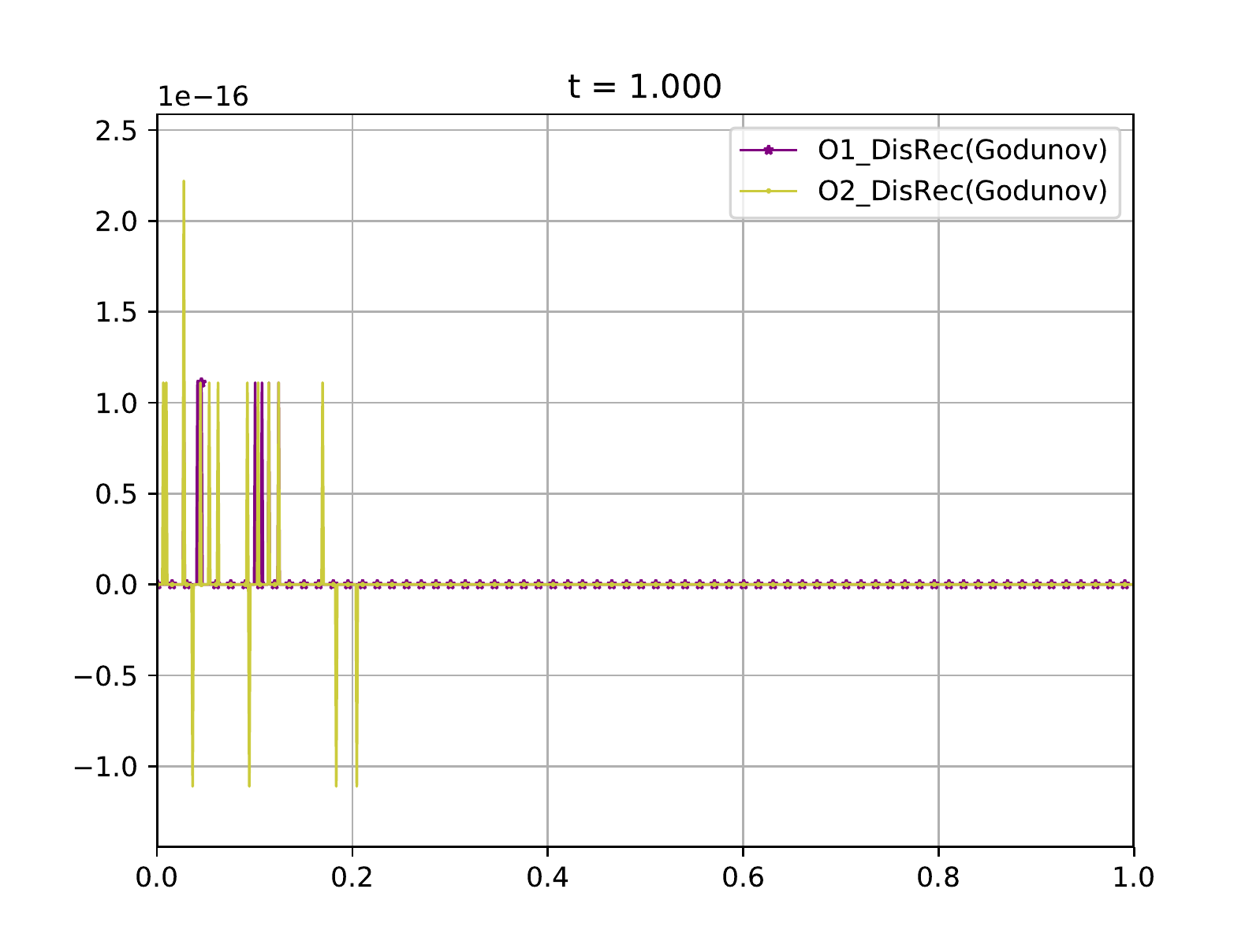}
		\end{subfigure}
		\caption{Coupled Burgers system. Test 5: difference between the numerical solution at $t=1.00$ and the stationary solution. Left: variable $u$. Right: variable $v$.}
		\label{fig:Burgers_TestWB_O1_O2_DisRec_Godunov_error_comparison}
	\end{figure}

\begin{table}[ht]
  	\centering
  	\begin{tabular}{|c|c|c|c|}
  		\hline 
  		$||\Delta u||_1$ (1st) & $||\Delta v||_1$ (1st) & $||\Delta u||_1$ (2nd) & $||\Delta v||_1$ (2nd) \\   
  		\hline 
  	    3.40e-19 & 8.88e-19 & 1.17e-18 & 1.78e-18 \\
  		\hline 
  	\end{tabular} 
	  	\caption{$L^{1}$ errors $||\Delta \cdot||_1$ at time $t=1$ for the Coupled Burgers model with initial conditions (\ref{TestWB}).}
  	\label{tab:Error_TestWB}
\end{table} 

\subsubsection*{Test 6: Perturbed stationary solution}

We consider finally the initial condition
\begin{equation}\label{TestPerturbedWB}
    \bu_{0}(x) = (u,v)_{0}(x) =(\sin(x)+0.2e^{-2000(r-0.5)^{2}},1-\sin(x)),
\end{equation}
that is the stationary solution \eqref{TestWB} with a perturbation in the variable $u$. Figures \ref{fig:Burgers_Test6_O1_vs_O2_DisRec_vs_noDisRec_Godunov_u} and \ref{fig:Burgers_Test6_O1_vs_O2_DisRec_vs_noDisRec_Godunov_v} show the numerical solutions obtained at time $t=0.2$ and $t=1$ using a 1000-cell mesh together with a reference solution obtained with the first order in-cell discontinuous reconstruction Godunov scheme using a 10000-cell mesh. As it can be seen the second order methods capture better the smooth parts of the solution and the ones with the in-cell reconstruction capture better the shock appearing in the perturbation. Observe that, in this case, the stationary solution \eqref{TestWB} is not restored: a different equilibrium with a stationary bump placed at the initial location of the perturbation is obtained once the waves generated by the perturbation leaves the computational domain.

\begin{figure}[h]
		\begin{subfigure}{0.5\textwidth}
			\includegraphics[width=1.1\linewidth]{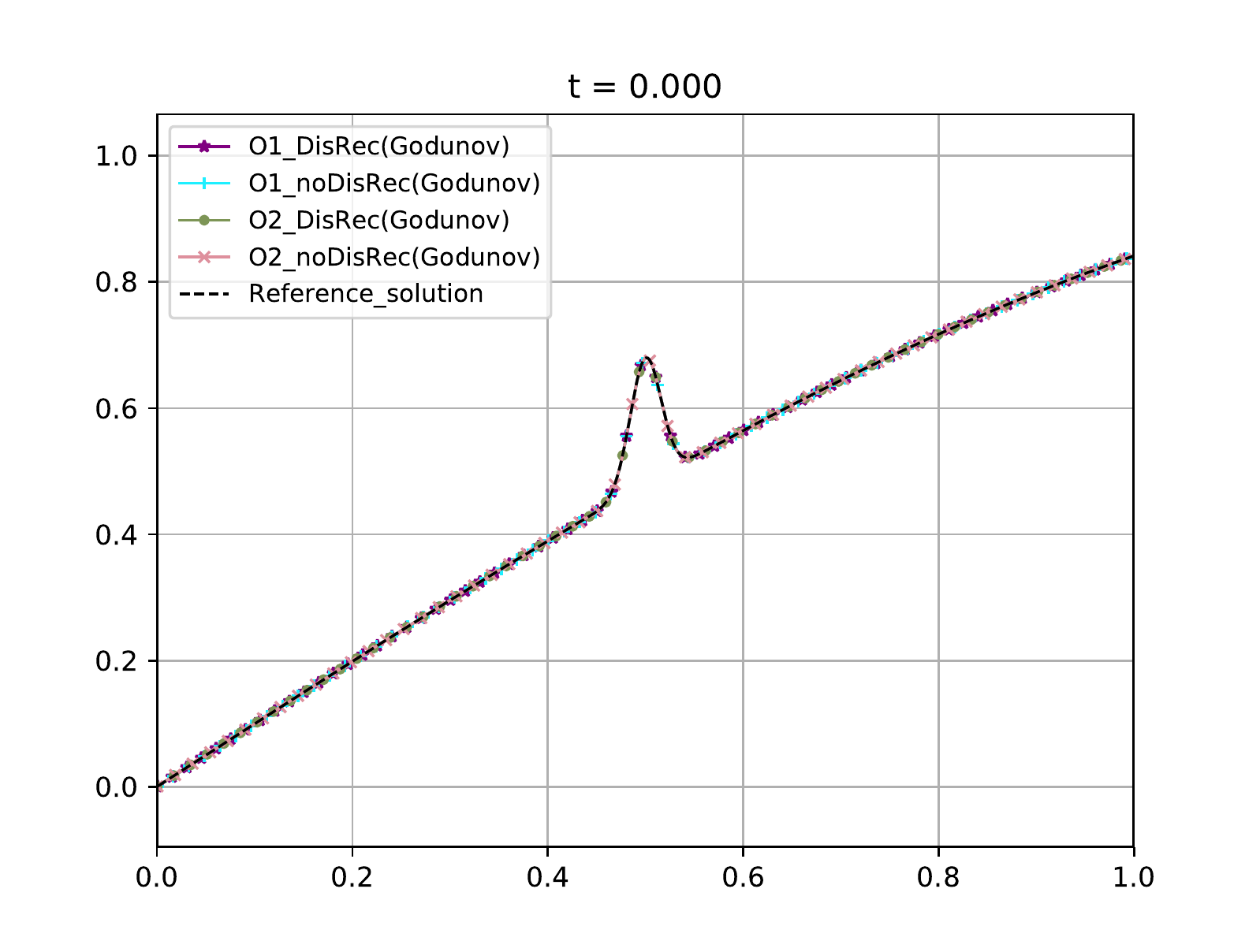}
		\end{subfigure}
		\begin{subfigure}{0.5\textwidth}
				\includegraphics[width=1.1\linewidth]{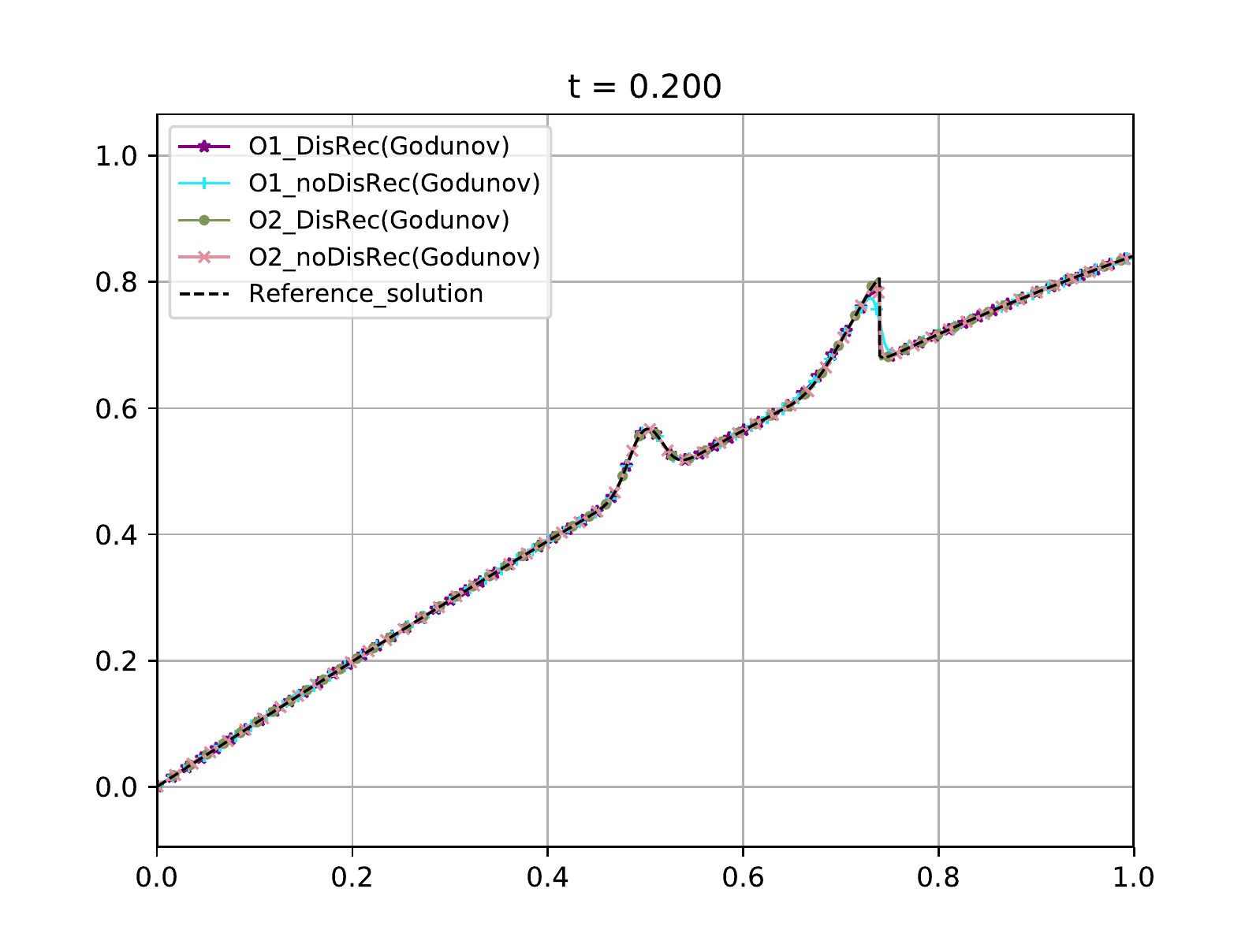}
		\end{subfigure}
		\newline
		\begin{subfigure}{0.5\textwidth}
			\includegraphics[width=1.1\linewidth]{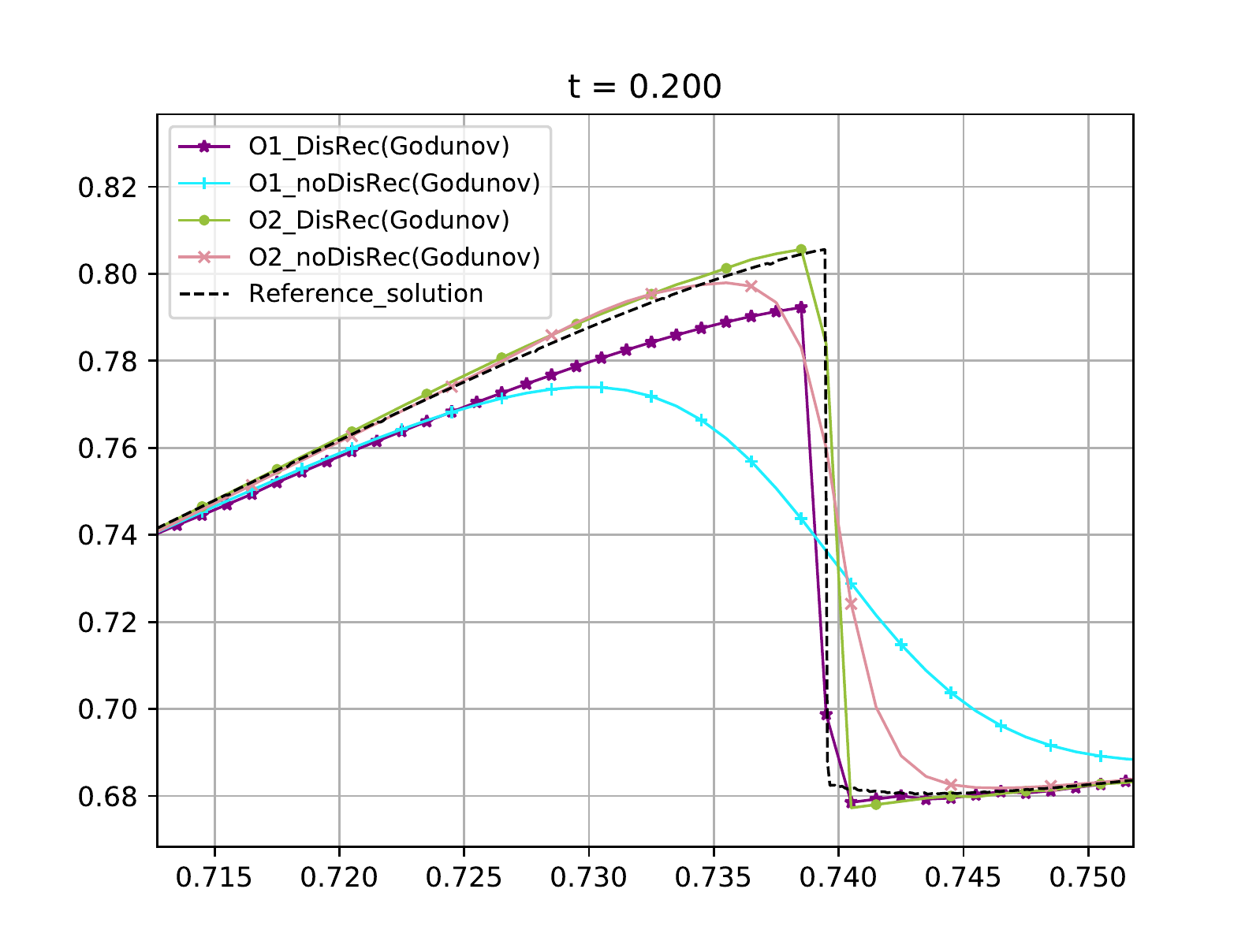}
			\caption{zoom}
		\end{subfigure}
		\begin{subfigure}{0.5\textwidth}
				\includegraphics[width=1.1\linewidth]{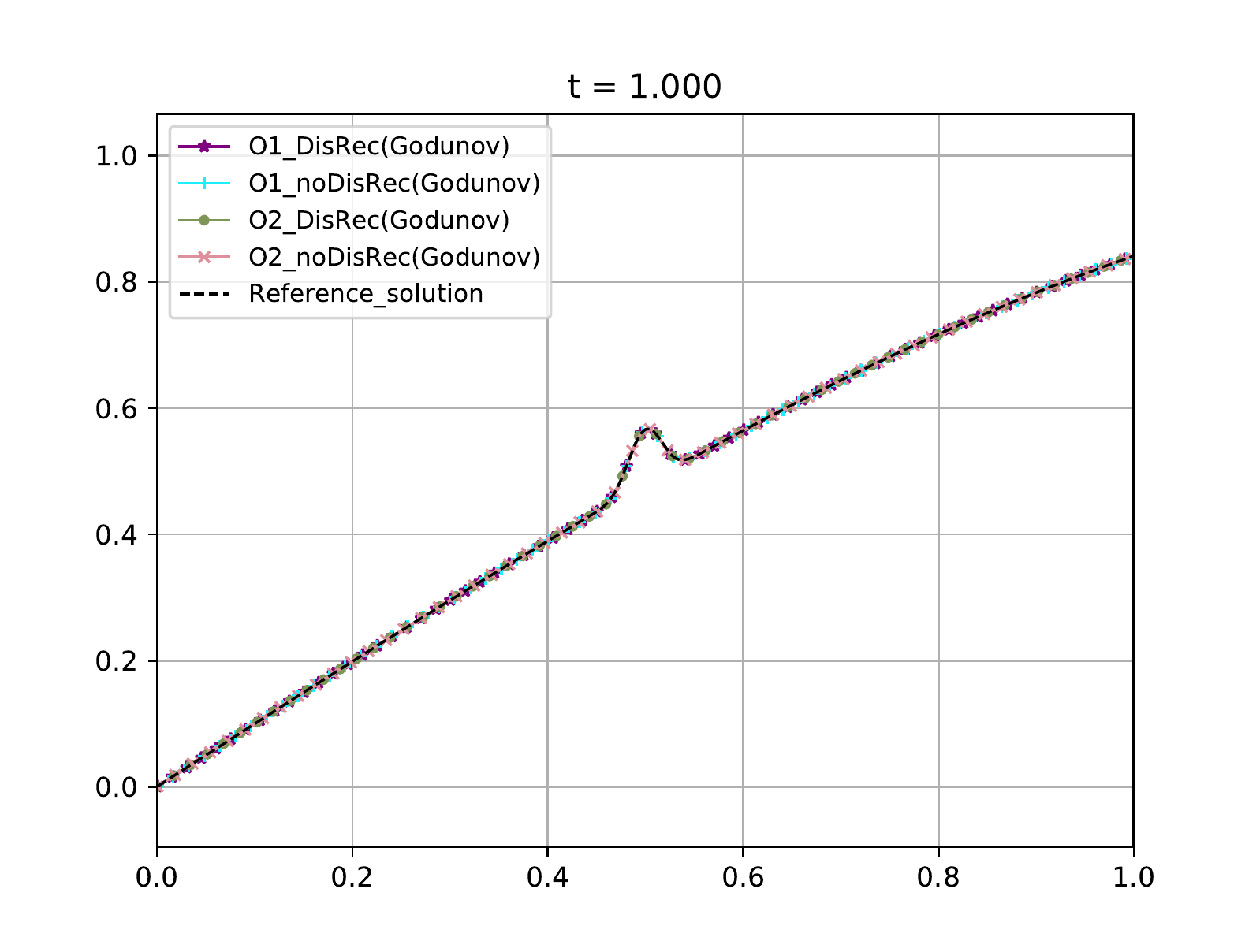}
		\end{subfigure}
		\caption{Coupled Burgers system. Test 6: variable $u$. Top: initial condition (left),  reference and numerical solutions obtained at time $t = 0.2$ with 1000 cells (right). Down: zoom of the perturbation area at time $t = 0.2$ (left), reference and numerical solutions obtained at time $t = 1$ (right).}
		\label{fig:Burgers_Test6_O1_vs_O2_DisRec_vs_noDisRec_Godunov_u}
	\end{figure}
	
\begin{figure}[h]
		\begin{subfigure}{0.5\textwidth}
			\includegraphics[width=1.1\linewidth]{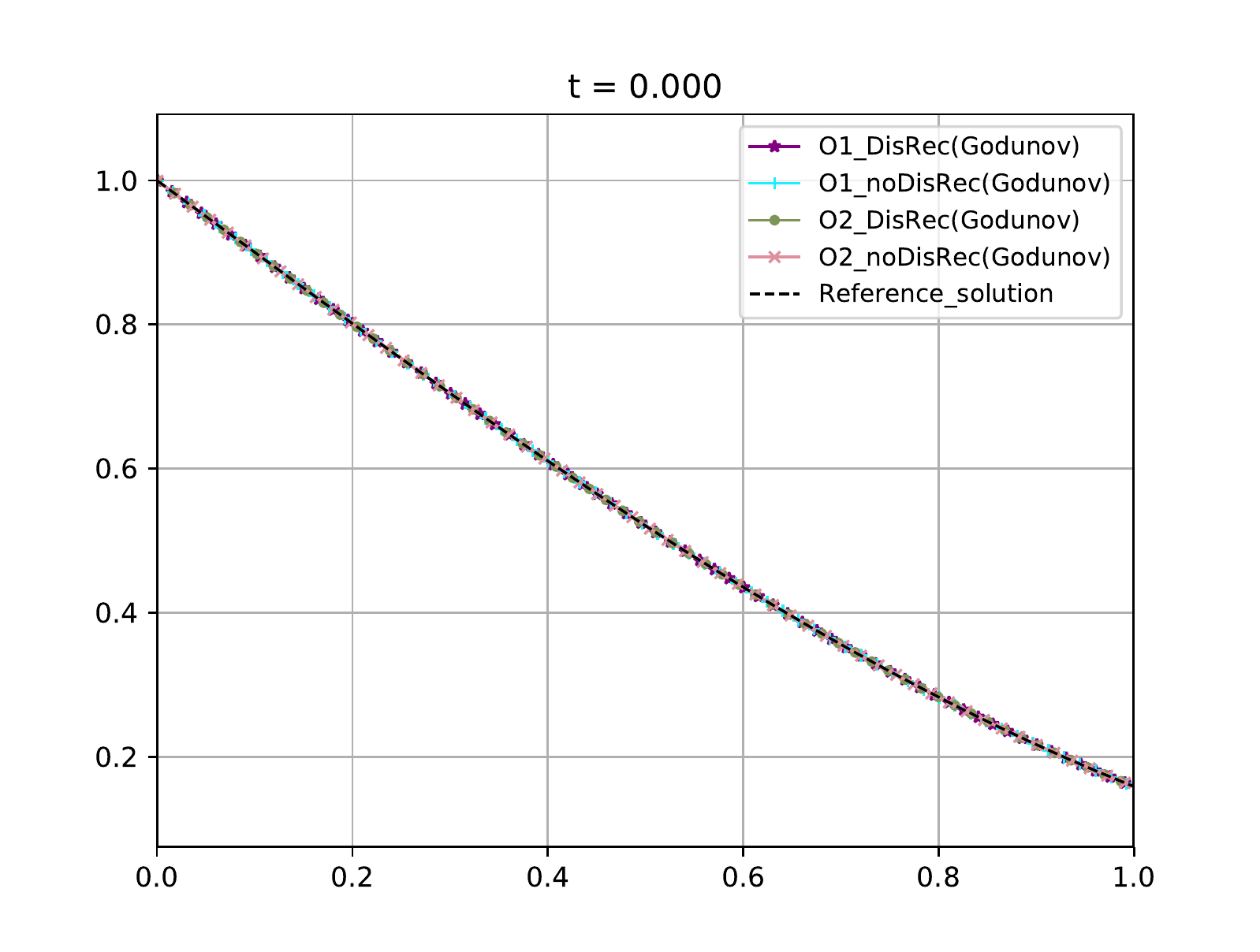}
		\end{subfigure}
		\begin{subfigure}{0.5\textwidth}
				\includegraphics[width=1.1\linewidth]{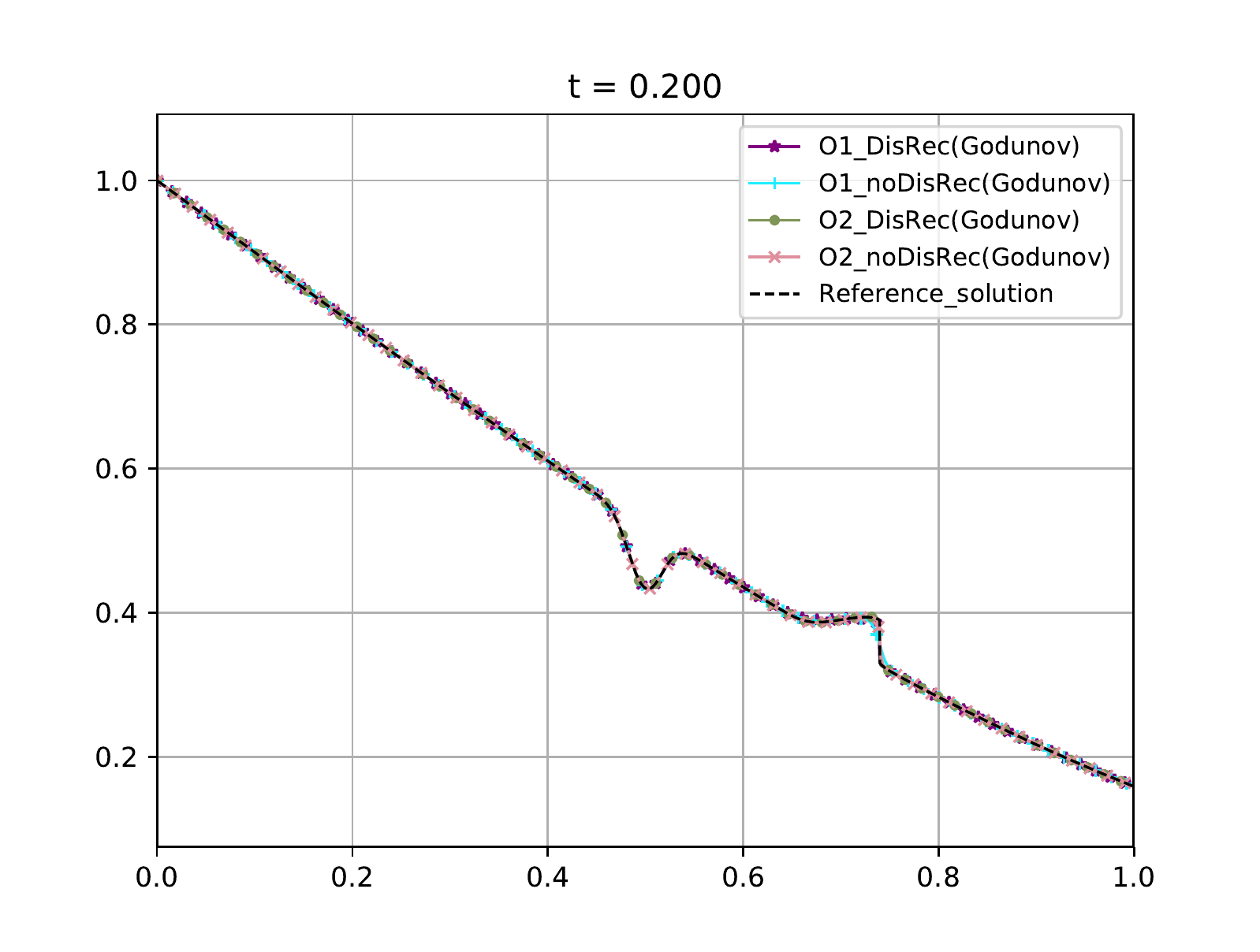}
		\end{subfigure}
		\newline
		\begin{subfigure}{0.5\textwidth}
			\includegraphics[width=1.1\linewidth]{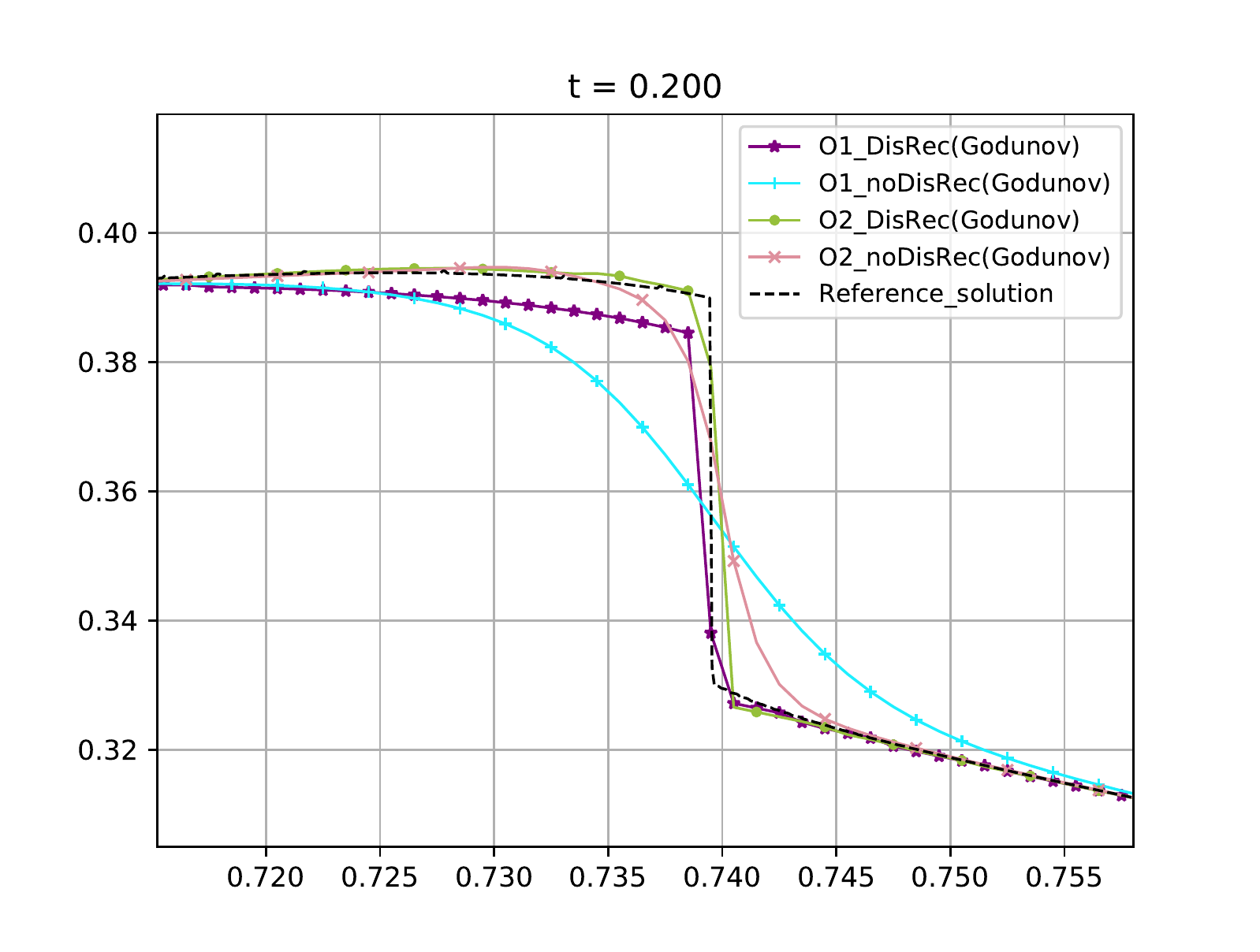}
			\caption{zoom}
		\end{subfigure}
		\begin{subfigure}{0.5\textwidth}
				\includegraphics[width=1.1\linewidth]{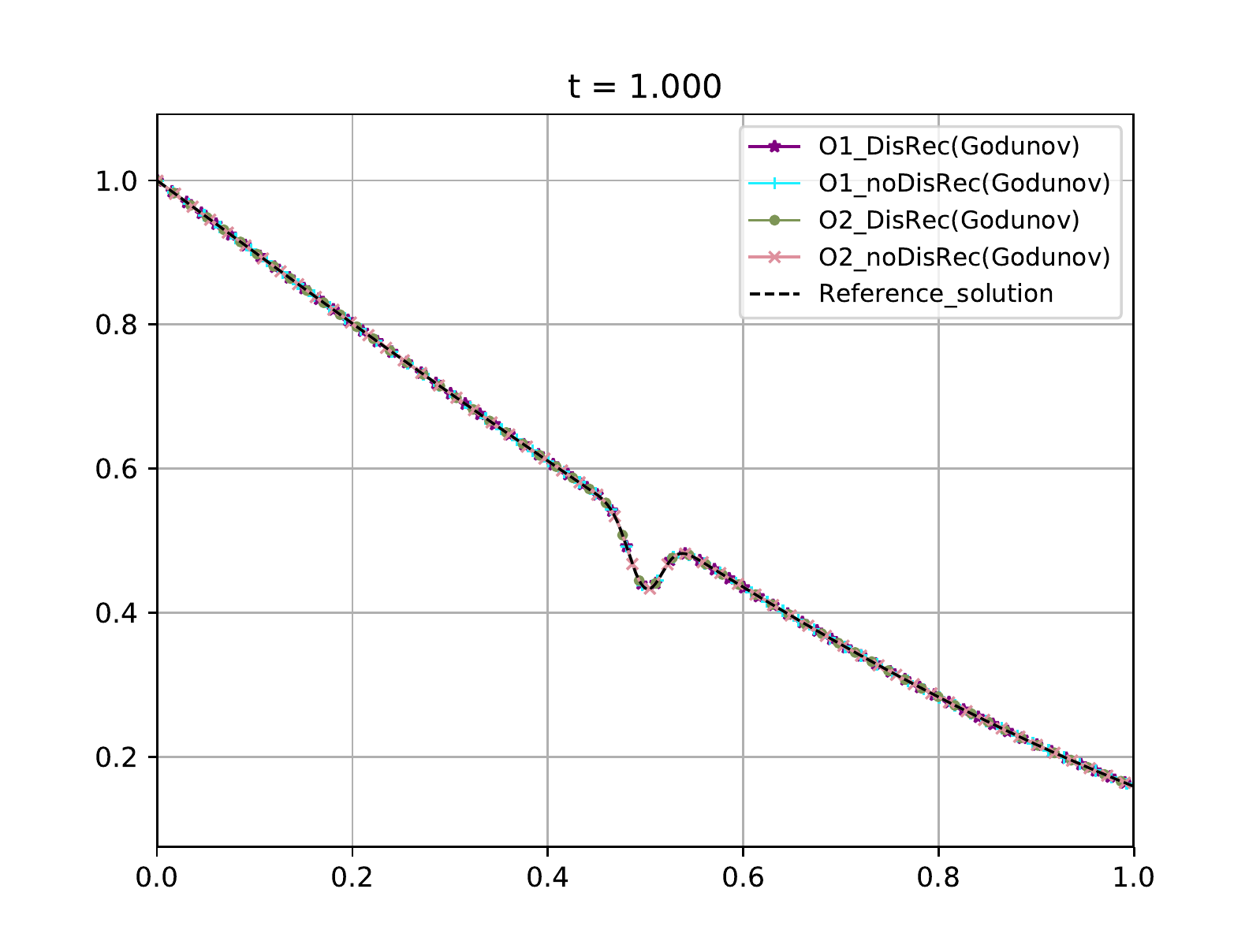}
		\end{subfigure}
		\caption{Coupled Burgers system. Test 6: variable $v$. Top: initial condition (left),  reference and numerical solutions obtained at time $t = 0.2$ with 1000 cells (right). Down: zoom of the perturbation area at time $t = 0.2$ (left), reference and numerical solutions obtained at time $t = 1$ (right).}
		\label{fig:Burgers_Test6_O1_vs_O2_DisRec_vs_noDisRec_Godunov_v}
	\end{figure}

\subsection{Gas dynamics equations in Lagrangian coordinates}
Let us consider the gas dynamics equations in Lagrangian coordinates: 
\begin{equation} \label{systeme}
\left\{
\begin{array}{l}
\partial_t \tau - \partial_x u = 0,\\
\partial_t u + \partial_x p = 0,\\
\partial_t E + \partial_x (pu) = 0,
\end{array}
\right.
\end{equation}
where $\tau >0$ represents the inverse of the density, $u$ is the velocity, $p=p(\tau,e) > 0$ is the pressure, $e$ is the internal energy,  and  $E=e+{u^2}/{2}$ the total energy. For the sake of simplicity, we consider a perfect gas equation of state $p (\tau,e) = (\gamma - 1) {e}/{\tau}$ 
where $\gamma > 1$. System (\ref{systeme}) can be rewritten in nonconservative form as follows
\begin{equation} \label{systemenon}
\left\{
\begin{array}{l}
\partial_t \tau - \partial_x u = 0,\\
\partial_t u + \partial_x p = 0,\\
\partial_t e + p\partial_x u = 0,
\end{array}
\right.
\end{equation}
that can be written in the form (\ref{sys:nonconservative}) with
$$\bu = \left(\begin{array}{c}
     \tau  \\
     u     \\
     e
\end{array}\right), \quad \mathcal{A}(\bu)=\left(\begin{array}{ccc}
    0 & -1 & 0\\
    \displaystyle -\frac{(\gamma -1)e}{\tau^{2}} & 0 & \displaystyle \frac{\gamma-1}{\tau} \\
    0 & \displaystyle \frac{(\gamma-1)e}{\tau} & 0
\end{array}\right).$$

The system is strictly hyperbolic with eigenvalues
$$
\lambda_1 (\bu) = -\sqrt{{\gamma p}/{\tau}}, \quad \lambda_2 (\bu) = 0, \quad \lambda_3 (\bu) = \sqrt{{\gamma p}/{\tau}},
$$
whose characteristic fields are given by the eigenvectors
$$
R_1(\bu) = [1 , \sqrt{{\gamma p}/{\tau}}, -p]^T, \quad R_2(\bu) = [1 , 0, p/(\gamma-1)], \quad R_3(\bu) = [1, -\sqrt{{\gamma p}/{\tau}}, -p]^T.
$$
$R_2(\bu)$ is linearly degenerate and $R_i(\bu)$, $i = 1,3$ are genuinely nonlinear: see \cite{godlewski1995numerical}. On the other hand, the admissible solutions of (\ref{systeme}) are selected by Lax entropy inequalities, which here are equivalent to:
\begin{equation}\label{gasentropy}
\sigma (\tau_+ - \tau_-) \geq 0,
\end{equation}
where $\tau_-$ and $\tau_+$ are the values of $\tau$ at both sides of the discontinuity and $\sigma$ its speed of propagation.

Once the family of paths has been chosen, the simple waves of this system are:
\begin{itemize}
    \item  Stationary contact discontinuities linking states $\bu_l$, $\bu_r$ such that
    $$ u_l = u_r.$$
    
    \item Rarefactions waves joining states $\bu_l$, $\bu_r$ that satisfy 
    $$ u_l< u_r,$$
    and the relations given by the Riemann invariants:
    \begin{itemize}
        \item 1-rarefactions: 
        $$ 2\sqrt{\frac{\gamma e_l}{\gamma - 1}}+u_l = 2\sqrt{\frac{\gamma e_r}{\gamma - 1}}+u_r, \quad  \frac{e_l}{\tau_l^{\gamma-1}} = \frac{e_r}{\tau_r^{\gamma-1}}.$$
        \item 2-rarefactions: 
        $$ 2\sqrt{\frac{\gamma e_l}{\gamma - 1}}-u_l = 2\sqrt{\frac{\gamma e_r}{\gamma - 1}}-u_r, \quad \frac{e_l}{\tau_l^{\gamma-1}} = \frac{e_r}{\tau_r^{\gamma-1}}.
        $$
    \end{itemize}
    
    \item Shock waves joining states $\bu_l$ and $\bu_r$  that satisfy
    $$ u_l  > u_r $$
and the jump conditions:
    \begin{eqnarray*}
    \sigma[\tau] & = & -\left[u \right],\\
    \sigma[u] & = & \left[p \right], \\
    \sigma[e] & = & \int_0^1 \phi_p(s; \bu_l, \bu_r) 
    \partial_s \phi_u(s; \bu_l, \bu_r) \,ds.
    \end{eqnarray*}
    
\end{itemize}

If, for instance, the family of straight segments is chosen for the variables $\tau,u,p$
$$
\phi_\tau(s; \bu_l, \bu_r) = \tau_l + s(\tau_r - \tau_l); \quad  \phi_u(s; \bu_l, \bu_r) = u_l + s(u_r - u_l); \quad \phi_p(s; \bu_l, \bu_r) = p_l + s(p_r - p_l),
$$
the jump conditions reduce to:
    \begin{eqnarray*}
    \sigma[\tau] & = & (u_l-u_r),\\
    \sigma[u] & = & p_r-p_l, \\
    \sigma[e] & = & \frac{1}{2}(p_r+p_l)(u_r-u_l).
    \end{eqnarray*}
It can be easily checked that these jump conditions are equivalent to the standard Rankine-Hugoniot conditions corresponding to the conservative formulation (\ref{systeme}) and thus, the weak solutions are the same.

A Roe matrix is given in this case by:
$$
\mathcal{A}(\bu_l, \bu_r) =\mathcal{A}(\bar{\bu}), \quad \bar{\bu}(\bu_l, \bu_r) = (\bar{\tau},\bar{u}, \bar{p}),
$$
with 
$$\bar{\tau} = \frac{\tau_l + \tau_r}{2}, \quad \bar{u} = \frac{u_l + u_r}{2}, \quad \bar{e} = \frac{\bar{p} \bar{\tau}}{\gamma - 1}, \quad \bar{p} = \frac{p_l + p_r}{2},$$
see \cite{munz1994godunov}.\\

In
\cite{chalons2019path} the in-cell discontinuous reconstruction technique has been used to correct the results that are obtained with the standard Roe path-conservative scheme. To apply this technique, a cell is marked if
$$
u^n_{j-1} \geq u^n_{j+1}.
$$
The second strategy to select the speed, and the left and right states of the discontinuous reconstruction based on the Roe matrix is used here (see
Subsection \ref{ss:strategy}). More precisely:
\begin{itemize}
    \item If $u^n_{j-1}= u^n_{j+1}$ then 
    $$
    \sigma_j^n = 0, \quad \bu^n_{j,l} = \bu^n_{j-1}, \quad \bu^n_{j,r} = \bu^n_{j+1}.
    $$
    
    \item If $u^n_{j-1} > u^n_{j+1}$ and $\tau_{j+1}^{n}-\tau_{j-1}^{n}<0$ then
$$
\sigma_j^n =-\sqrt{\gamma\bar{p}/\bar{\tau}}, \quad  \bu^n_{j,l} = \bu^n_{j-1},  \quad \bu^n_{j,r} = \bu^n_{j-1} + \alpha_{1}R_1(\bu^n_{j-1},  \bu^n_{j+1}).
$$

\item If $u^n_{j-1} > u^n_{j+1}$ and $\tau_{j+1}^{n}-\tau_{j-1}^{n}>0$ then
$$
\sigma_j^n =\sqrt{\gamma\bar{p}/\bar{\tau}}, \quad \bu^n_{j,l} =  \bu^n_{j+1} - \alpha_{3}R_3(\bu^n_{j-1},  \bu^n_{j+1}), \quad \bu^n_{j,r} = \bu_{j+1}.
$$    
\end{itemize}
Here $\bar{p}$ and $\bar{\tau}$ represent the Roe intermediate values of $p$ and $\tau$, and  $\alpha_{k}$, $k=1, 2, 3$ the coordinates of $\bu^n_{j+1}-\bu^n_{j-1}$ in the basis of eigenvectors of the Roe matrix, i.e. $ \bu^n_{j+1}-\bu^n_{j-1} = \sum_{k=1}^{3} \alpha_{k}R_{k}(\bu^n_{j-1}, \bu^n_{j+1})$.  
This method is extended here to second order by following the procedure described in Section \ref{sec:methods}.

\subsubsection*{Test 1: Isolated 1-shock}

Let us consider the following initial condition taken from \cite{chalons2019path}

$$(\tau, u, p)_{0}(x)= \begin{cases}
     (2.09836065573770281, 2.3046638387921279, 1) & \text{if $x<0.5$,}  \\
     (8,0,0.1) & \text{otherwise.}
\end{cases} $$
The solution of the Riemann problem consists of a 1-shock wave joining the left and right states.
Figures \ref{fig:Gas_Test1_O1_vs_O2_DisRec_vs_noDisRec_Roe_tau},  \ref{fig:Gas_Test1_O1_vs_O2_DisRec_vs_noDisRec_Roe_u}, and \ref{fig:Gas_Test1_O1_vs_O2_DisRec_vs_noDisRec_Roe_e} compare the exact solution with the numerical approximations at time $t = 0.5$ obtained with Roe method, its second order extension based on the standard MUSCL reconstruction, and the first and second order discontinuous in-cell reconstruction schemes using 300-cell mesh and CFL = 0.5: as it can be seen Roe methods does not capture the discontinuities properly what is not the case for the two other methods. 

\begin{figure}[h]
		\begin{subfigure}{0.5\textwidth}
			\includegraphics[width=1.1\linewidth]{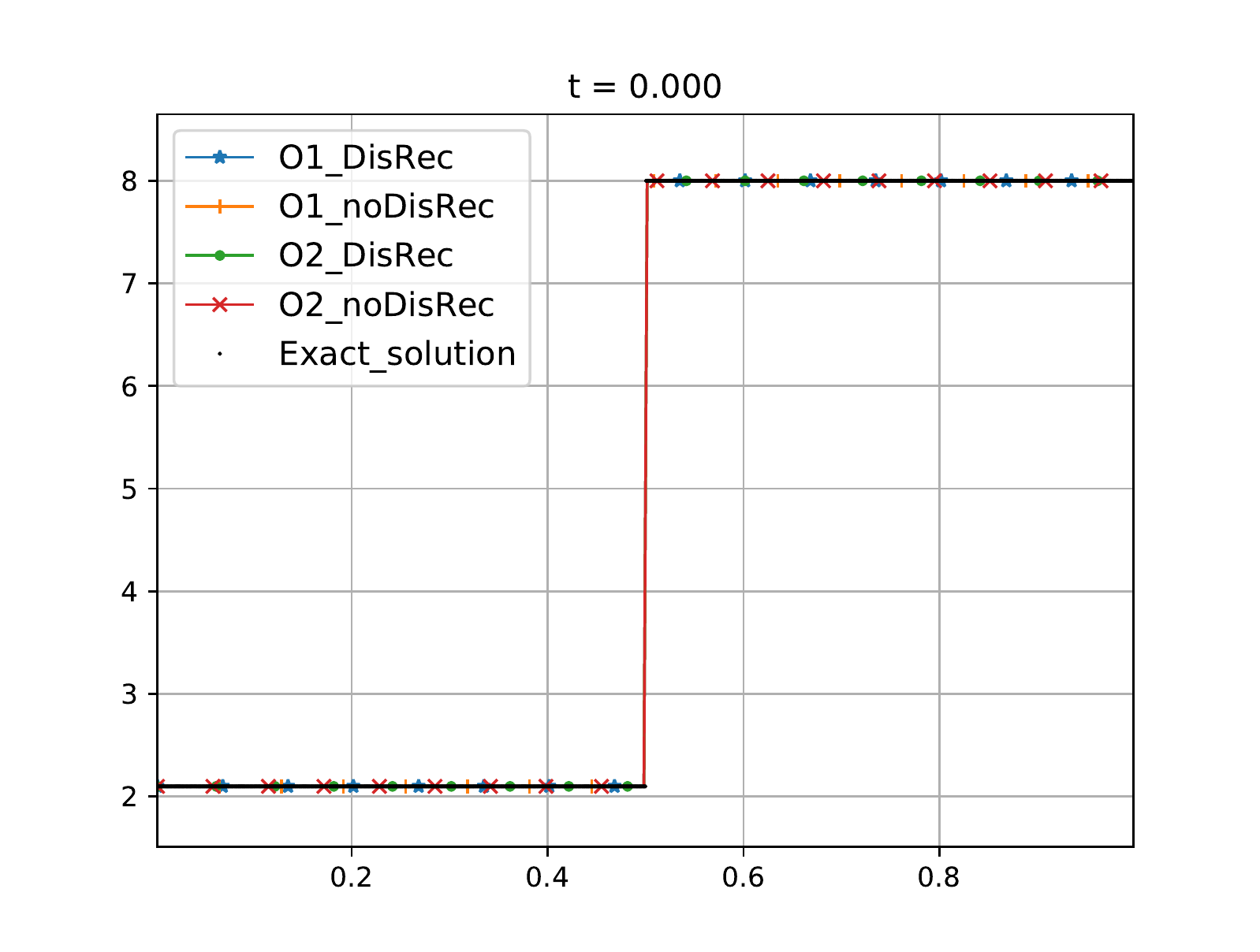}
		\end{subfigure}
		\begin{subfigure}{0.5\textwidth}
				\includegraphics[width=1.1\linewidth]{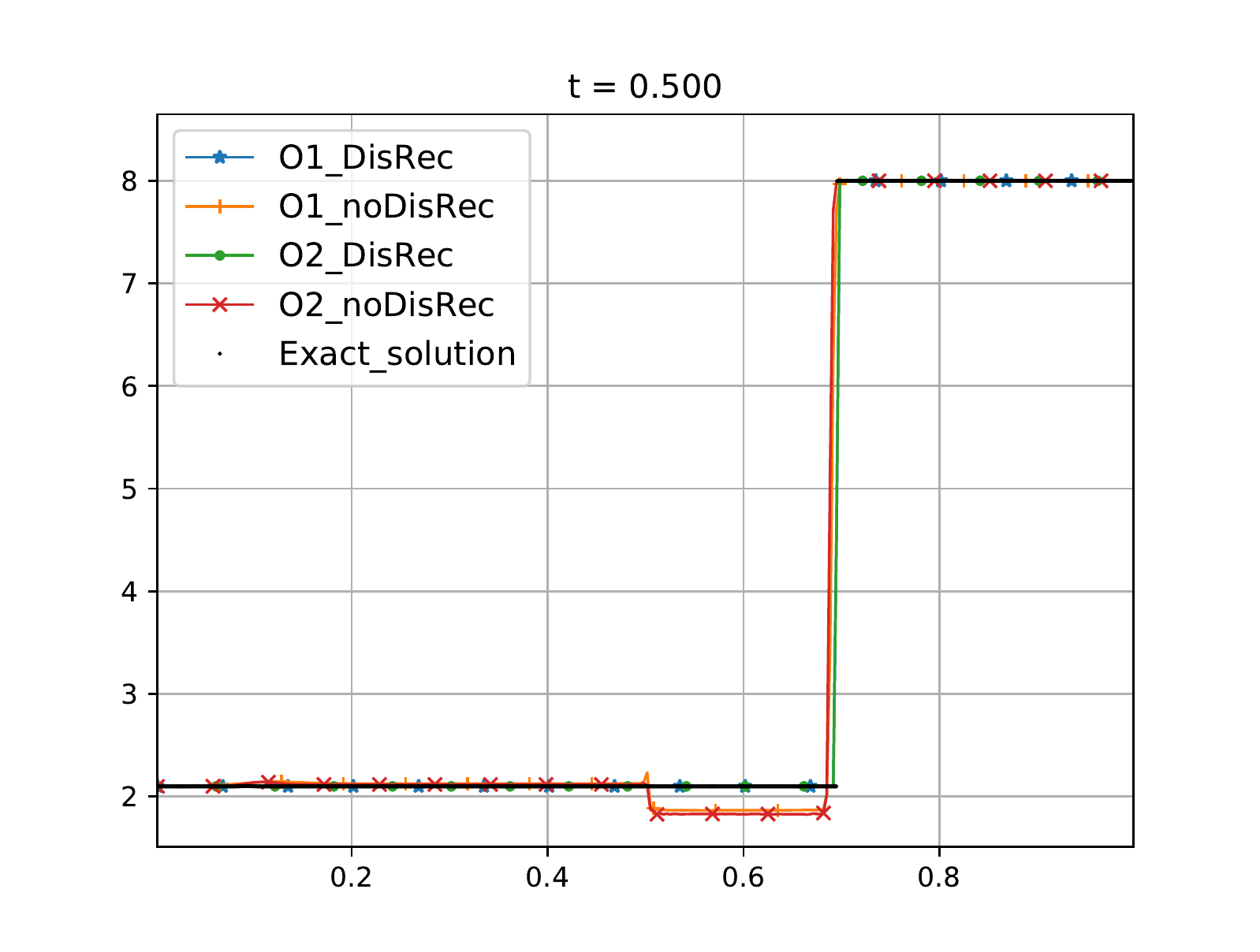}
		\end{subfigure}
		\caption{Gas dynamics equations in Lagrangian coordinates. Test 1: variable $\tau$. Left: initial condition. Right: exact solution and numerical solutions obtained at time $t = 0.5$ with 300 cells.}
		\label{fig:Gas_Test1_O1_vs_O2_DisRec_vs_noDisRec_Roe_tau}
	\end{figure}
	
	\begin{figure}[h]
		\begin{subfigure}{0.5\textwidth}
			\includegraphics[width=1.1\linewidth]{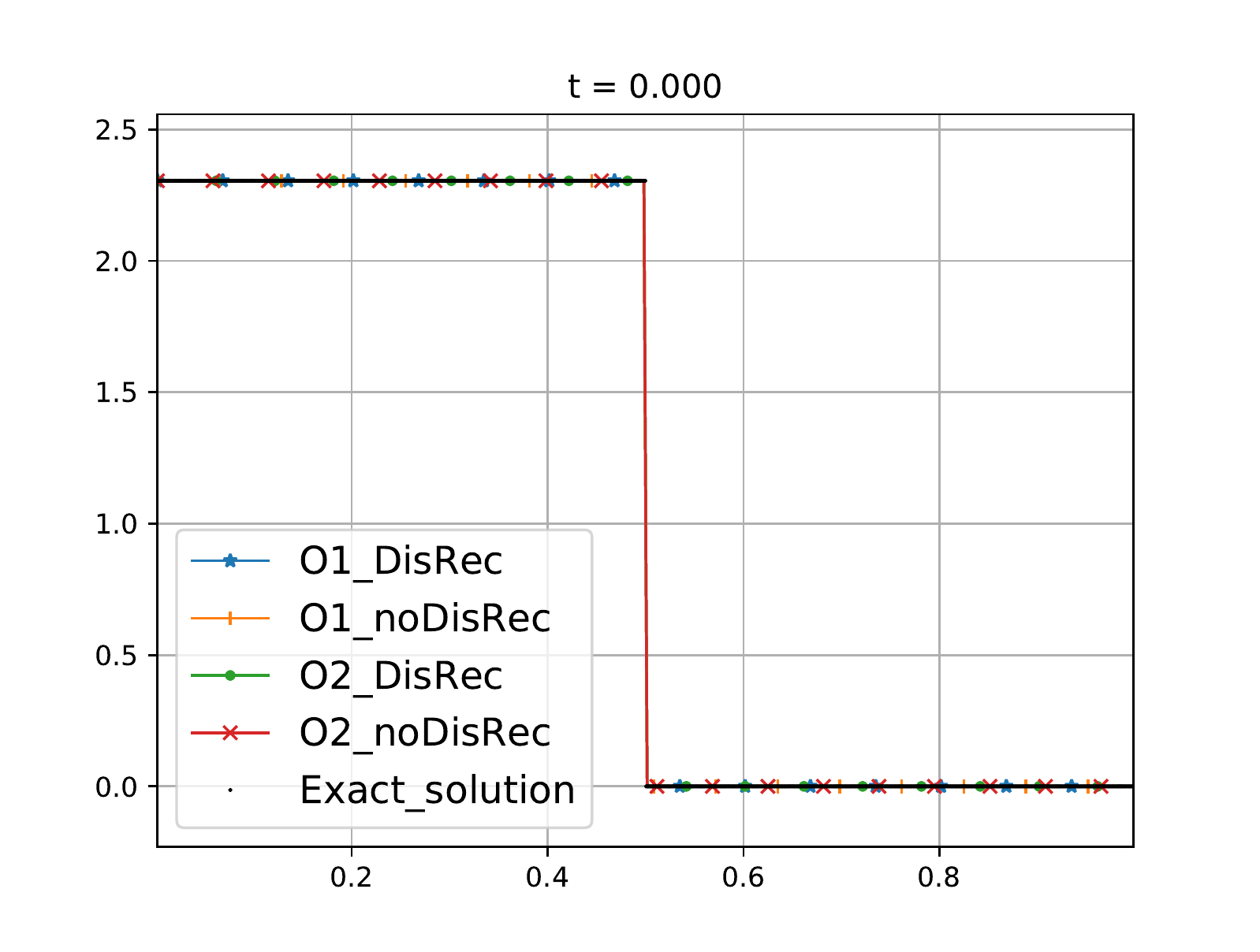}
		\end{subfigure}
		\begin{subfigure}{0.5\textwidth}
				\includegraphics[width=1.1\linewidth]{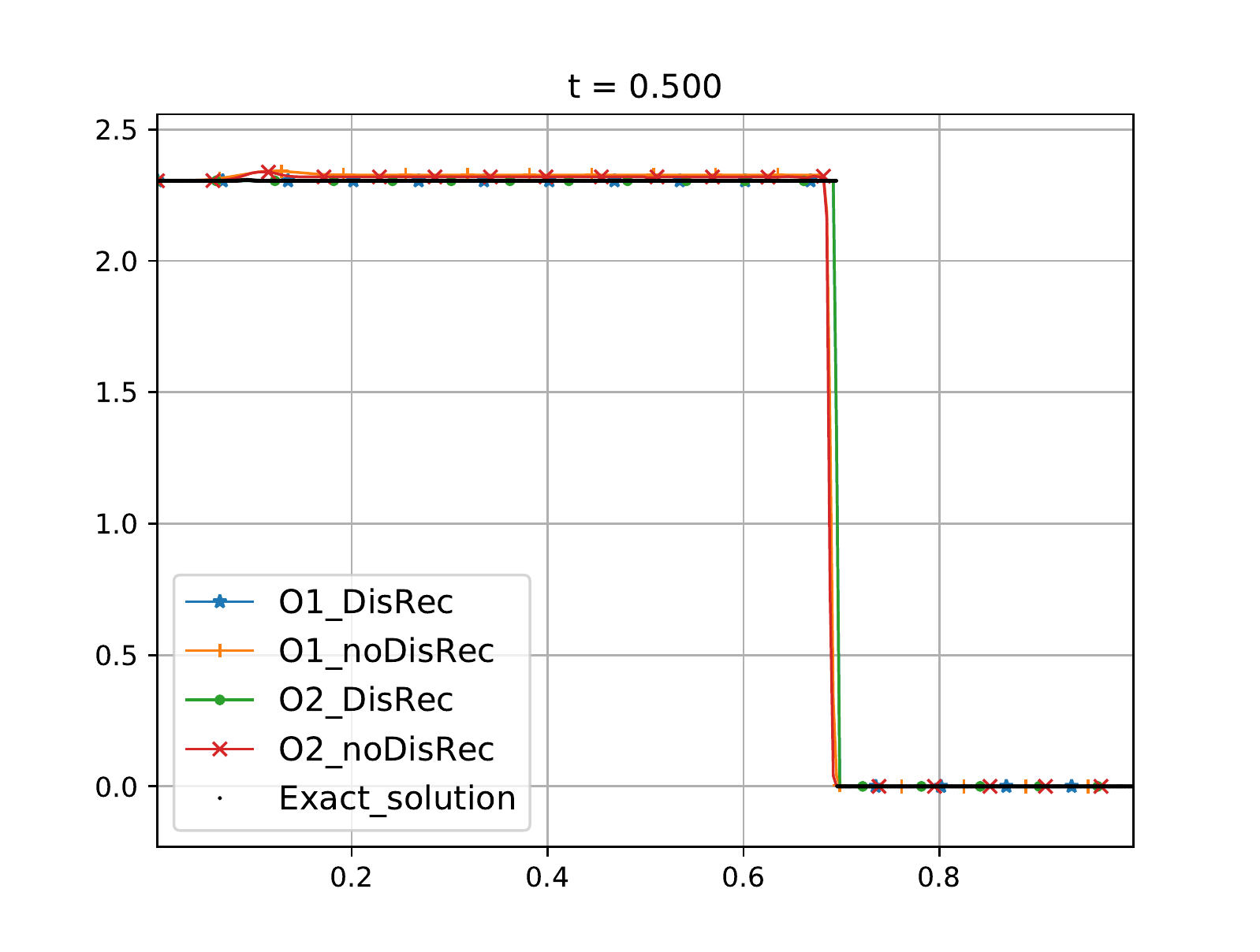}
		\end{subfigure}
		\caption{Gas dynamics equations in Lagrangian coordinates. Test 1: variable $u$. Left: initial condition. Right: exact solution and numerical solutions obtained at time $t = 0.5$ with 300 cells.}
		\label{fig:Gas_Test1_O1_vs_O2_DisRec_vs_noDisRec_Roe_u}
	\end{figure}
	
	\begin{figure}[h]
		\begin{subfigure}{0.5\textwidth}
			\includegraphics[width=1.1\linewidth]{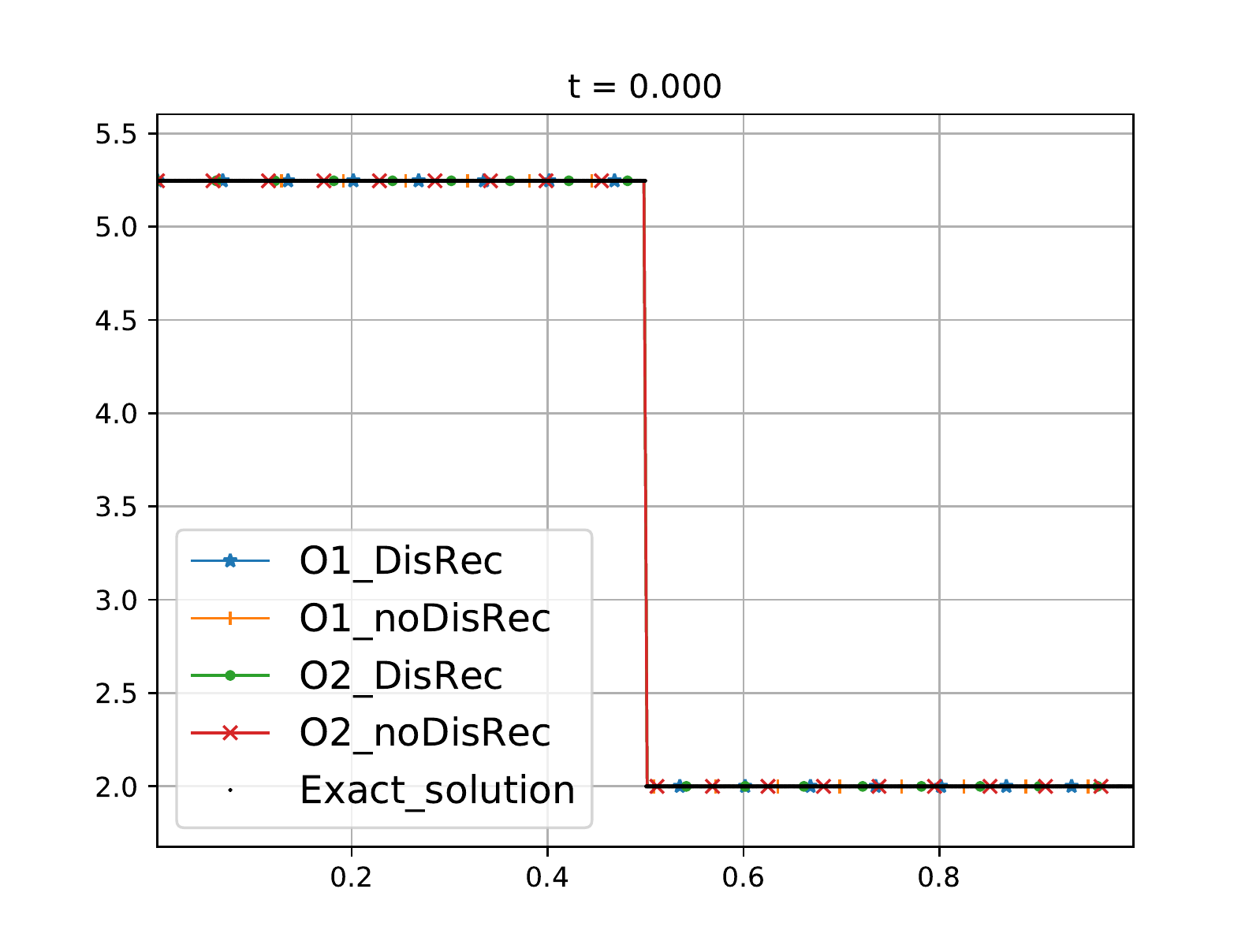}
		\end{subfigure}
		\begin{subfigure}{0.5\textwidth}
				\includegraphics[width=1.1\linewidth]{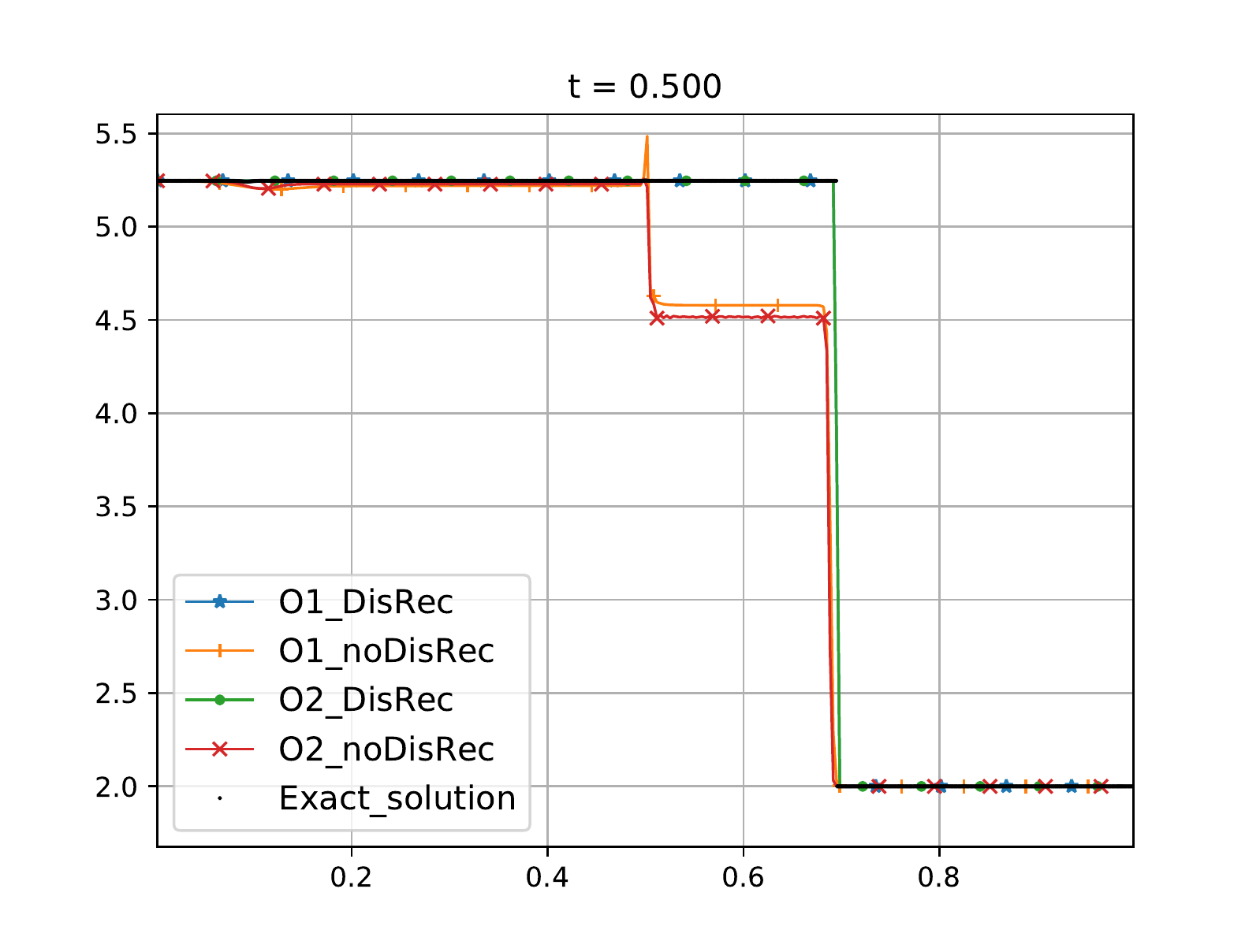}
		\end{subfigure}
		\caption{Gas dynamics equations in Lagrangian coordinates. Test 1: variable $e$. Left: initial condition. Right: exact solution and numerical solutions obtained at time $t = 0.5$ with 300 cells.}
		\label{fig:Gas_Test1_O1_vs_O2_DisRec_vs_noDisRec_Roe_e}
	\end{figure}
	
\subsubsection*{Test 2: 1-shock + contact discontinuity + 3-shock}

Let us consider the following initial condition taken from \cite{chalons2019path}
$$(\tau, u, p)_{0}(x)= \begin{cases}
     (5, 3.323013993227, 0.481481481481) & \text{if $x<0.5$,}  \\
     (8,0,0.1) & \text{otherwise.}
\end{cases}$$

The solution of the Riemann problem consists of a 1-shock wave with negative speed, a stationary contact discontinuity, and a 3-shock that coincides with the one in the first test problem.
Figures \ref{fig:Gas_Test2_O1_vs_O2_DisRec_vs_noDisRec_Roe_tau},  \ref{fig:Gas_Test2_O1_vs_O2_DisRec_vs_noDisRec_Roe_u}, and \ref{fig:Gas_Test2_O1_vs_O2_DisRec_vs_noDisRec_Roe_e} show the numerical solutions at time
$t = 0.5$ using a mesh of 300 cells and CFL = 0.5 and the conclusions are the same: the in-cell discontinuous reconstruction methods of order 1 and 2 get the exact solution while Roe method and its second-order extension based on the standard MUSCL reconstruction do not.

\begin{figure}[h]
		\begin{subfigure}{0.5\textwidth}
			\includegraphics[width=1.1\linewidth]{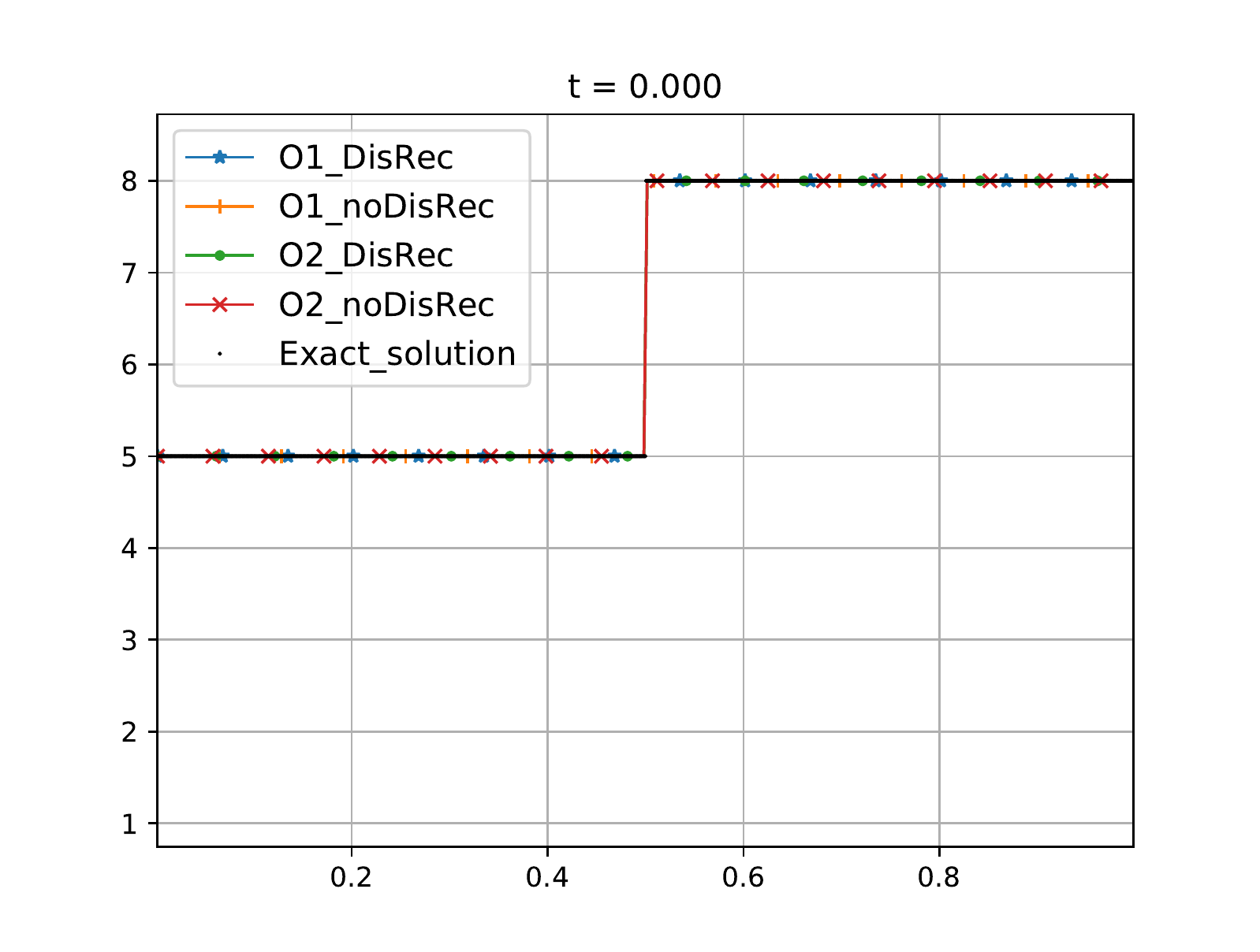}
		\end{subfigure}
		\begin{subfigure}{0.5\textwidth}
				\includegraphics[width=1.1\linewidth]{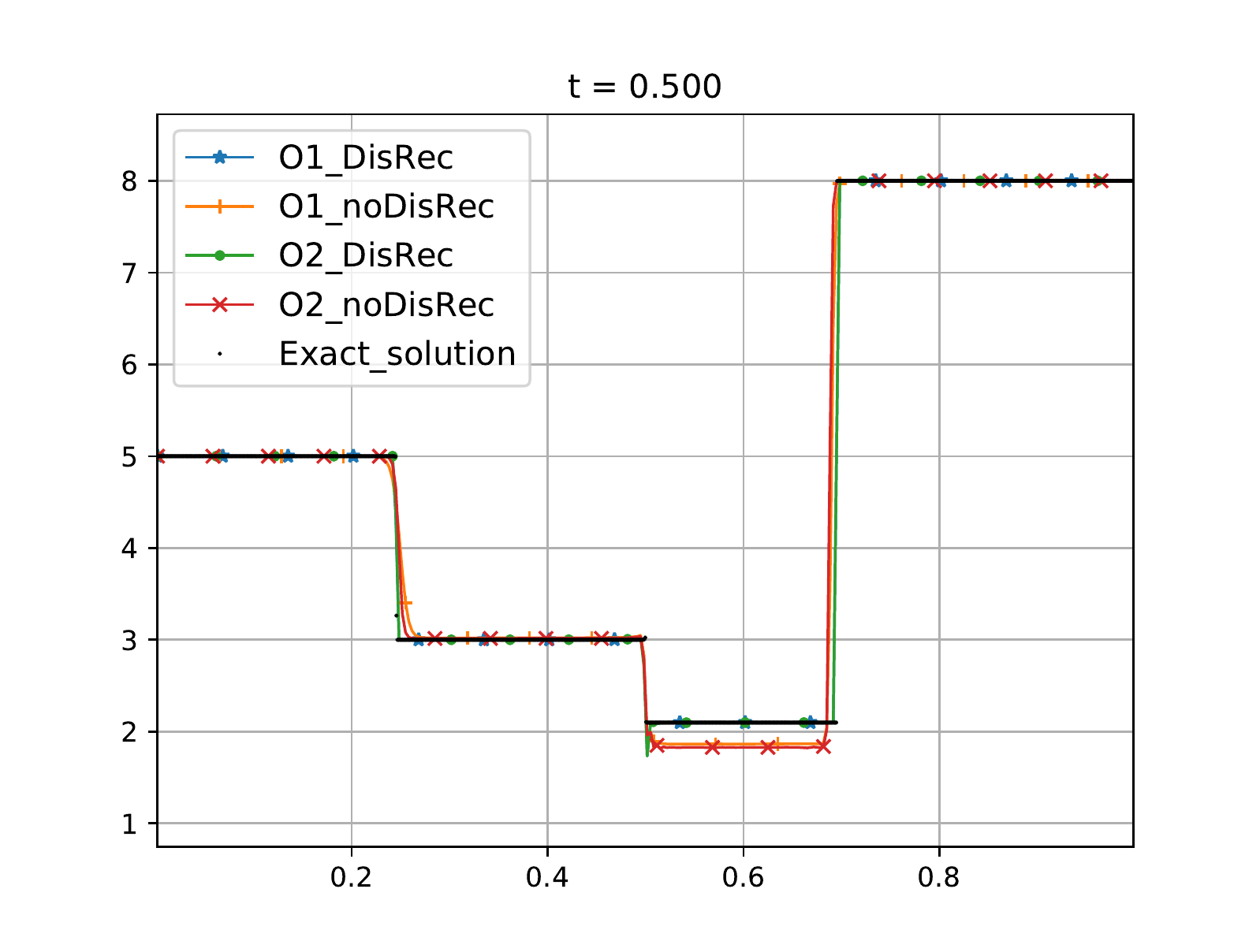}
		\end{subfigure}
		\caption{Gas dynamics equations in Lagrangian coordinates. Test 2: variable $\tau$. Left: initial condition. Right: exact solution and numerical solutions obtained at time $t = 0.5$ with 300 cells.}
		\label{fig:Gas_Test2_O1_vs_O2_DisRec_vs_noDisRec_Roe_tau}
	\end{figure}
	
	\begin{figure}[h]
		\begin{subfigure}{0.5\textwidth}
			\includegraphics[width=1.1\linewidth]{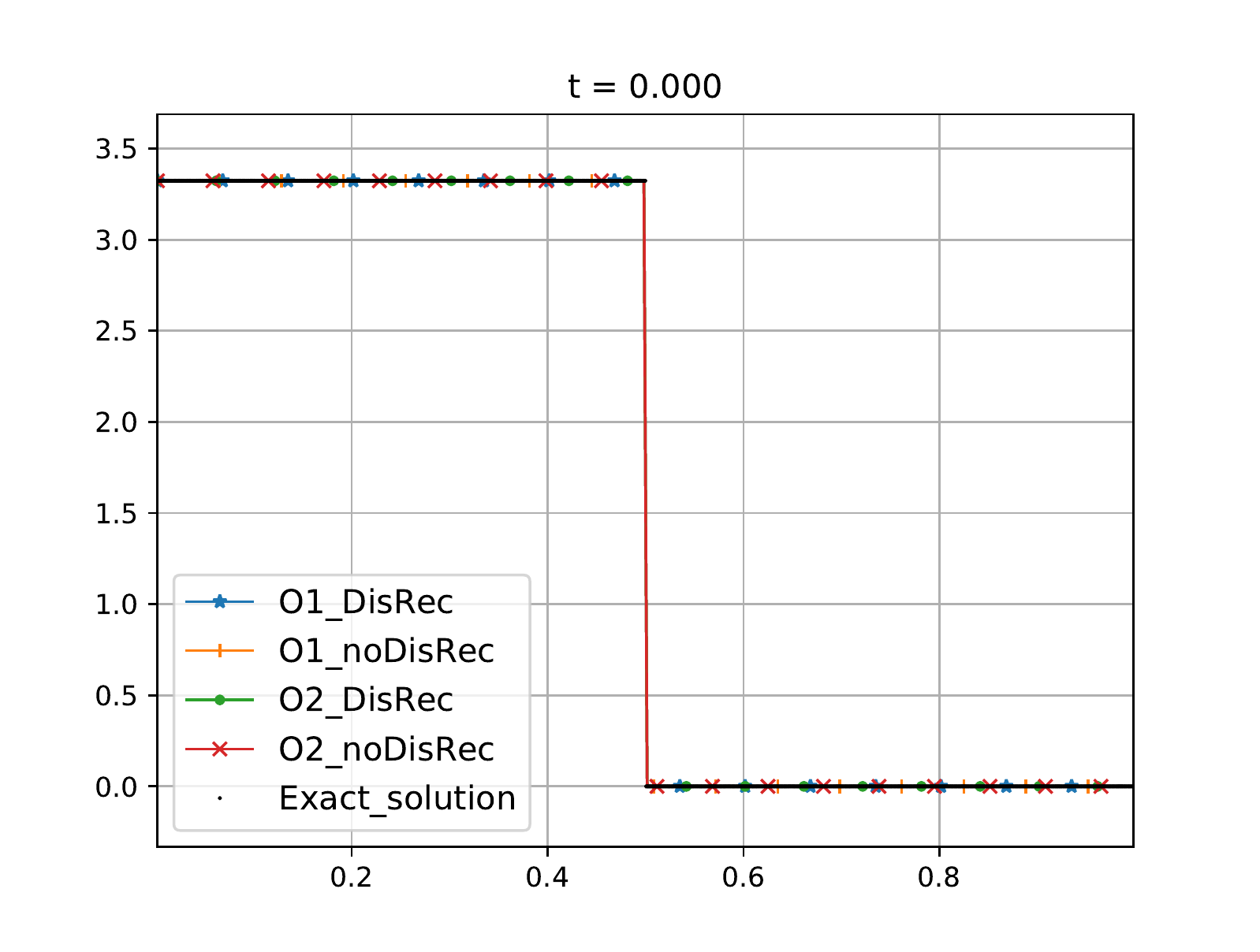}
		\end{subfigure}
		\begin{subfigure}{0.5\textwidth}
				\includegraphics[width=1.1\linewidth]{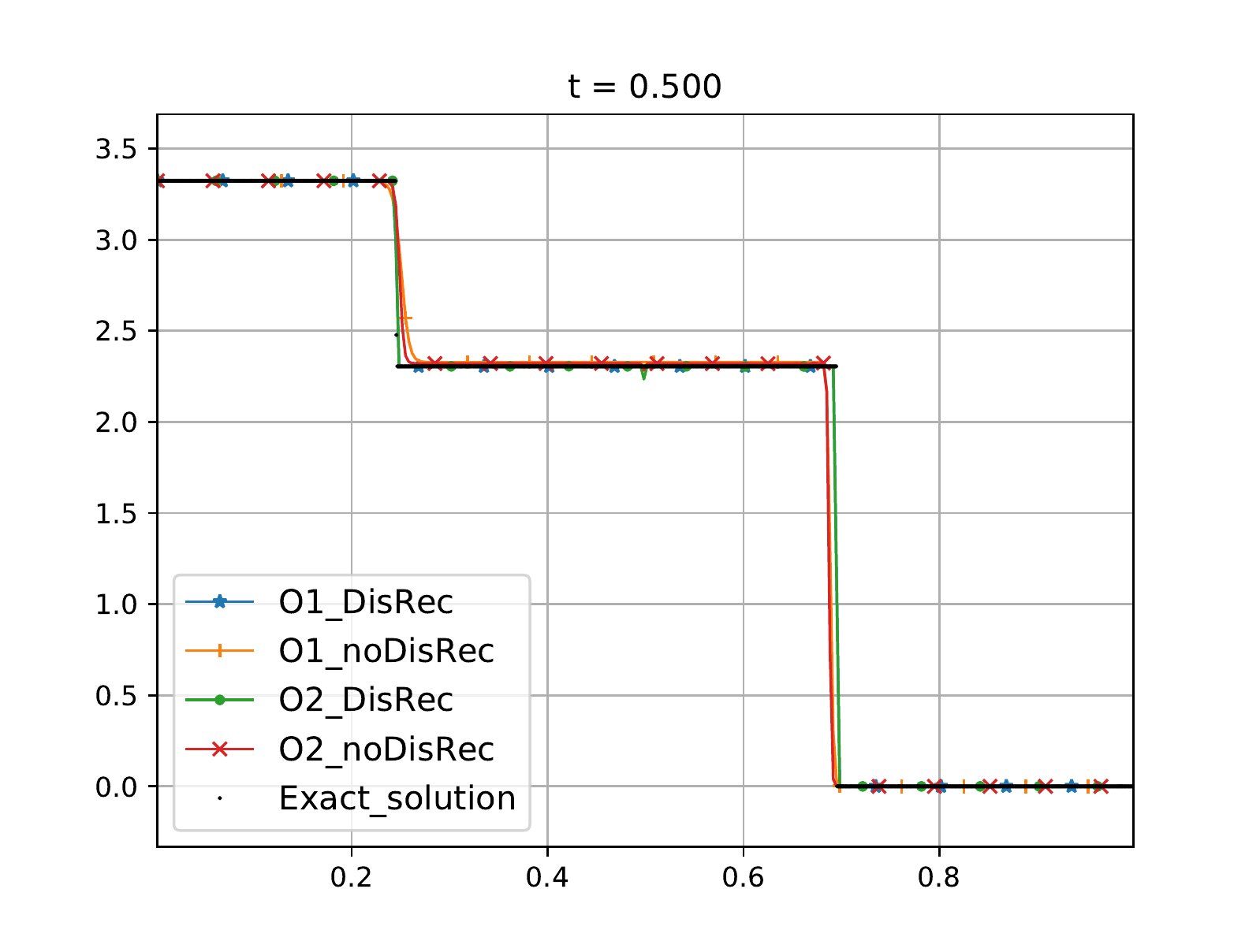}
		\end{subfigure}
		\caption{Gas dynamics equations in Lagrangian coordinates. Test 2: variable $u$. Left: initial condition. Right: exact solution and numerical solutions obtained at time $t = 0.5$ with 300 cells.}
		\label{fig:Gas_Test2_O1_vs_O2_DisRec_vs_noDisRec_Roe_u}
	\end{figure}
	
	\begin{figure}[h]
		\begin{subfigure}{0.5\textwidth}
			\includegraphics[width=1.1\linewidth]{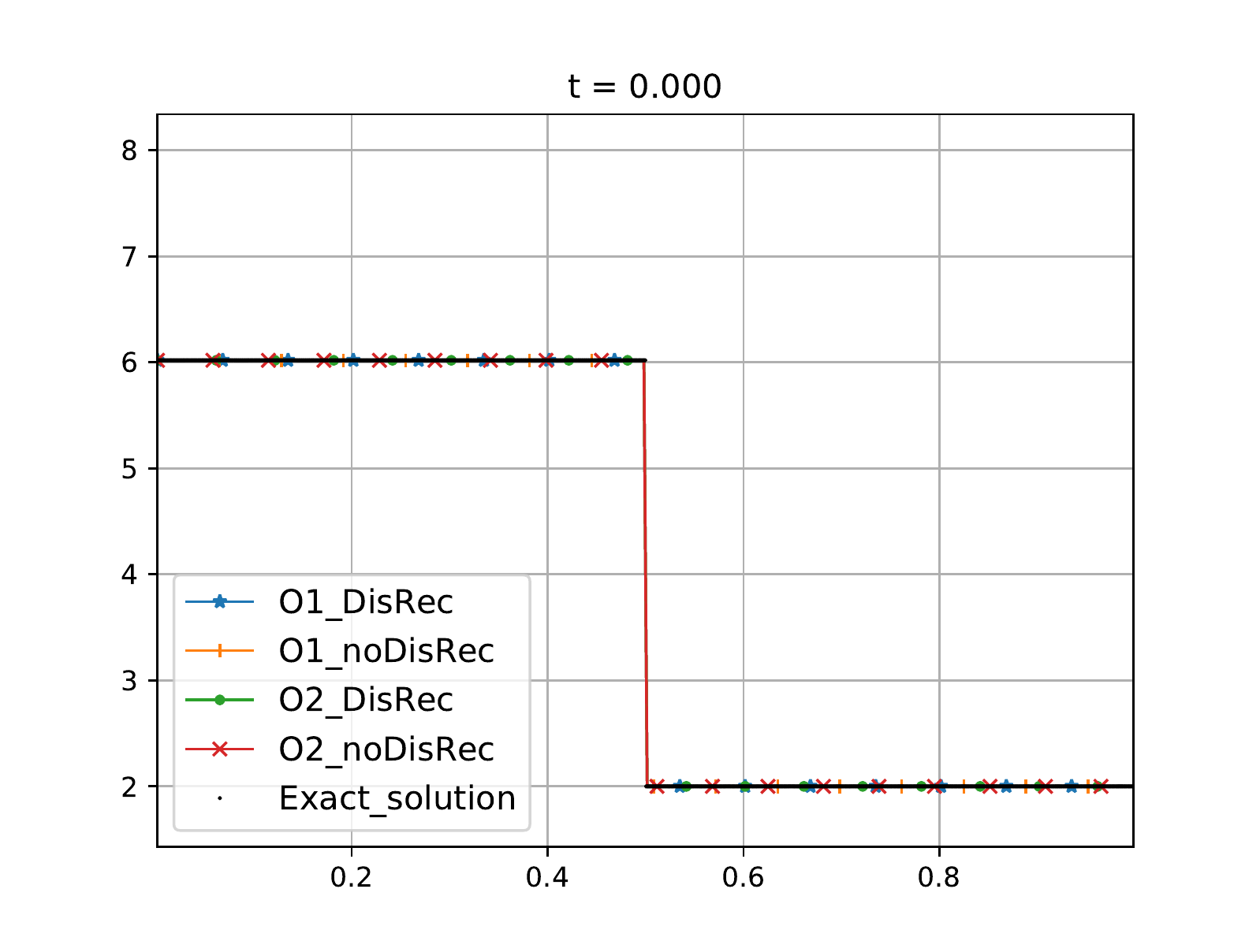}
		\end{subfigure}
		\begin{subfigure}{0.5\textwidth}
				\includegraphics[width=1.1\linewidth]{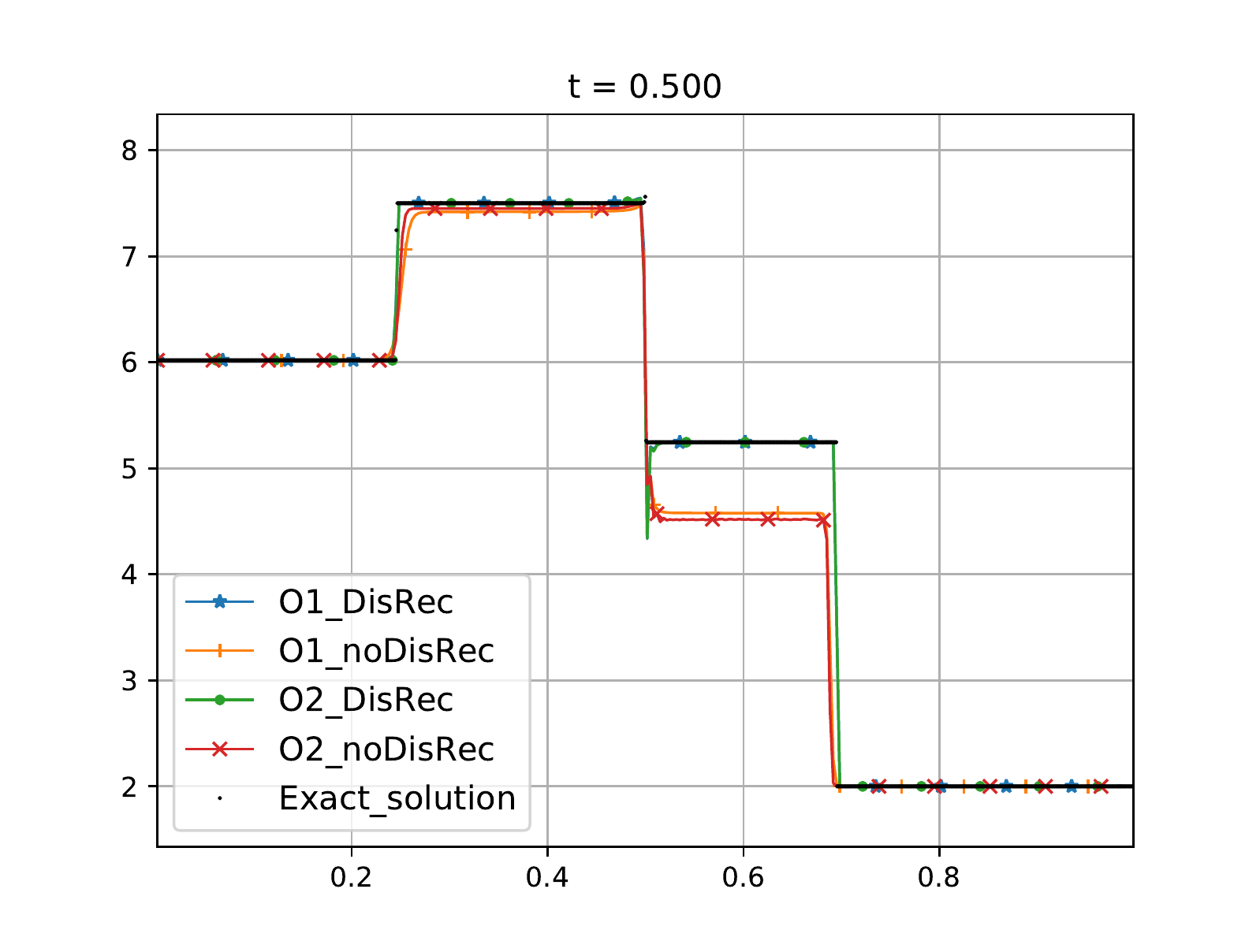}
		\end{subfigure}
		\caption{Gas dynamics equations in Lagrangian coordinates. Test 2: variable $e$. Left: initial condition. Right: exact solution and numerical solutions obtained at time $t = 0.5$ with 300 cells.}
		\label{fig:Gas_Test2_O1_vs_O2_DisRec_vs_noDisRec_Roe_e}
	\end{figure}
	
\subsubsection*{Test 3: 1-rarefaction + contact discontinuity + 3-shock}

Let us consider now the initial condition
$$(\tau, u, p)_{0}(x)= \begin{cases}
     (2.09836065573770281, 3.323013993227, 1) &\text{if $x<0.5$,}  \\
     (8,4,0.1)& \text{otherwise.}
\end{cases}$$
The solution of the Riemann problem consists of a 1-rarefaction wave whose head and tail have negative speeds, a stationary contact discontinuity,  and a 3-shock with positive speed.
Figures \ref{fig:Gas_Test3_O1_vs_O2_DisRec_vs_noDisRec_Roe_tau},  \ref{fig:Gas_Test3_O1_vs_O2_DisRec_vs_noDisRec_Roe_u}, and \ref{fig:Gas_Test3_O1_vs_O2_DisRec_vs_noDisRec_Roe_e} show the numerical solutions at time $t = 0.5$ using a mesh of 300 cells and CFL = 0.5. Although all the methods capture correctly the rarefaction wave, second order methods do it better, as expected; concerning the stationary contact discontinuity and the shock wave,  only the first and second order in-cell discontinuous reconstruction methods capture the exact solution.

\begin{figure}[h]
		\begin{subfigure}{0.5\textwidth}
			\includegraphics[width=1.1\linewidth]{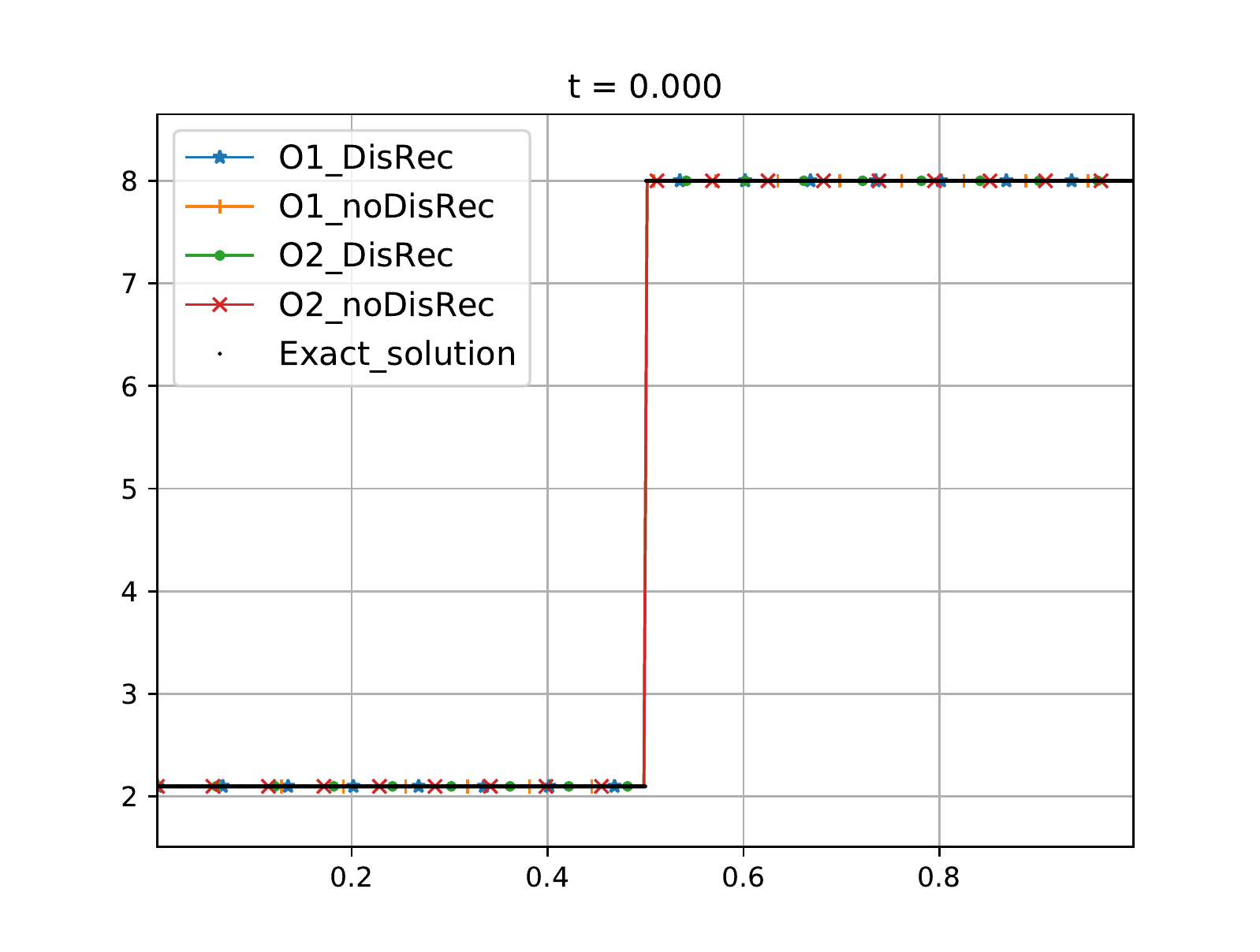}
		\end{subfigure}
		\begin{subfigure}{0.5\textwidth}
				\includegraphics[width=1.1\linewidth]{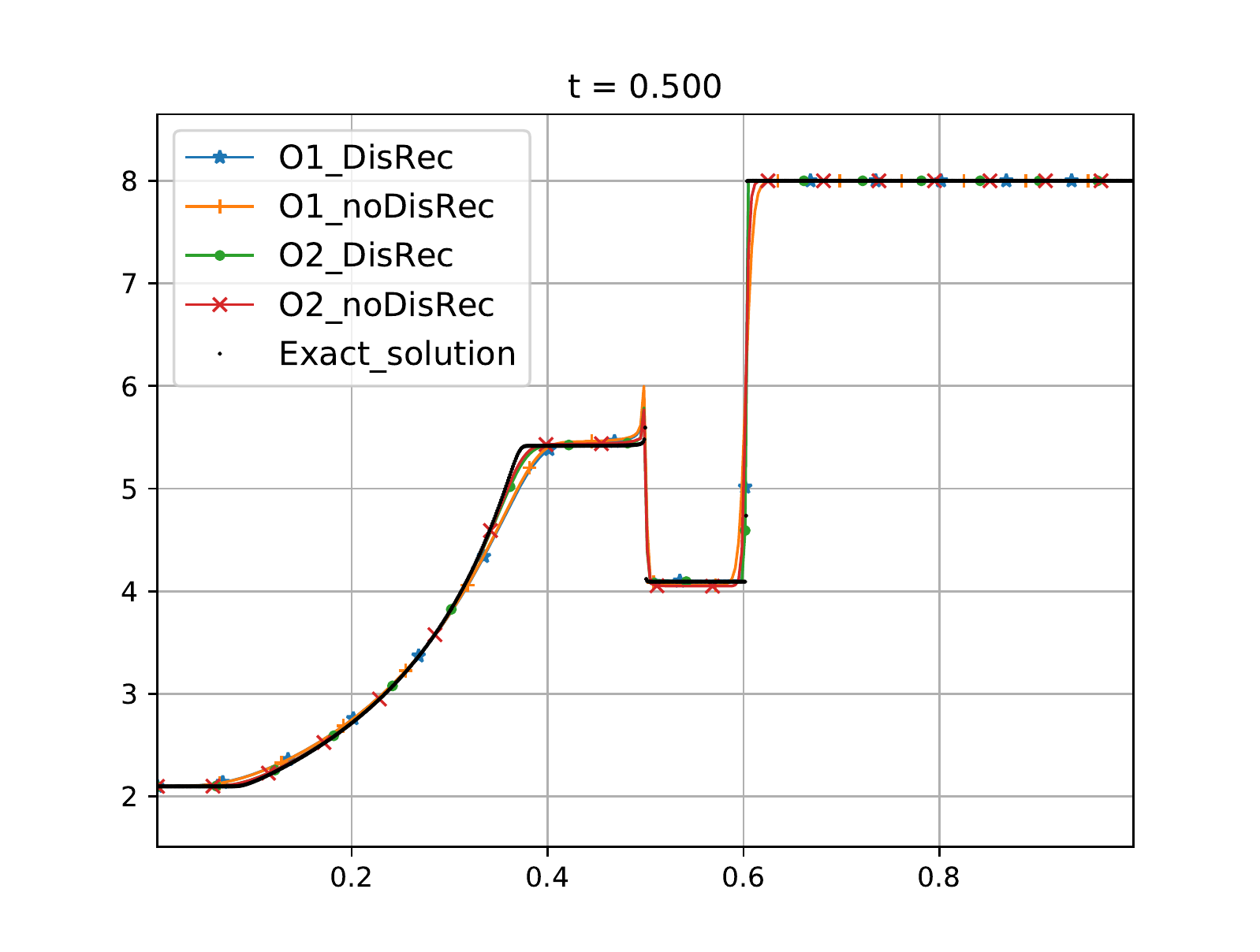}
		\end{subfigure}
		\newline
		\begin{subfigure}{0.5\textwidth}
			\includegraphics[width=1.1\linewidth]{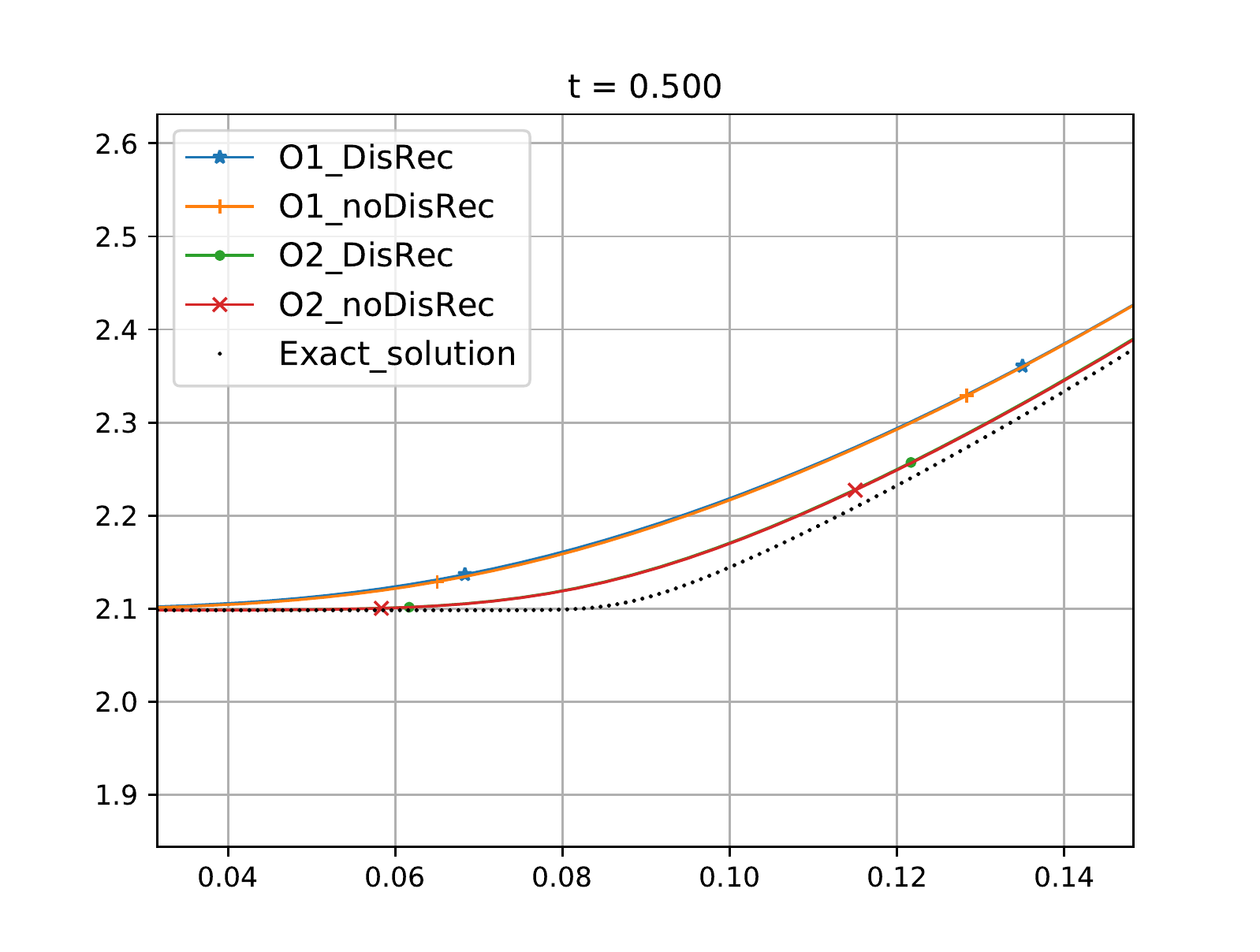}
			\caption{zoom rarefaction}
		\end{subfigure}
		\begin{subfigure}{0.5\textwidth}
				\includegraphics[width=1.1\linewidth]{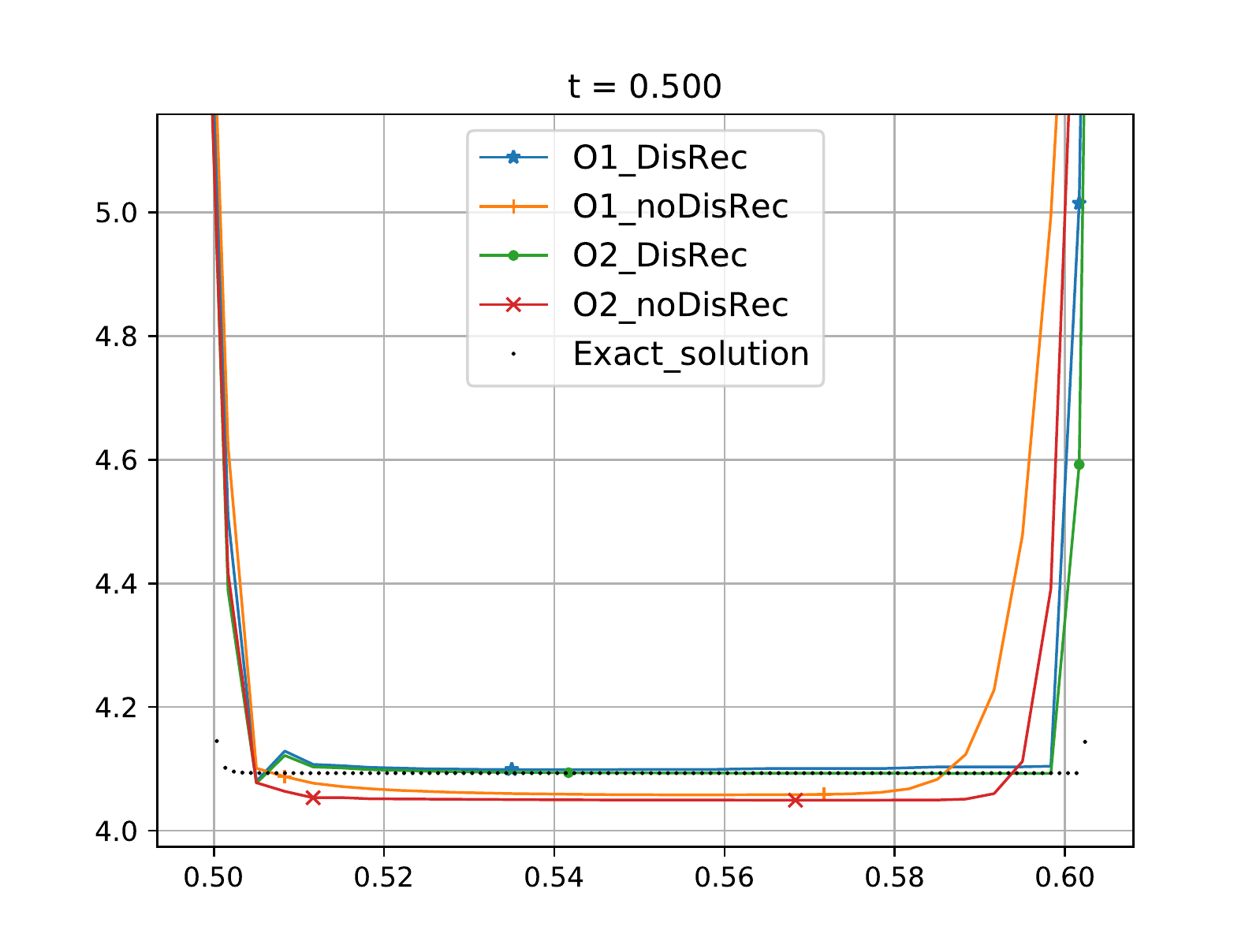}
				\caption{zoom shock}
		\end{subfigure}
		\caption{Gas dynamics equations in Lagrangian coordinates. Test 3: variable $\tau$. Top: initial condition (left),  exact solution and numerical solutions obtained at time $t = 0.5$ with 300 cells (right). Down: zooms of the rarefaction (left) and the shock waves (right) at time $t = 0.5$.}
		\label{fig:Gas_Test3_O1_vs_O2_DisRec_vs_noDisRec_Roe_tau}
	\end{figure}
	
	\begin{figure}[h]
		\begin{subfigure}{0.5\textwidth}
			\includegraphics[width=1.1\linewidth]{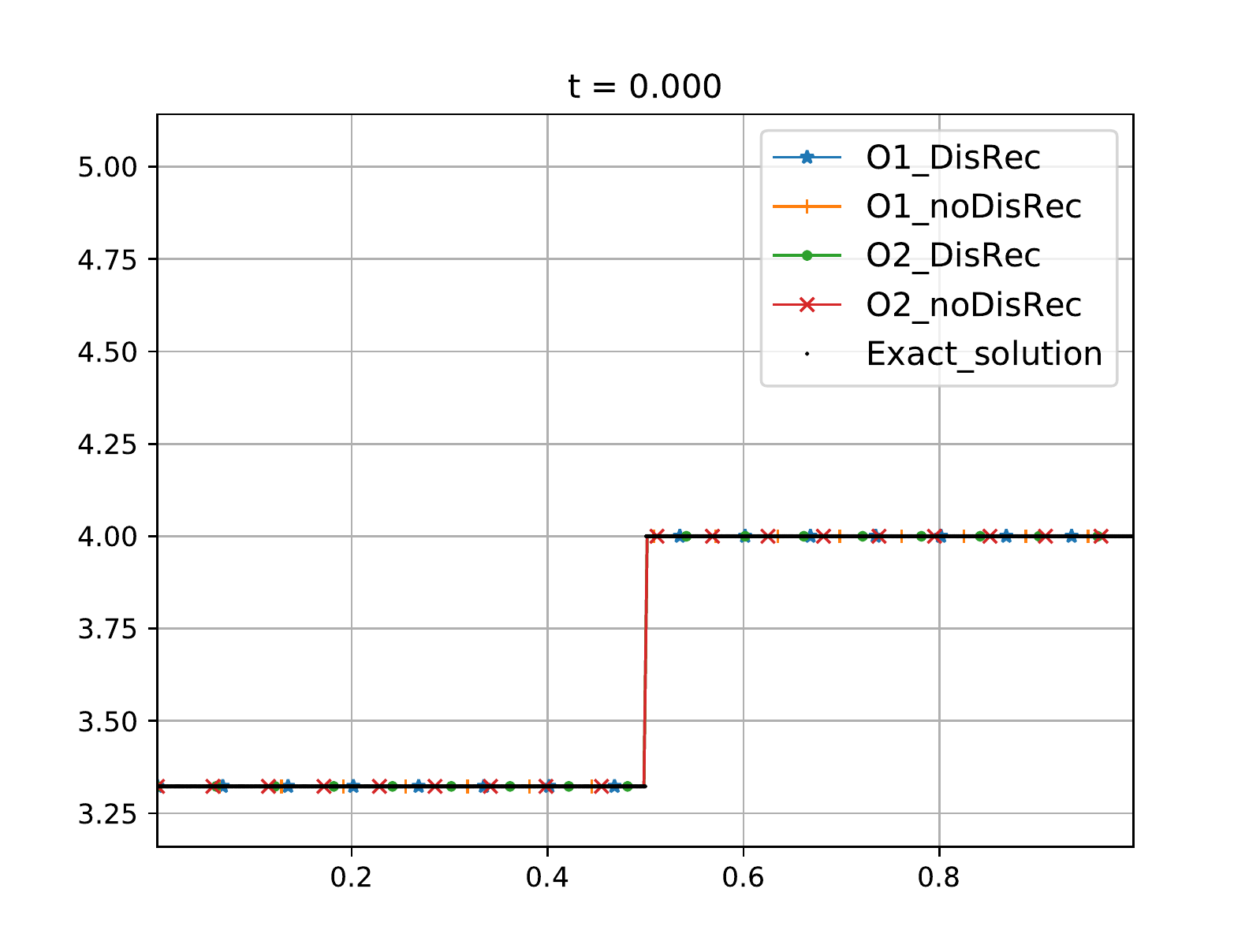}
		\end{subfigure}
		\begin{subfigure}{0.5\textwidth}
				\includegraphics[width=1.1\linewidth]{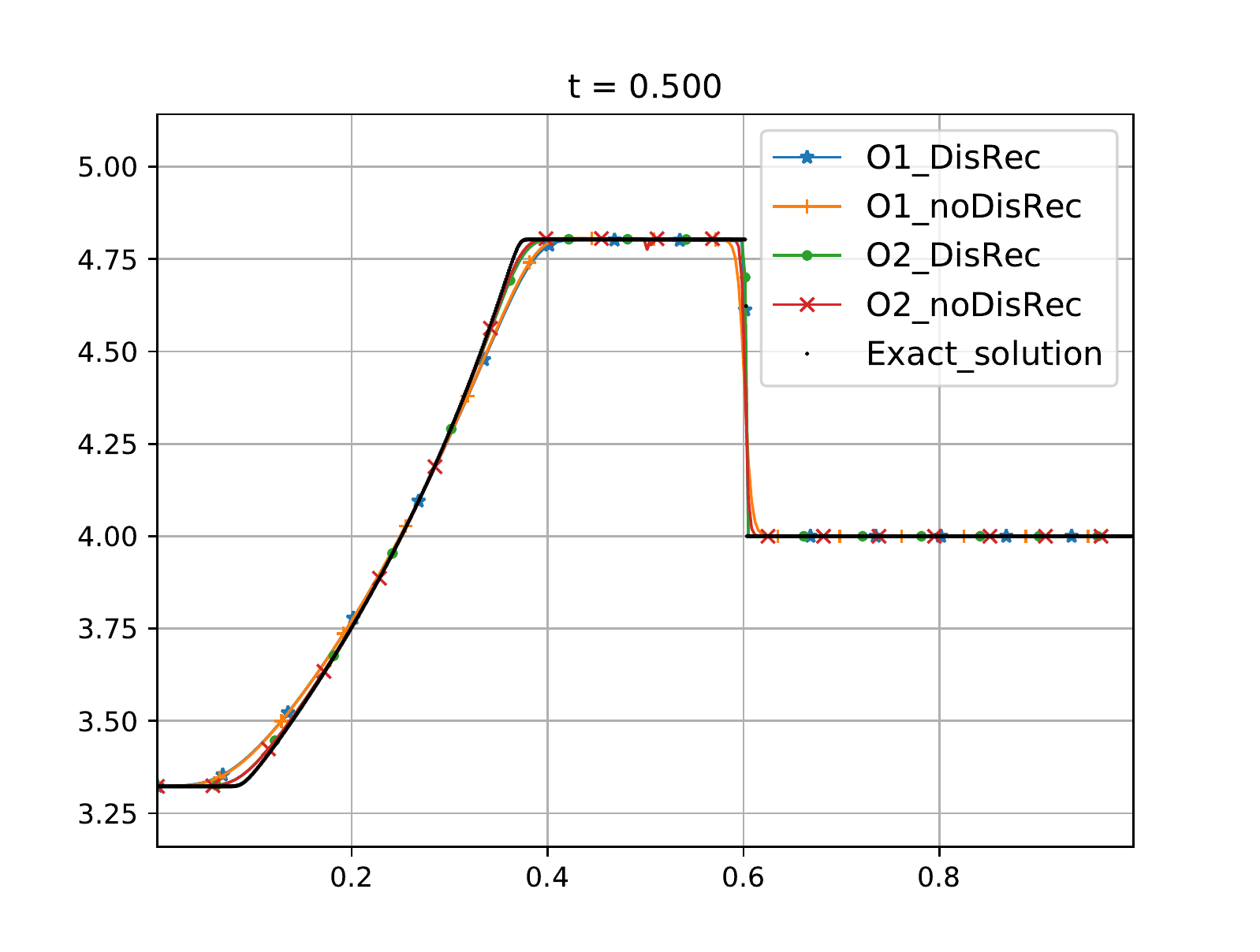}
		\end{subfigure}
		\newline
		\begin{subfigure}{0.5\textwidth}
			\includegraphics[width=1.1\linewidth]{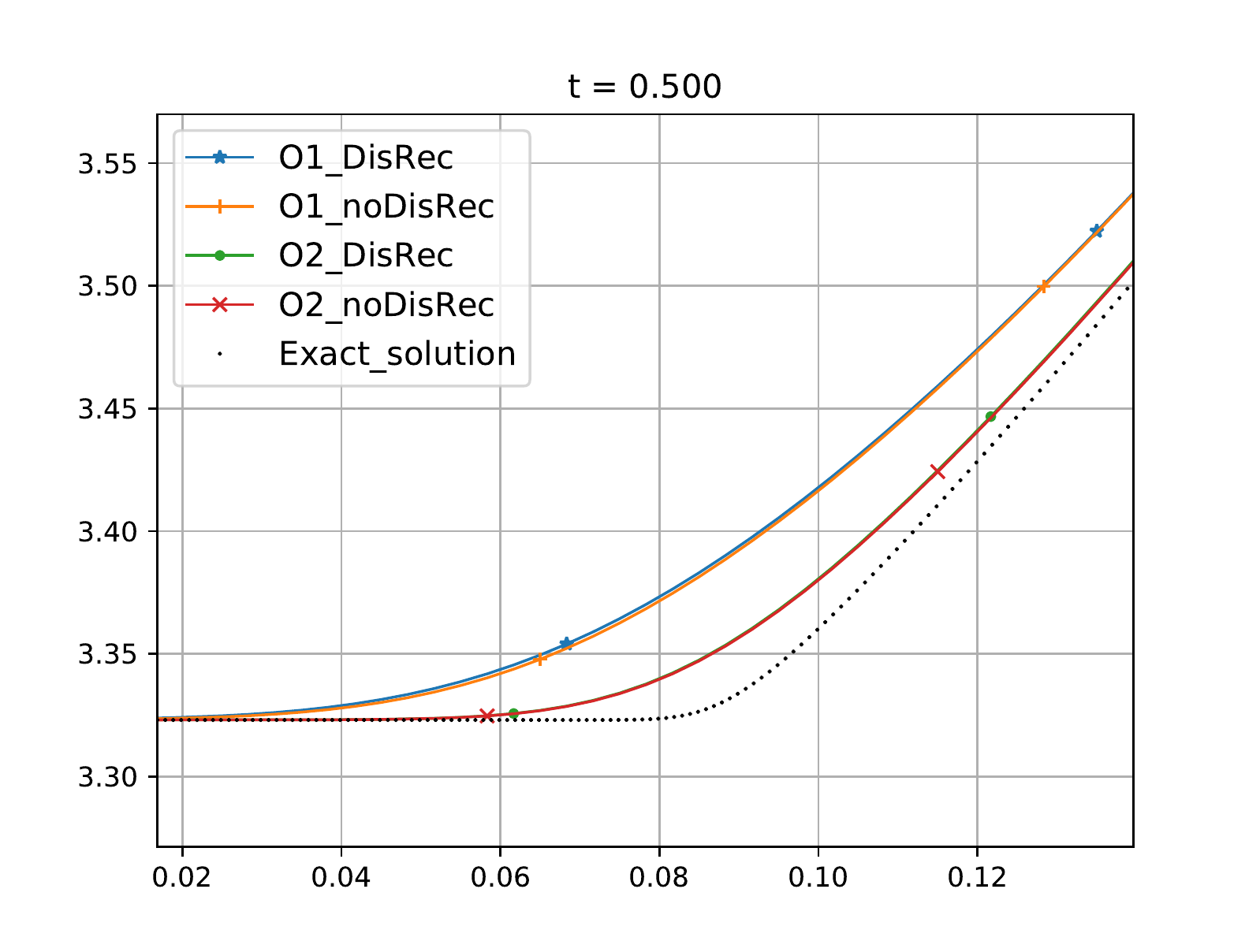}
			\caption{zoom rarefaction}
		\end{subfigure}
		\begin{subfigure}{0.5\textwidth}
				\includegraphics[width=1.1\linewidth]{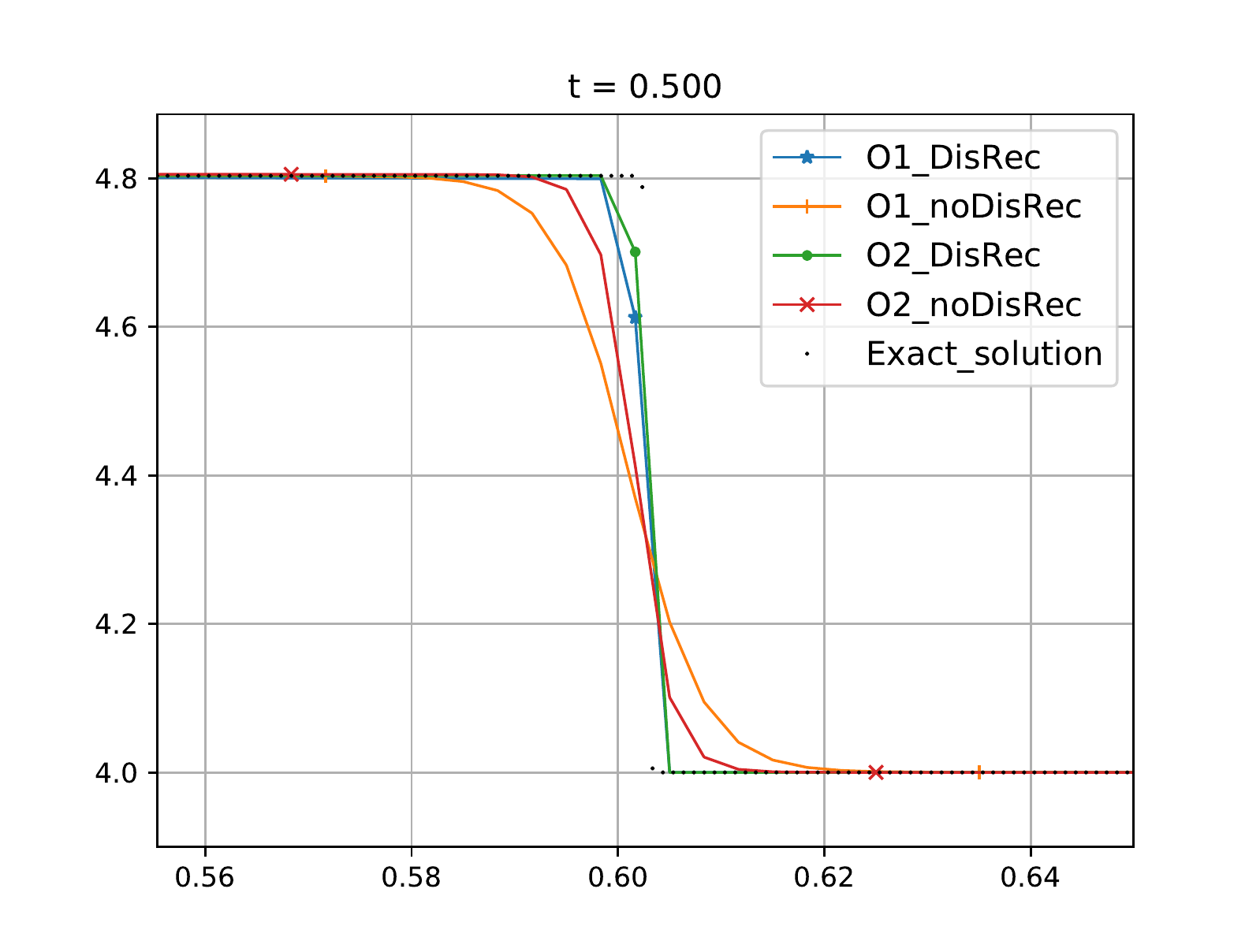}
				\caption{zoom shock}
		\end{subfigure}
		\caption{Gas dynamics equations in Lagrangian coordinates. Test 3: variable $u$. Top: initial condition (left),  exact solution and numerical solutions obtained at time $t = 0.5$ with 300 cells (right). Down: zooms of the rarefaction (left) and the shock waves (right) at time $t = 0.5$.}
		\label{fig:Gas_Test3_O1_vs_O2_DisRec_vs_noDisRec_Roe_u}
	\end{figure}
	
	\begin{figure}[h]
		\begin{subfigure}{0.5\textwidth}
			\includegraphics[width=1.1\linewidth]{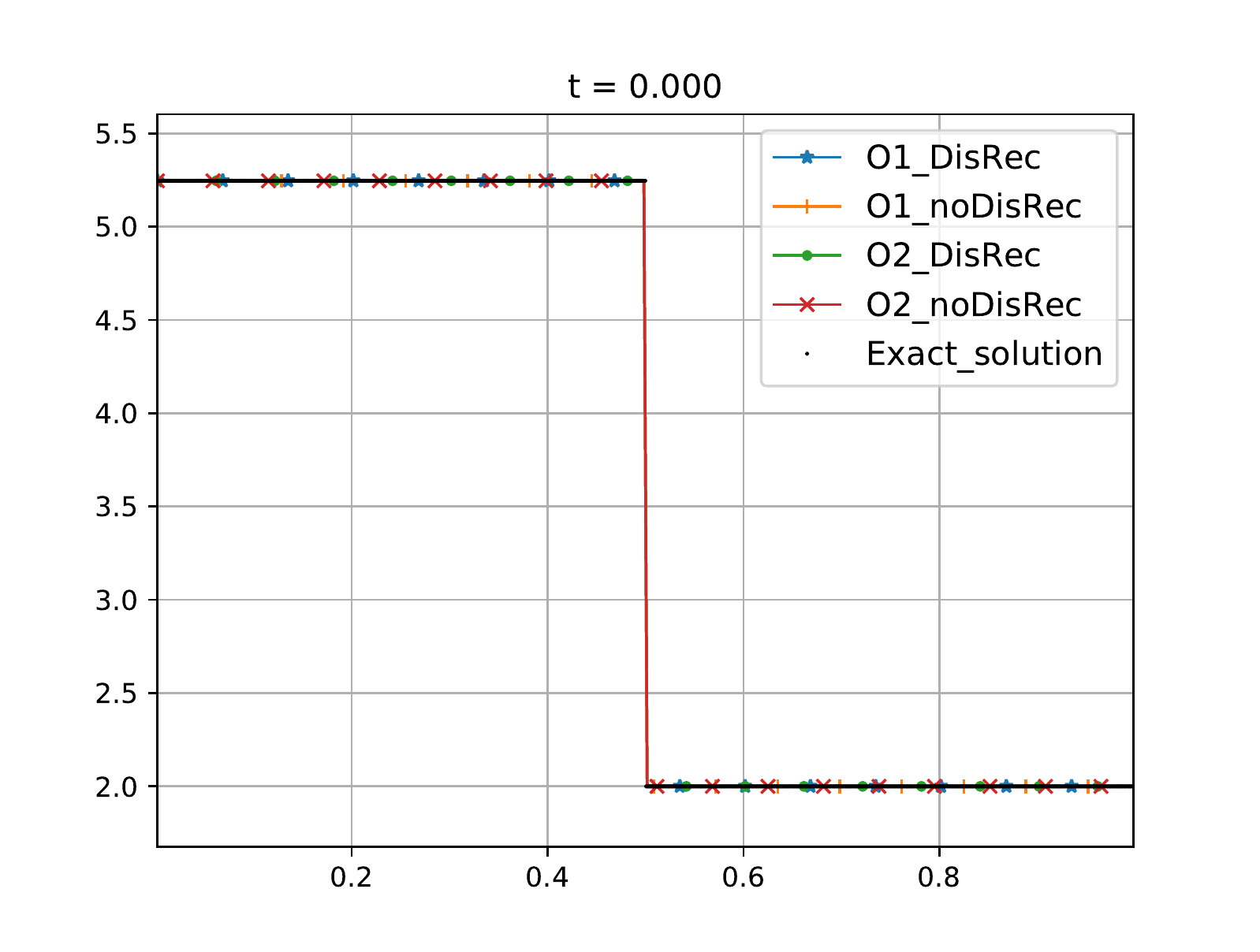}
		\end{subfigure}
		\begin{subfigure}{0.5\textwidth}
				\includegraphics[width=1.1\linewidth]{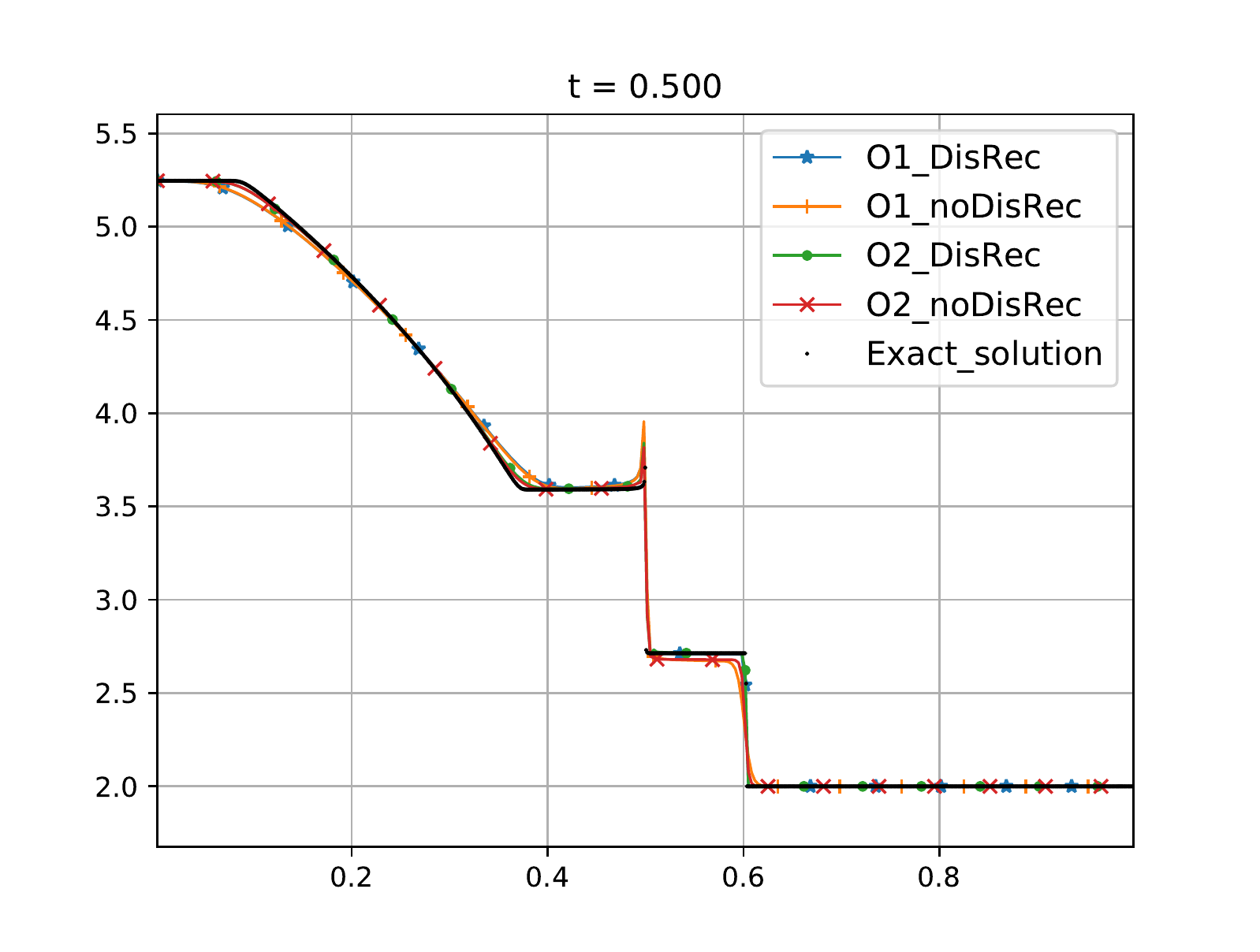}
		\end{subfigure}
		\newline
		\begin{subfigure}{0.5\textwidth}
			\includegraphics[width=1.1\linewidth]{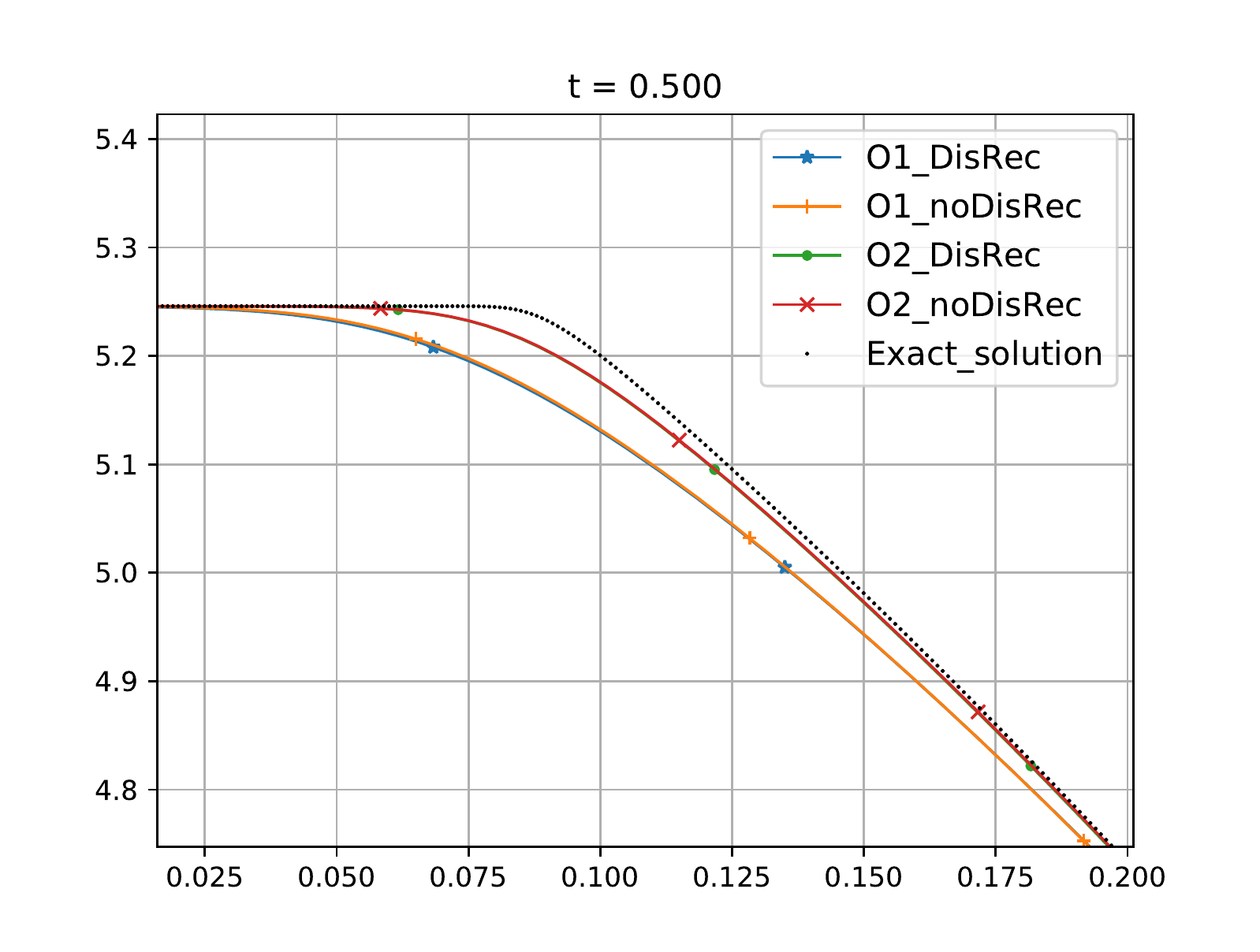}
			\caption{zoom rarefaction}
		\end{subfigure}
		\begin{subfigure}{0.5\textwidth}
				\includegraphics[width=1.1\linewidth]{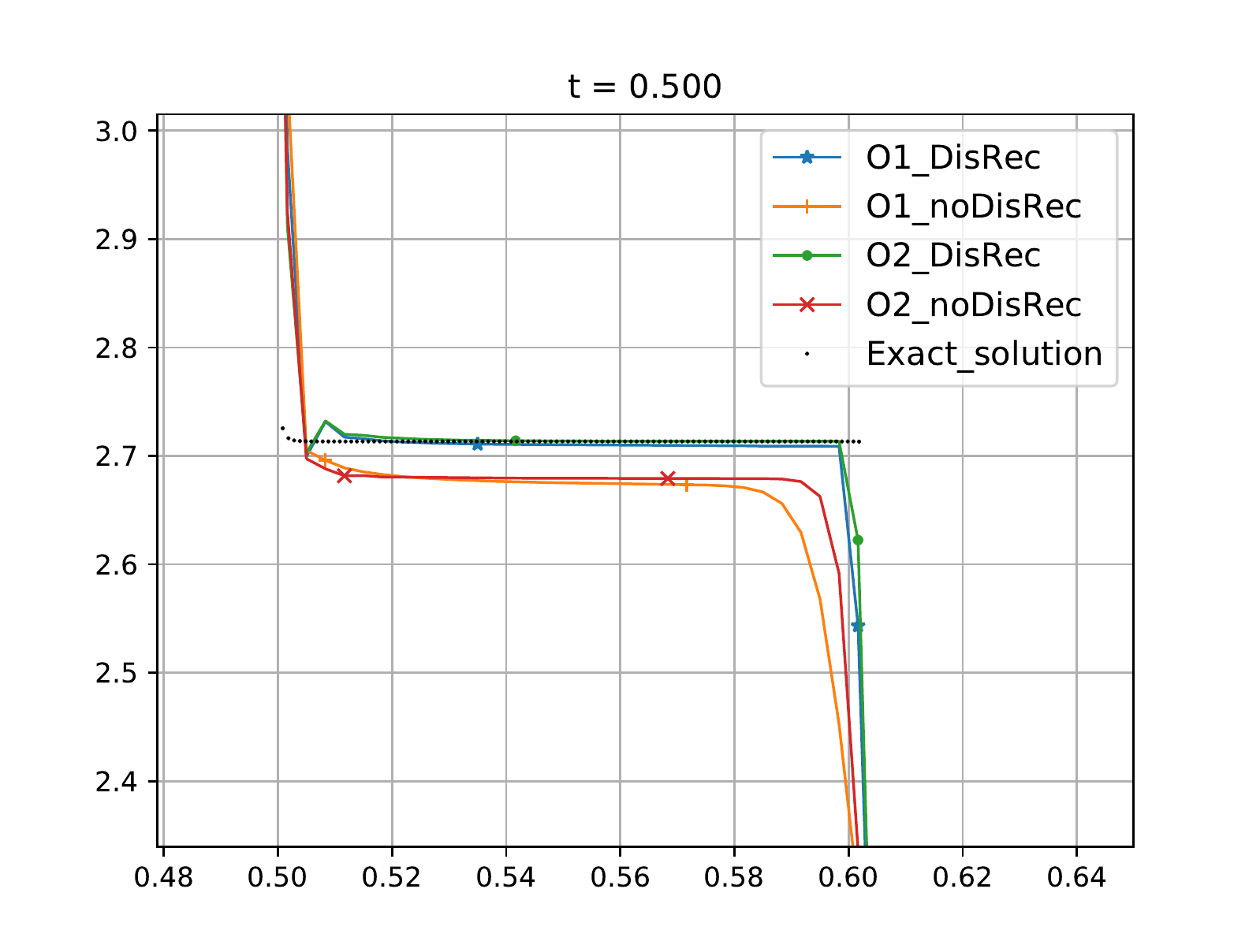}
				\caption{zoom shock}
		\end{subfigure}
		\caption{Gas dynamics equations in Lagrangian coordinates. Test 3: variable $e$. Top: initial condition (left),  exact solution and numerical solutions obtained at time $t = 0.5$ with 300 cells (right). Down: zooms of the rarefaction (left) and the shock waves (right) at time $t = 0.5$.}
		\label{fig:Gas_Test3_O1_vs_O2_DisRec_vs_noDisRec_Roe_e}
	\end{figure}
	
\subsection{Modified Shallow Water system}
Let us consider the modified Shallow Water system introduced in \cite{Castro2008}: 
\begin{equation} \label{simplifiedSW}
\left\{
\begin{array}{l}
\partial_t h + \partial_x q = 0,\\
\smallskip
\displaystyle\partial_t q + \partial_x \left(\frac{q^2}{h}\right) + qh\partial_x h = 0,
\end{array}
\right.
\end{equation}
where $\bu = (h,q)^t$ belongs to $\Omega = \{\bu \in \mathbb{R}^{2}| \quad 0<q, \ 0<h<(16q)^{1/3}\}$. This system can be written in the form (\ref{sys:nonconservative}) with
$$
\mathcal{A}(\bu) = \left[ \begin{array}{cc} 0  & 1 \\ -u^2+uh^2 & 2u \end{array} \right],
$$
being $u=q/h$. The system is strictly hyperbolic $\Omega$ with eigenvalues
$$\lambda_{1}(\bu) = u-h\sqrt{u}, \quad \lambda_{2}(\bu) = u+h\sqrt{u},$$
whose characteristic fields, given by the eigenvectors
$$
R_1(\bu) = [1 , u-h\sqrt{u}]^T, \quad R_2(\bu) = [1, u+h\sqrt{u}]^T,
$$
are genuinely nonlinear. Once the family of paths has been chosen, the simple waves of this system are:
\begin{itemize}
    \item 1-rarefaction waves joining states $\bu_l$, $\bu_r$ such that
    $$ h_r < h_l, \quad \sqrt{u_l} +h_l/2 = \sqrt{u_r} +h_r/2, $$
    and 2-rarefaction waves joining states $\bu_l$, $\bu_r$ such that
    $$ h_l < h_r, \quad \sqrt{u_l} -h_l/2 = \sqrt{u_r} -h_r/2. $$
    
    \item 1-shock and 2-shock waves joining states $\bu_l$ and $\bu_r$ such that
    $h_l <  h_r$ or $h_r < h_l$ respectively, that satisfy the jump conditions:
    \begin{eqnarray*}
    \sigma[h] & = & \left[q \right],\\
    \sigma[q] & = & \left[\frac{q^2}{h}\right] + \int_0^1 \phi_q(s; \bu_l, \bu_r) \phi_h(s; \bu_l, \bu_r) 
    \partial_s \phi_h(s; \bu_l, \bu_r) \,ds.
    \end{eqnarray*}
\end{itemize}
If, for instance, the following family of path is chosen:
$$
\phi(s; \bu_l, \bu_r) =  \left[\begin{array}{c}
     \phi_h(s; \bu_l, \bu_r)  \\
     \phi_q(s; \bu_l, \bu_r) 
\end{array}\right] = \left\{\begin{array}{l}
     \left[\begin{array}{c}
     h_l + 2s(h_r-h_l)  \\
     q_l 
\end{array}\right] \quad \text{if $0\leq s \leq \frac{1}{2}$},  \\
\\
     \left[\begin{array}{c}
     h_r  \\
     q_l + (2s-1)(q_r-q_l) 
\end{array}\right] \quad \text{if $\frac{1}{2} \leq s \leq 1$},
\end{array}\right.
$$
the jump conditions reduce to:
    \begin{eqnarray*}
    \sigma[h] & = & \left[q \right],\\
    \sigma[q] & = & \left[\frac{q^2}{h}\right] + q_l  \left[\frac{h^2}{2}\right].
    \end{eqnarray*}
 If this family of paths has been selected and Lax's entropy criterion is used, the simple waves of the system are as follows:
    \begin{itemize}
        \item Given a left-hand state $\bu_l$, the 1-shock $\mathcal{S}_1(\bu_l)$ and the 2-shock $\mathcal{S}_2(\bu_l)$ curves consisting of all the right-hand states that can be connected with $\bu_l$ through a 1-shock and a 2-shock wave respectively, are:
        \begin{equation}
            \mathcal{S}_1(\bu_l): u=u_l-\sqrt{\frac{u_l(h+h_l)}{2h}}(h-h_l), \quad h>h_l,
        \end{equation}
        \begin{equation}
            \mathcal{S}_2(\bu_l): u=u_l-\sqrt{\frac{u_l(h+h_l)}{2h}}(h-h_l), \quad h<h_l.
        \end{equation}
        Moreover, given two states $\bu_l$ and $\bu_r$ connected by a 1-shock wave or a 2-shock wave, the speed of the shock is given by:
        \begin{equation}
            \sigma_1(\bu_l,\bu_r) = u_l-\sqrt{h_ru_l\frac{h_l+h_r}{2}},
        \end{equation}
        \begin{equation}
            \sigma_2(\bu_l,\bu_r) = u_l+\sqrt{h_ru_l\frac{h_l+h_r}{2}},
        \end{equation}
        respectively.
        
        \item Given a left-hand state $\bu_l$, the 1-rarefaction $\mathcal{R}_1(\bu_l)$ and the 2-rarefaction $\mathcal{R}_2(\bu_l)$ consisting of all the right-hand states that can be connected with $\bu_l$ through a 1-rarefaction and a 2-rarefaction wave, respectively, are:
        \begin{equation}
            \mathcal{R}_1(\bu_l): u=\left(\frac{h_l-h}{2} + \sqrt{u_l}\right)^2, \quad h<h_l,
        \end{equation}
        \begin{equation}
            \mathcal{R}_2(\bu_l): u=\left(\frac{h-h_l}{2} + \sqrt{u_l}\right)^2, \quad h>h_l.
        \end{equation}
    \end{itemize}

The criterion to mark the cells  is the following:
\begin{enumerate}
    \item If $h^n_{j+1} > h^n_{j-1}$ and
        $$
        u^n_{j-1} -\sqrt{\frac{u^n_{j-1}
        (h^n_{j+1}+h^n_{j-1})}{2h^n_{j+1}}}(h^n_{j+1}-h^n_{j-1}) < u^n_{j+1}< \left(\frac{h^n_{j+1}-h^n_{j-1}}{2}+\sqrt{u^n_{j-1}}\right)^2,
        $$
        the solution of the Riemann problem consists of a 1-shock and a 2-rarefaction waves:   the cell is marked.
    \item If $h^n_{j+1} < h^n_{j-1}$ and 
        $$
        u^n_{j-1} +\sqrt{\frac{u^n_{j-1}
        (h^n_{j+1}+h^n_{j-1})}{2h^n_{j+1}}}(h^n_{j+1}-h^n_{j-1}) < u^n_{j+1}< \left(\frac{h^n_{j-1}-h^n_{j+1}}{2}+\sqrt{u^n_{j-1}}\right)^2,
        $$
        the solution of the Riemann problem consists of a 1-rarefaction and a 2-shock waves: the cell is marked. 
    \item If $h^n_{j+1} > h^n_{j-1}$ and 
    $$
    u^n_{j+1}< u^n_{j-1} -\sqrt{\frac{u^n_{j-1}
    (h^n_{j+1}+h^n_{j-1})}{2h^n_{j+1}}}(h^n_{j+1}-h^n_{j-1}),
    $$
    or $h^n_{j+1} < h^n_{j-1}$ and
     $$
     u^n_{j+1}< u^n_{j-1} +\sqrt{\frac{u^n_{j-1}
    (h^n_{j+1}+h^n_{j-1})}{2h^n_{j+1}}}(h^n_{j+1}-h^n_{j-1}),
    $$
    the solution of the Riemann problem consists of  a 1-shock and a 2-shock waves: the cell is marked.  
    \item Otherwise the solution of the Riemann problem consists of two rarefactions and the cell is not marked.
\end{enumerate}

A Roe matrix is given in this case by
$$
\mathcal{A}(\bu_l,\bu_r)= \left[\begin{array}{cc}
   0  &  1\\
   -\bar{u}^{2}+q_l\bar{h}  & 2\bar{u}
\end{array}\right],
$$
where
$$
\bar{u} = \frac{\sqrt{h_l}u_{l} + \sqrt{h_r}u_{r}}{\sqrt{h_l}+\sqrt{h_r}}, \quad \bar{h} = \frac{h_l+h_r}{2}.
$$
The following strategy based on this Roe matrix (see Subsection \ref{ss:strategy}) is used to select the speed, and the left and right states of the discontinuous reconstruction: 
\begin{itemize}
    \item If the solution of the Riemann problem consists of a  1-shock and a 2-rarefaction waves (case 1):
        $$
        \sigma_j^n =\bar{u} - h_{j-1}^n\sqrt{\bar{u}}, \quad  \bu^n_{j,l} = \bu^n_{j-1},  \quad \bu^n_{j,r} = \bu^n_{j-1} + \alpha_{1}R_1(\bu^n_{j-1},  \bu^n_{j+1}),
        $$
        where $\bar u$ is the Roe average of $u^n_{j-1}$ and $u^n_{j+1}$,  and $\alpha_{k}$, $k=1, 2$ represent the coordinates of $\bu^n_{j+1}-\bu^n_{j-1}$ in the basis of eigenvectors of the Roe matrix, i.e. $ \bu^n_{j+1}-\bu^n_{j-1} = \sum_{k=1}^{2} \alpha_{k}R_{k}(\bu^n_{j-1}, \bu^n_{j+1})$.
    \item If the solution of the Riemann problem consists of a 1-rarefaction and a 2-shock waves (case 2):
        $$
         \sigma_j^n =\bar{u} + h_{j-1}^n\sqrt{\bar{u}}, \quad \bu^n_{j,l} =  \bu^n_{j+1} - \alpha_{2}R_2(\bu^n_{j-1},  \bu^n_{j+1}), \quad \bu^n_{j,r} = \bu^n_{j+1}.
        $$
    \item If the solution of the Riemann problem consists of  a 1-shock and a 2-shock waves (case 3) we select one of them depending on the amplitude of the $\alpha_1$ and $\alpha_2$ coefficients in order to choose the 'dominant' one: 
    \begin{itemize}
        \item If $|\alpha_1| \leq |\alpha_2|$ then:
        $$
        \sigma_j^n =\bar{u} + h_{j-1}^n\sqrt{\bar{u}}, \quad \bu^n_{j,l} =  \bu^n_{j+1} - \alpha_{2}R_2(\bu^n_{j-1},  \bu^n_{j+1}), \quad \bu^n_{j,r} = \bu_{j+1}.
        $$
        \item If $|\alpha_1| > |\alpha_2|$ then:
        $$
        \sigma_j^n =\bar{u} - h_{j-1}^n\sqrt{\bar{u}}, \quad  \bu^n_{j,l} = \bu^n_{j-1},  \quad \bu^n_{j,r} = \bu^n_{j-1} + \alpha_{1}R_1(\bu_{j-1},  \bu_{j+1}).
        $$
    \end{itemize}
\end{itemize}
The variable $h$ is selected in \eqref{conservation}.

According to Theorem \ref{th}, the corresponding first and second-order in-cell discontinuous reconstruction methods capture correctly isolated shock waves and, as it will be also seen in Test \hyperref[ssstest3]{3}, it also captures correctly the solution of Riemann problems consisting of two shock waves traveling in the same direction. Nevertheless, although it improves the results obtained with the standard methods and gets closer to the exact solution when the mesh is refined, it fails in capturing exactly the solution of Riemann problems involving two shocks traveling in opposite directions: the reason is that the intermediate state linking the two shocks is not exactly captured by Roe method.

% Nevertheless, they fail in capturing correctly Riemann solutions involving a 1-shock and a 2-shock traveling in opposite directions, as it will be seen in the numerical tests. {\color{red}{The problem is that we are not able to obtain the intermediate exactly so we will not be able to capture the two isolated shocks. It is true that if initially we consider the two isolated shocks we get both correctly but the problem is to reach that situation.}}

Nevertheless, a more sophisticated strategy based on the exact solution of the Riemann problems (see Subsection \ref{ss:strategy}) allows one to  handle correctly with these situations. The key ingredients are:
\begin{itemize}
    \item The solution of the Riemann problem with initial data $\bu^{n-1}_{j-1,r}$ and $\bu^{n-1}_{j+1,l}$ is used to mark the cells instead of the one corresponding to the initial data $\bu^n_{j-1}$ and $\bu^n_{j+1}$, where $\bu^{n-1}_{j-1,r}$ and $\bu^{n-1}_{j+1,l}$ are the states selected in the discontinuous reconstruction in the previous time step.
    
    \item The exact intermediate state is used when the solution of the Riemann problem involves two shock  waves.
    
    \item If the solution of this Riemann problem involves two shock waves traveling in the same direction, a reconstruction with two discontinuities (one for each of the shock waves)  is considered, so that the complete structure of the Riemann solution is imposed.
    
\end{itemize}

 In order to avoid an excess of indices the following notation will be used:
$$
\bu^{n-1}_{j-1,r} = \bu_L = [h_L,q_L]^T, \quad \bu^{n-1}_{j+1,l}=\bu_R = [h_R,q_R]^T.
$$
The discontinuous reconstruction is then as follows:

\begin{itemize}
        \item If the solution of the Riemann problem consists of 1-shock and a 2-rarefaction (case 1) then
        $$
        \sigma_j^n =\sigma_1(\bu_l,\bu_*), \quad  \bu^n_{j,l} = \bu^n_{j-1},  \quad \bu^n_{j,r} = \bu_*,
        $$        
        where $\bu_* = [h_*, q_*]^T$ is the intermediate state in the solution of the Riemann problem: $h_*$ is  the root of the function: 
        $$
        f_{s,r}(h) = \left(\frac{h-h_r}{2}+\sqrt{u_r}\right)^2 - u_l + \sqrt{\frac{u_l(h+h_l)}{2h}}(h-h_l),
        $$
        such that $h_l<h_*<h_r$. Once $h_*$ has been computed,  $q_*$ is given by
        $$
        q_* = h_*\left(\frac{h_*-h_r}{2}+\sqrt{u_r}\right)^2.
        $$

        \item If the solution of the Riemann problem consists of  a 1-rarefaction and a 2-shock (case 2), then:
      $$
        \sigma_j^n =\sigma_2(\bu_*,\bu_r), \quad  \bu^n_{j,l} = \bu_*,  \quad \bu^n_{j,r} = \bu^n_{j+1},
        $$        
        where $\bu_* = [h_*, q_*]^T$ is the intermediate state: $h_*$ is the root of the function: 
        $$
        f_{r,s}(h) = \left(\frac{h_l-h}{2}+\sqrt{u_l}\right)\left(\frac{h_l-h}{2}+\sqrt{u_l}+\sqrt{\frac{h_r+h}{2h_r}}(h_r-h)\right)-u_r,
        $$
        such that $h_r<h_*<h_l$. Once $h_*$ has been computed, $q_*$ is given by
        $$
        q_* = h_*\left(\frac{h_l-h_*}{2}+\sqrt{u_l}\right)^2.
        $$

    \item If the solution of the Riemann problem consists of  a 1-shock and a 2-shock (case 3), the intermediate state $\bu_* = [h_*, q_*]^T$ can be computed as follows: $h_*$ is the root of the function 
    $$
    f_{s,s}(h) = u_*(h) + \sqrt{\frac{u_*(h)(h+h_r)}{2h_r}}(h_r-h)-u_r,
    $$
    where
    $$
    u_*(h) = u_l - \sqrt{\frac{u_l(h+h_l)}{2h}}(h-h_l),
    $$
    such that $h_*< h_l$ and $h_*<h_r$. Once  $h_*$ has been computed, $q_*$ is obtained by:
    $$ q_* = h_*u_*(h_*).$$
    Let us denote by $\sigma_1$ and $\sigma_2$ the speeds of the 1 and the 2 shock waves $\sigma_1(\bu_l, \bu_*)$ and $\sigma_2(\bu_*, \bu_r)$.
    The discontinuous reconstruction is then selected as follows:
    \begin{itemize}
        \item If $\sigma_1< 0 < \sigma_2$: let $d_1$ and $d_2$ be given by
        $$
        d_1 = \frac{h_*-h^n_j}{h_*-h_l}, \quad d_2 = \displaystyle \frac{h_r-h^n_j}{h_r-h_*}.
        $$
        Then:
        \begin{itemize}
         \item If  $|\sigma_1| \leq |\sigma_2|$: 
            \begin{itemize}
                    \item If $0\leq d_2 \leq 1$, then 
                     $$
                     \sigma_j^n =\sigma_2, \quad \bu^n_{j,l} =  \bu_*, \quad \bu^n_{j,r} = \bu^n_{j+1}.
                     $$
                     \item Otherwise, if $0\leq d_1 \leq 1$, then 
                     $$
                     \sigma_j^n =\sigma_1, \quad  \bu^n_{j,l} =  \bu^n_{j-1},  \quad \bu^n_{j,r} = \bu_*.
                     $$
                \end{itemize}
         \item If   $|\sigma_1| > |\sigma_2|$: 
            \begin{itemize}
                    \item If $0\leq d_1 \leq 1$, then 
                     $$
                     \sigma_j^n =\sigma_1, \quad  \bu^n_{j,l} =  \bu^n_{j-1},  \quad \bu^n_{j,r} = \bu_*.
                     $$
                     \item Otherwise, if  $0\leq d_2 \leq 1$, then
                     $$
                     \sigma_j^n =\sigma_2, \quad \bu^n_{j,l} =  \bu_*, \quad \bu^n_{j,r} = \bu^n_{j+1}.
                     $$

                \end{itemize}

        \end{itemize}
    \item Otherwise (i.e. if $0 \leq  \sigma_1 < \sigma_2$  or $\sigma_1 < \sigma_2 \leq 0$): let  $d_1$ and $d_2$ be such that
                \begin{equation}\label{conservation_2}
                \begin{cases}
                 d_1 h_{l} + (d_2 -  d_1) h_{*} + (1-d_2) h_r = h_{j}^n, \\
                 d_1 q_{l} + (d_2 -  d_1) q_{*} + (1-d_2) q_r = q_{j}^n.
                \end{cases}
                \end{equation}
                Then:
                \begin{equation}\label{Reconstruction_operator_2}
                    P_{j}^{n}(x,t) = 
                \begin{cases}
                \bu_l & \text{ if $x \leq x_{j - 1/2} + d_1 \Delta x + \sigma_1 (t -t_n) $,}\\
                \bu_* &  \text{ if $x_{j - 1/2} + d_1 \Delta x + \sigma_1(t -t_n) \leq x \leq x_{j - 1/2} + d_2 \Delta x + \sigma_2 (t -t_n)$,}\\
                \bu_r & \text{otherwise.}
                \end{cases}
                \end{equation}
                 This in-cell discontinuous reconstruction can only be done if $0 \leq d_1, d_2 \leq 1$, otherwise the cell is unmarked. Moreover, if $d_1 = d_2 = 1$ and the speeds of the shocks are positive (resp. if $d_1 = d_2 =0$ and the speeds of the shocks are negative) the  cell is unmarked and the cell $I_{j+1}$ (resp. the cell $I_{j-1}$) is marked if necessary.

\end{itemize}

\end{itemize}

Observe that, when the speeds of the shocks have the same sign, the discontinuous reconstruction coincides with the solution of the Riemann problem. 

The numerical methods using the first strategy for the discontinuous reconstruction (based on the Roe matrix) will be labeled again by
 O$p$\_DisRec  and those using the second one (based on the exact solutions of the Riemann problems) by 
 O$p$\_ExactDisRec.

\subsubsection*{Test 1: Isolated 1-shock}

Let us consider the following initial condition taken from \cite{Castro2008}

$$
(h,q)_{0}(x)= \begin{cases}
     (1, 1) & \text{if $x<0$,}  \\
     (1.8,0.530039370688997) & \text{otherwise.}
\end{cases} 
$$
The solution of the Riemann problem consists of a 1-shock wave joining the left and right states. Figure \ref{fig:1DSimplifiedSW_Test1_ko1_vs_ko2_Dis_vs_NoDis} compares the exact solution and the numerical approximations at time $t = 0.15$ obtained with Roe method, its second order extension based on the standard MUSCL-Hancock reconstruction, and the first and second order discontinuous in-cell reconstruction schemes based on the Roe matrix using 1000-cell mesh and CFL = 0.5: as it can be seen the standard Roe methods does not capture the discontinuities properly what is not the case for the in-cell discontinuous reconstruction methods based on the Roe structure. The results obtained with O$p$\_ExactDisRec are similar. 

\begin{figure}[h]
		\begin{subfigure}{0.5\textwidth}
			\includegraphics[width=1.1\linewidth]{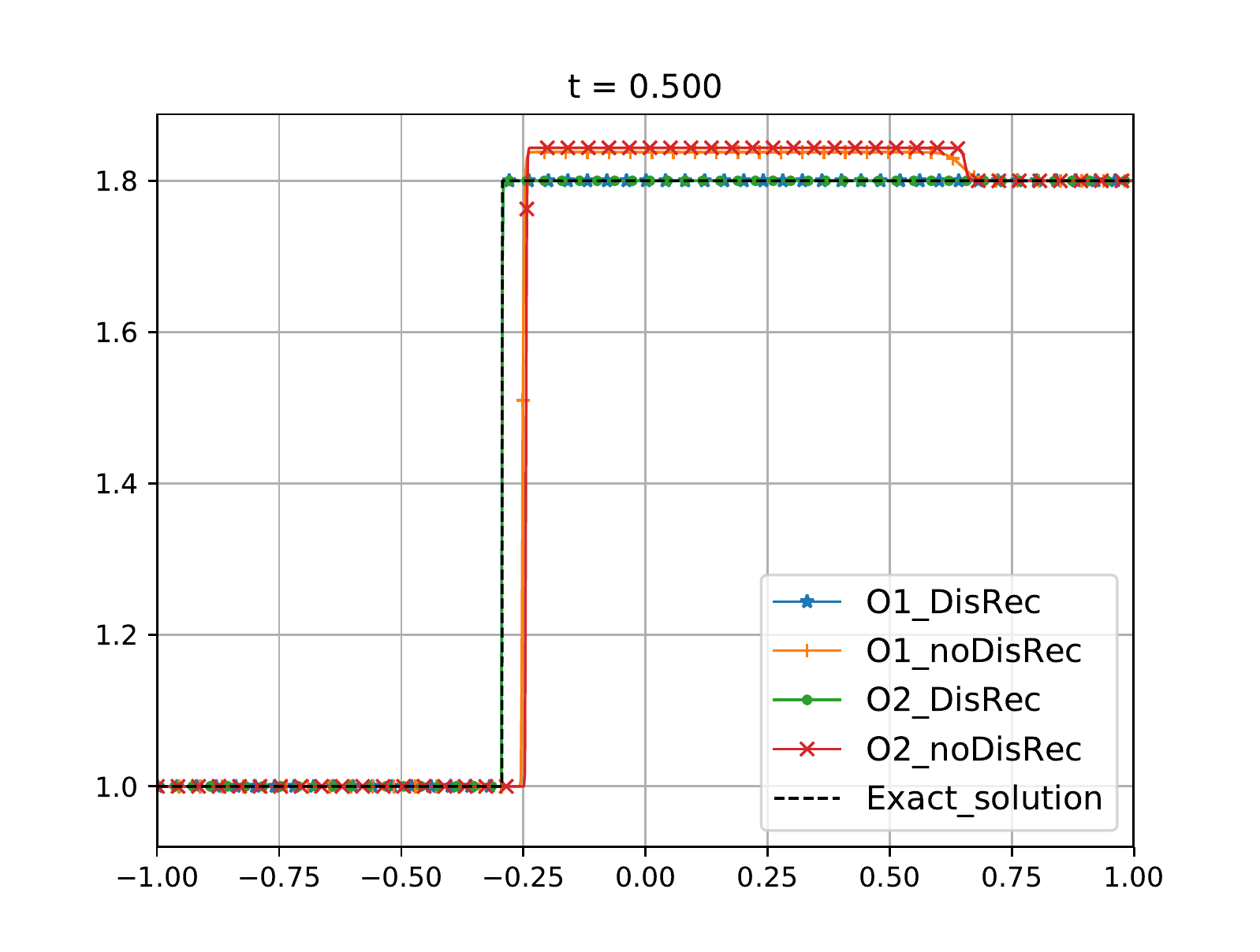}
			\caption{Variable $h$}
		\end{subfigure}
		\begin{subfigure}{0.5\textwidth}
				\includegraphics[width=1.1\linewidth]{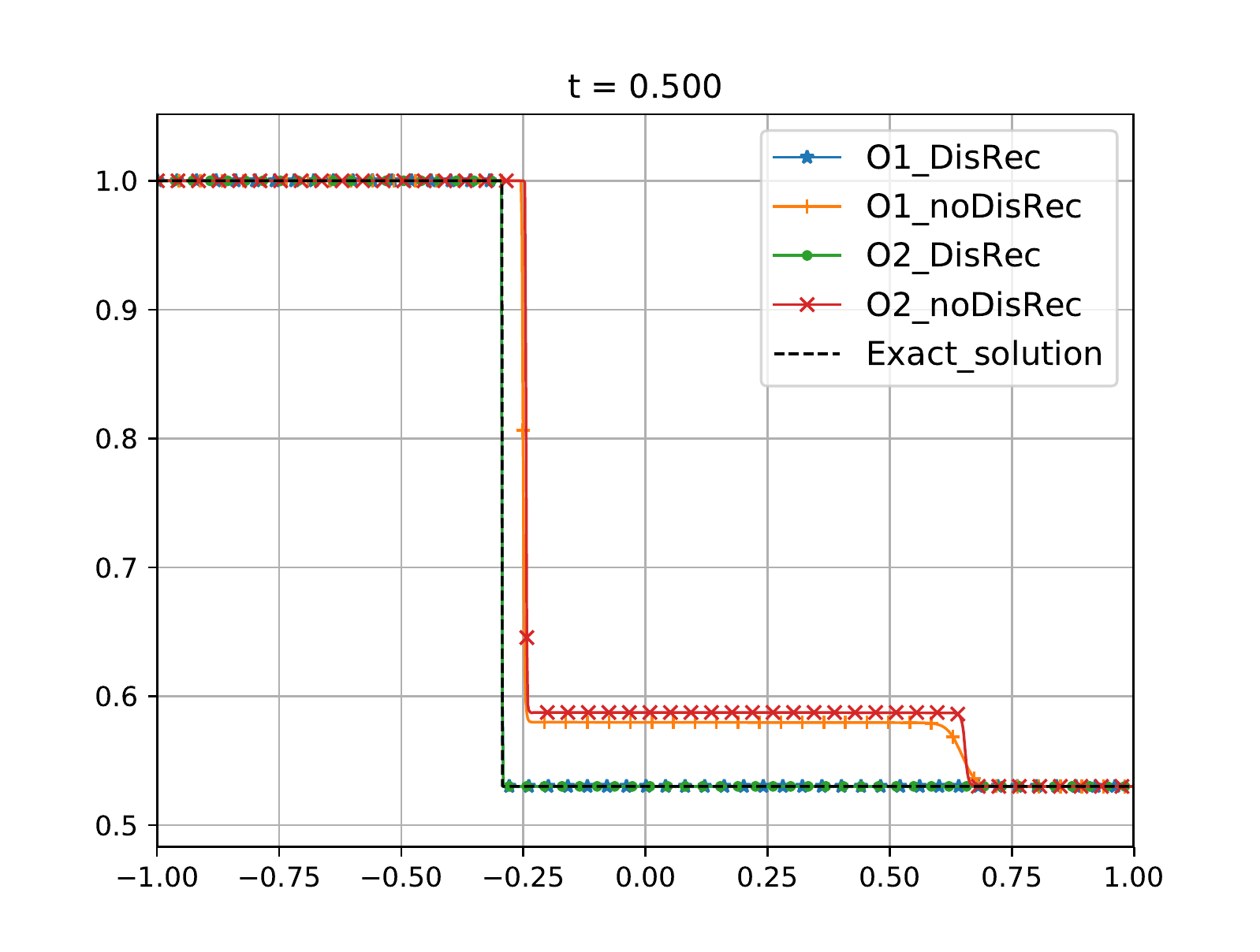}
				\caption{Variable $q$}
		\end{subfigure}
		\caption{Modified Shallow Water system. Test 1: Numerical solutions obtained withe the  first and second-order methods with and without discontinuous reconstruction based on the Roe matrix at time $t=0.5$ with 1000 cells.  Left: variable $h$. Right: variable $q$}
		\label{fig:1DSimplifiedSW_Test1_ko1_vs_ko2_Dis_vs_NoDis}
	\end{figure}

\subsubsection*{Test 2: left-moving 1-shock + right-moving 2-shock}

Let us consider the following initial condition
\begin{equation}\label{eq:1DS_Test2}
    (h, q)_{0}(x)= \begin{cases}
     (1, 1) & \text{if $x<0$,}  \\
     (1.5,0.1855893974385) & \text{otherwise.}
\end{cases} 
\end{equation}
The solution of the Riemann problem consists of a 1-shock wave with negative speed and a 2-shock with positive speed  with intermediate state $\bu_* = [1.8, 0.530039370688997]^T$. Figures \ref{fig:1DSimplifiedSW_Test2_ko1_vs_ko2_Dis_vs_NoDis_h}
and \ref{fig:1DSimplifiedSW_Test2_ko1_vs_ko2_Dis_vs_NoDis_q}  compare the exact solution with the numerical approximations at time $t = 0.15$ obtained with Roe method, its second order extension based on the standard MUSCL-Hancock reconstruction, and the first and second order discontinuous in-cell reconstruction schemes based on the Roe matrix using 1000-cell mesh and CFL = 0.5: as it can be seen none of them capture the discontinuities properly, although the ones with using in-cell discontinuous reconstruction do it better. Figures \ref{fig:1DSimplifiedSW_Test2_ko1_convergence_comparison_h} and \ref{fig:1DSimplifiedSW_Test2_ko1_convergence_comparison_q} show the numerical solutions obtained with the first-order method with discontinuous reconstruction based on the Roe matrix at time $t=0.15$ using different cell meshes: as we can see the numerical solutions seem to converge to the exact solution when $\Delta x \rightarrow 0$. In Figure \ref{fig:1DSimplifiedSW_Test2_ko1_vs_ko2_ExactDis} the results given by the first and second order in-cell discontinuous schemes based in the exact solution of the Riemann problem are shown: we observe that both of them capture exactly the two shocks.

%Table \ref{tab:1DS_Error_Test2} and  Figure \ref{fig:1DS_Error_Test2_Plot}  show the $L^1$-errors corresponding to the first and second order in-cell discontinuous reconstruction schemes with different meshes: although the errors decrease as the mesh is refined, they seem not to converge to 0

\begin{figure}[h]
		\begin{subfigure}{0.5\textwidth}
			\includegraphics[width=1.1\linewidth]{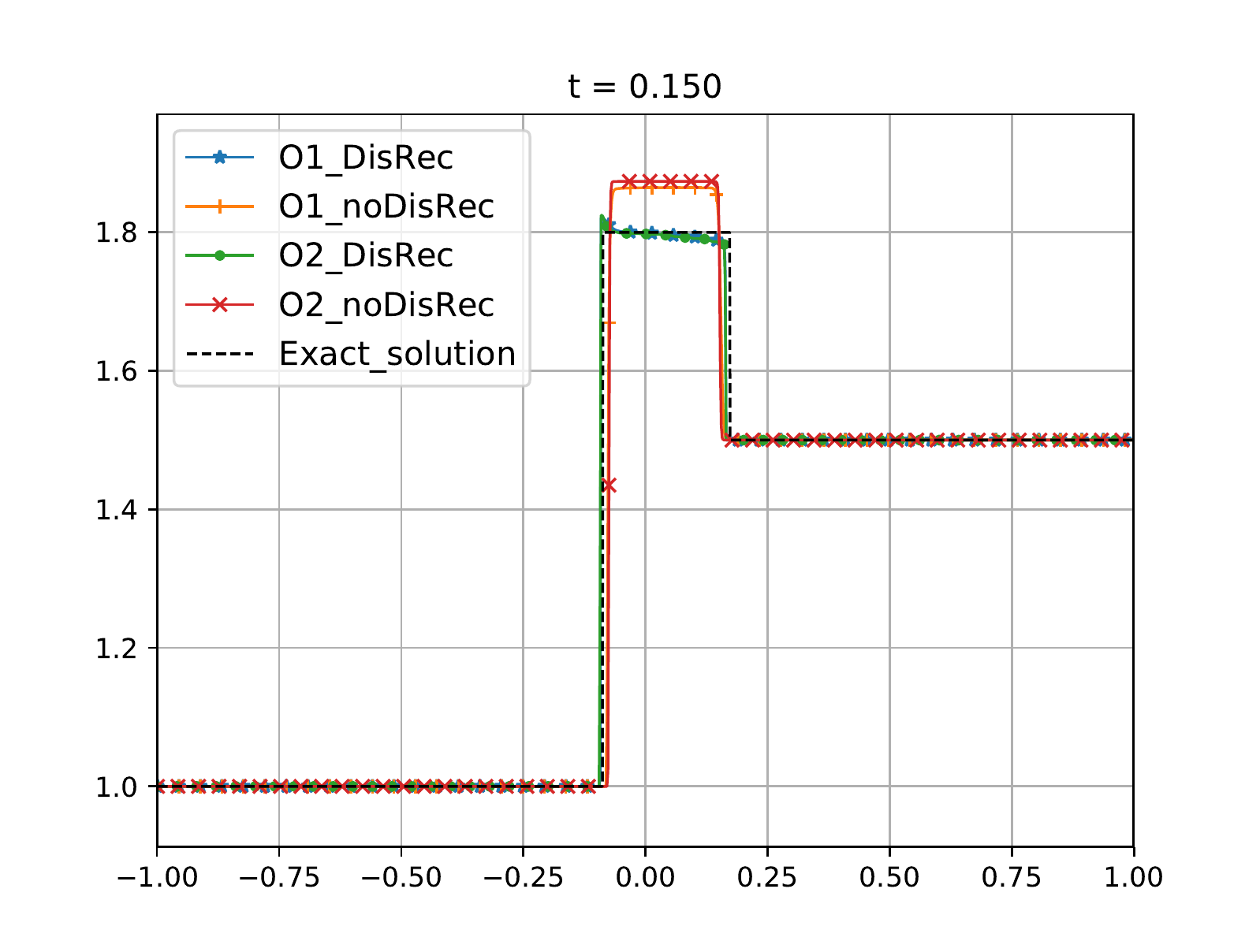}
		\end{subfigure}
		\begin{subfigure}{0.5\textwidth}
				\includegraphics[width=1.1\linewidth]{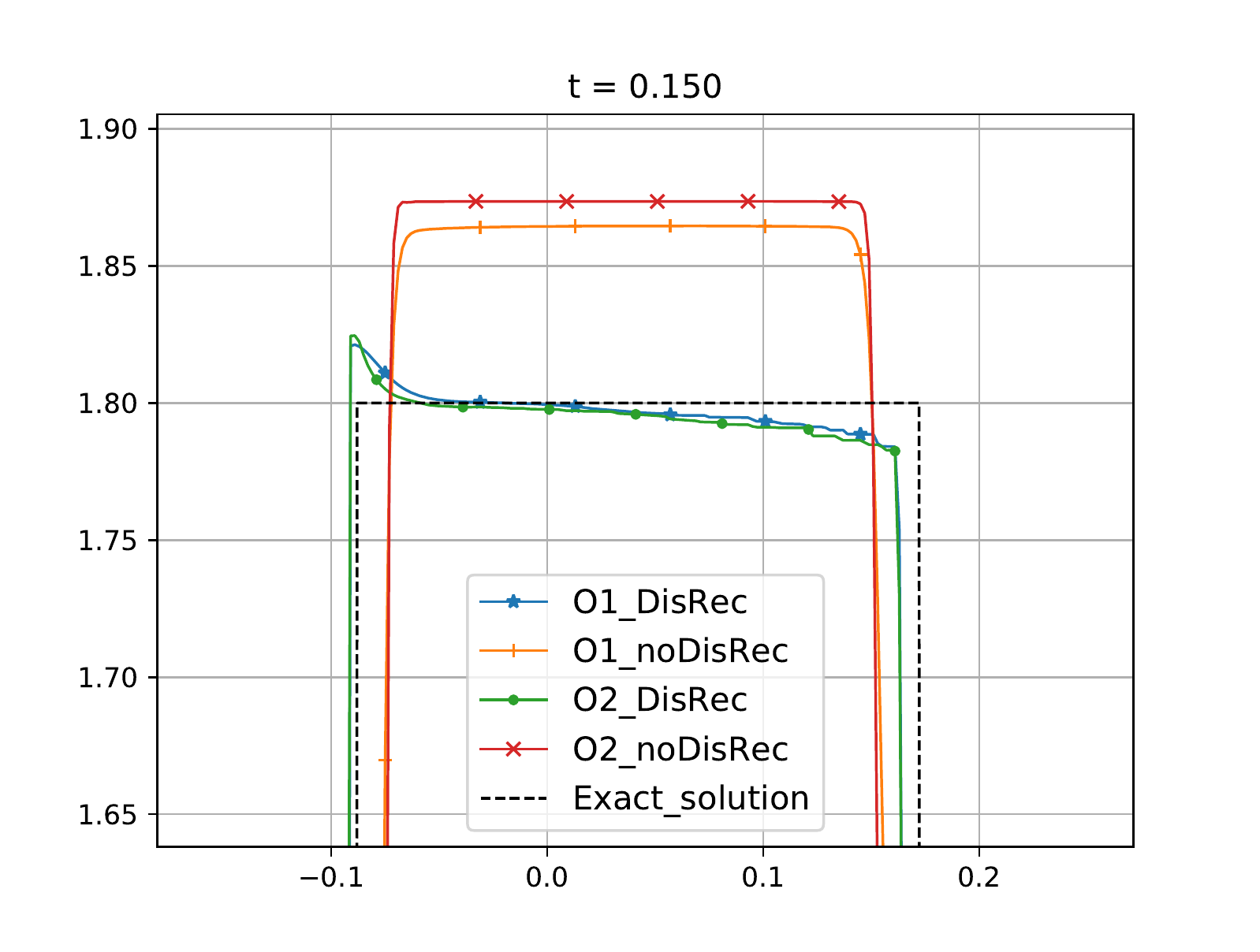}
				\caption{Zoom}
		\end{subfigure}
		\caption{Modified Shallow Water system. Test 2: variable $h$. Left: Numerical solutions obtained with the  first and second-order methods with and without discontinuous reconstruction based on the Roe matrix at time $t=0.15$ with 1000 cells. Right: zoom}
		\label{fig:1DSimplifiedSW_Test2_ko1_vs_ko2_Dis_vs_NoDis_h}
	\end{figure}
	
\begin{figure}[h]
		\begin{subfigure}{0.5\textwidth}
			\includegraphics[width=1.1\linewidth]{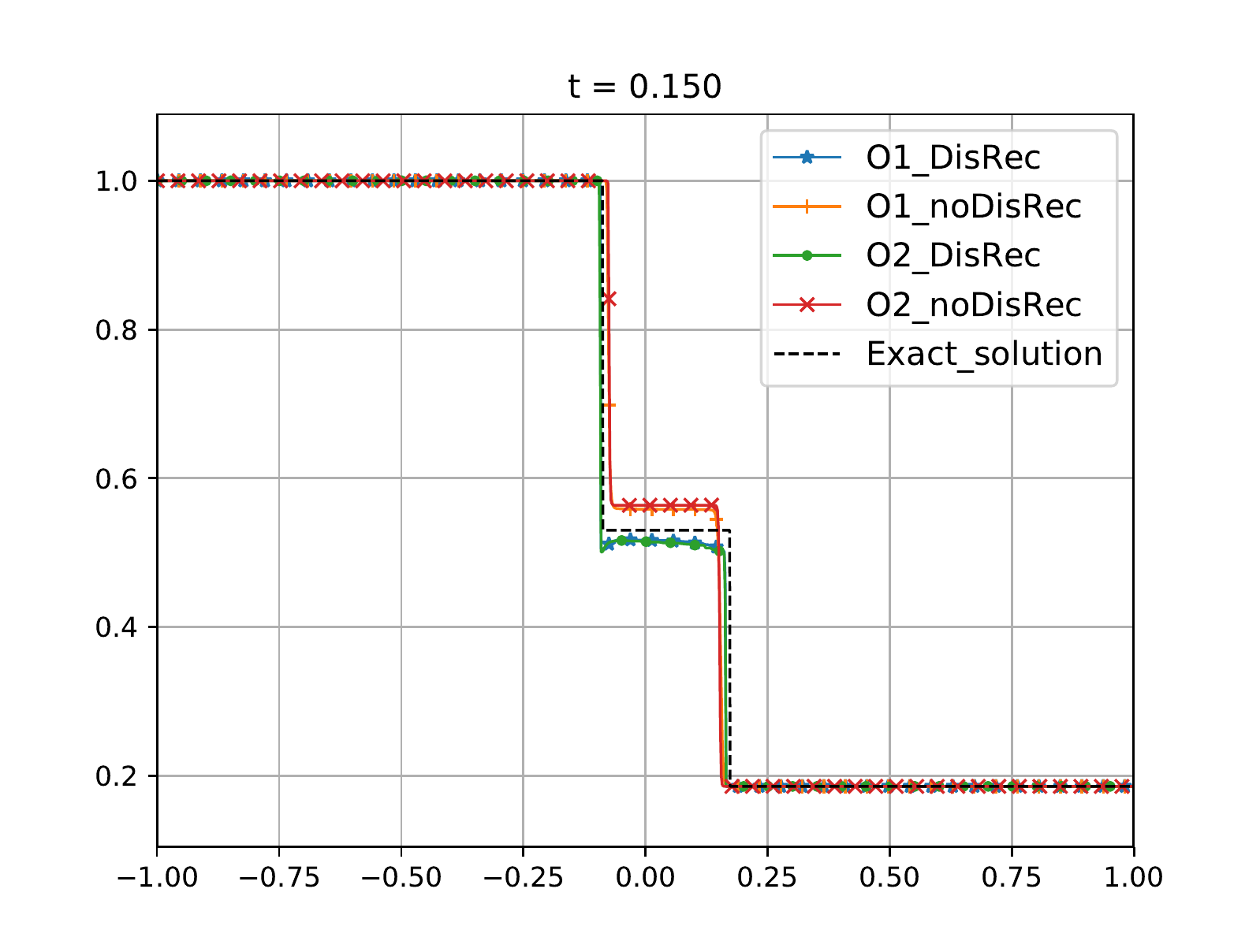}
		\end{subfigure}
		\begin{subfigure}{0.5\textwidth}
				\includegraphics[width=1.1\linewidth]{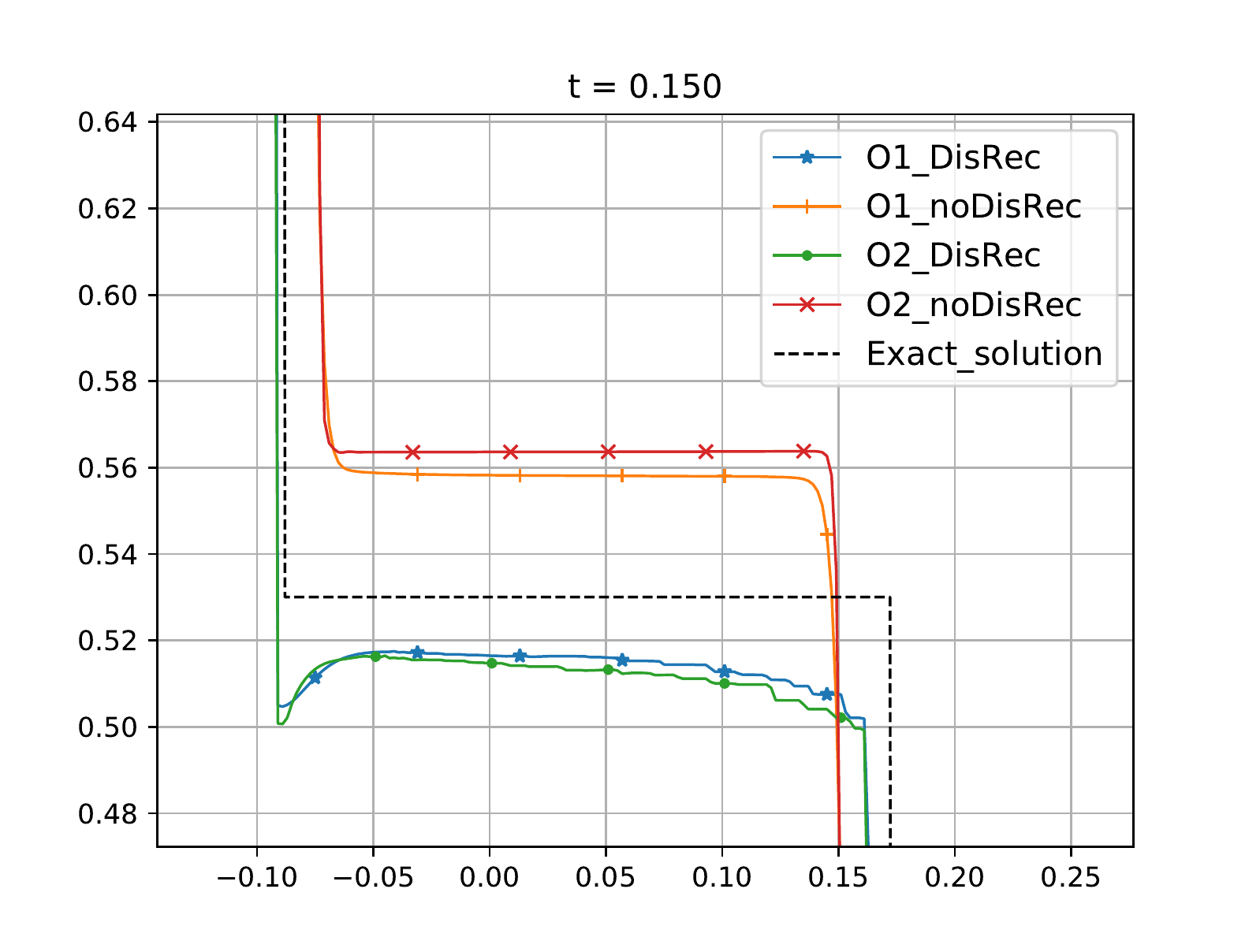}
				\caption{Zoom}
		\end{subfigure}
		\caption{Modified Shallow Water system. Test 2: variable $q$. Left: Numerical solutions obtained with the  first and second-order methods with and without discontinuous reconstruction based on the Roe matrix at time $t=0.15$ with 1000 cells. Right: zoom}
		\label{fig:1DSimplifiedSW_Test2_ko1_vs_ko2_Dis_vs_NoDis_q}
	\end{figure}
	
\begin{comment}
\begin{table}[ht]
  	\centering
  	\begin{tabular}{|c|c|c|c|c|}
  		\hline 
  		                &\multicolumn{2}{|c|} {O1\_DisRec}& \multicolumn{2}{|c|}{O2\_DisRec}\\
  		Number of cells & $h$ & $q$  & $h$  & $q$ \\   
  		\hline 
  	    1000 & 3.55e-3 & 4.48e-3 & 3.97e-3 & 4.95e-3 \\
  		\hline 
  		2000 & 2.56e-3 & 3.19e-3 & 2.84e-3 & 3.53e-3 \\
  		\hline 
  		4000 & 1.81e-3 & 2.20e-3 & 2.02e-3 & 2.38e-3 \\
  		\hline 
  		8000 & 1.26e-3 & 1.51e-3 & 1.42e-3 & 1.63e-3 \\
  		\hline 
  		16000 & 8.75e-4 & 1.04e-3 & 9.80e-4 & 1.14e-3 \\
  		\hline 
  		32000 & 5.90e-4 & 6.95e-4 & 6.51e-4 & 7.60e-4 \\
  		\hline 
  	\end{tabular} 
	  	\caption{Modified Shallow Water system. Test 2: $L^{1}$ errors  at time $t=0.15$ for O$p$\_DisRec, $p = 1,2$.}
  	\label{tab:1DS_Error_Test2}
\end{table} 

\begin{figure}[h]
			\includegraphics[width=1.\linewidth]{Figures/1DSimplifiedSW_Test2_Error_convergence_logarithm_scale.pdf}
		\caption{Modified Shallow Water system. Test 2: $L^{1}$ errors  at time $t=0.15$ for O$p$\_DisRec, $p = 1,2$. It has been used the logarithm scale in the $x$ axis.}
		\label{fig:1DS_Error_Test2_Plot}
	\end{figure}
\end{comment}
	
\begin{figure}[h]
		\begin{subfigure}{0.5\textwidth}
			\includegraphics[width=1.1\linewidth]{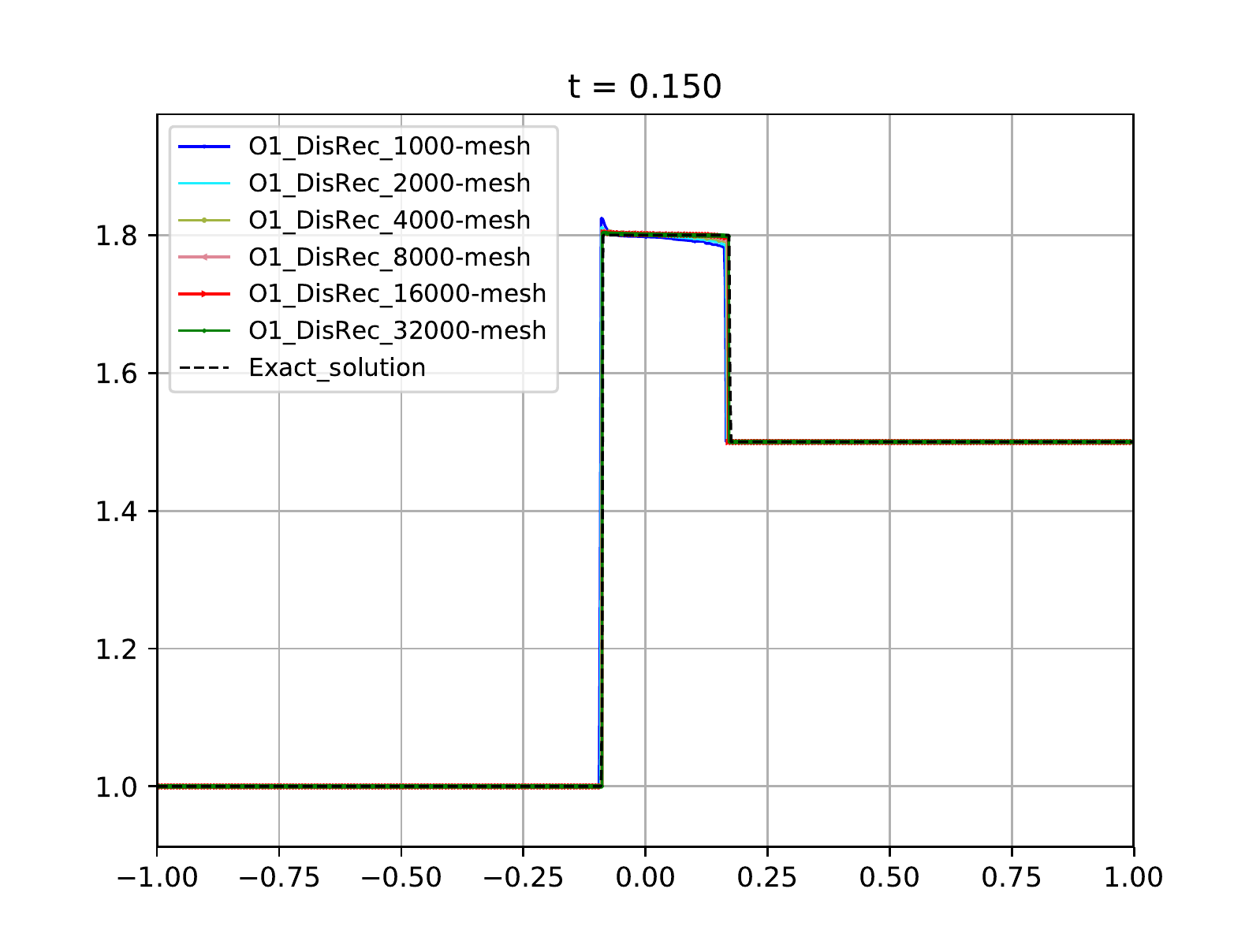}
		\end{subfigure}
		\begin{subfigure}{0.5\textwidth}
				\includegraphics[width=1.1\linewidth]{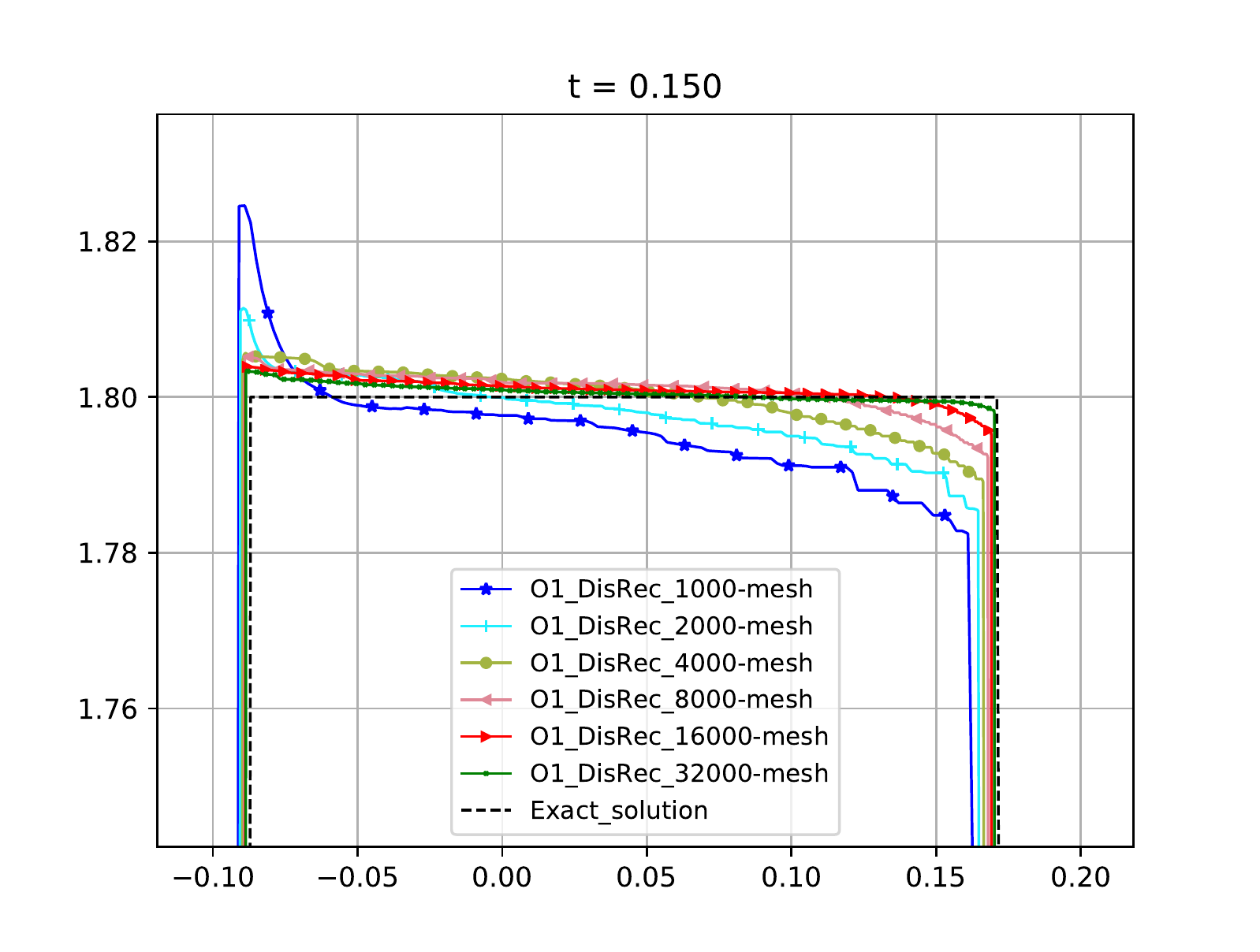}
				\caption{Zoom}
		\end{subfigure}
		\caption{Modified Shallow Water system. Test 2: variable $h$. Left: Numerical solutions obtained with the first-order methods with discontinuous reconstruction based on the Roe matrix at time $t=0.15$ with different cell meshes. Right: zoom}
		\label{fig:1DSimplifiedSW_Test2_ko1_convergence_comparison_h}
	\end{figure}
	
\begin{figure}[h]
		\begin{subfigure}{0.5\textwidth}
			\includegraphics[width=1.1\linewidth]{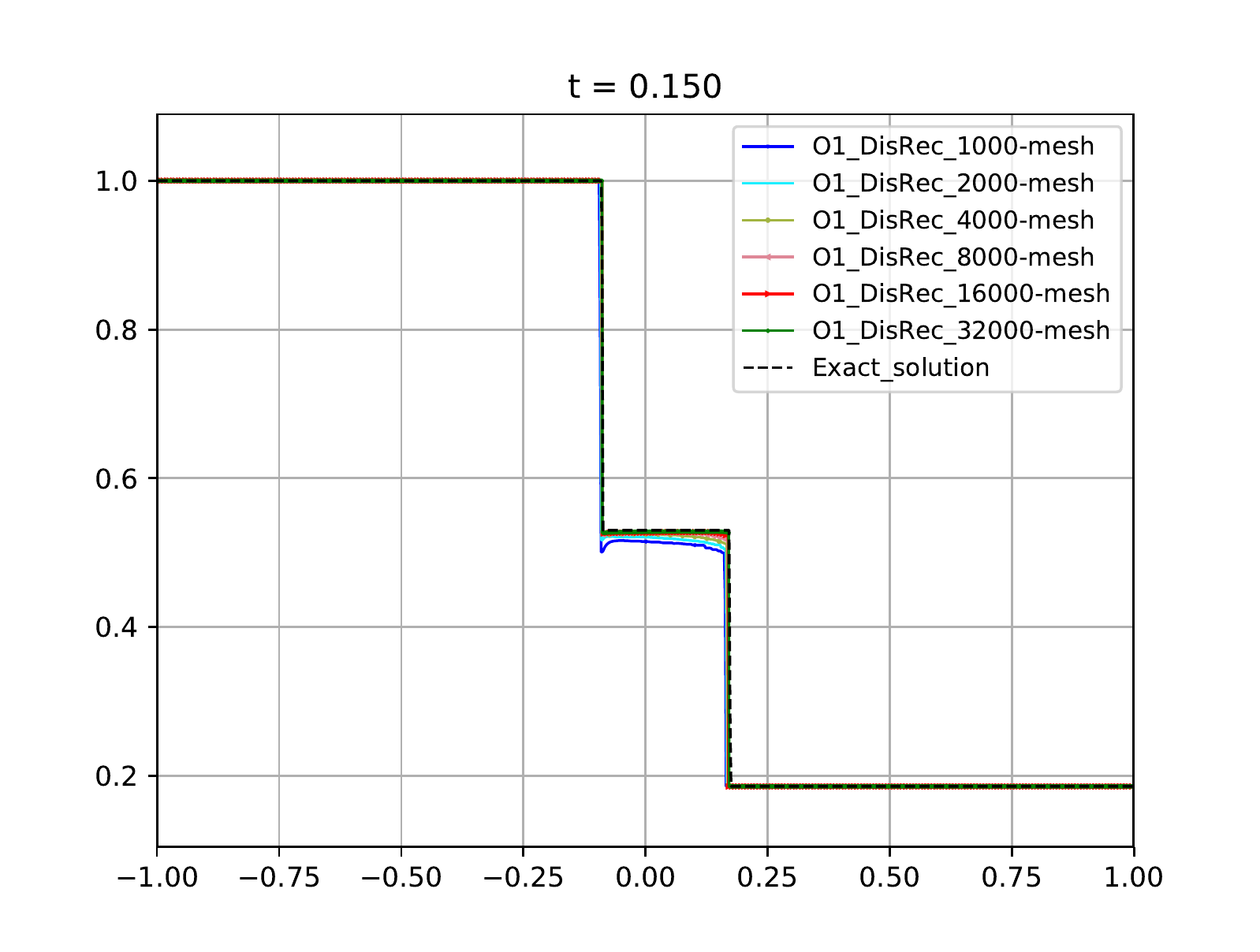}
		\end{subfigure}
		\begin{subfigure}{0.5\textwidth}
				\includegraphics[width=1.1\linewidth]{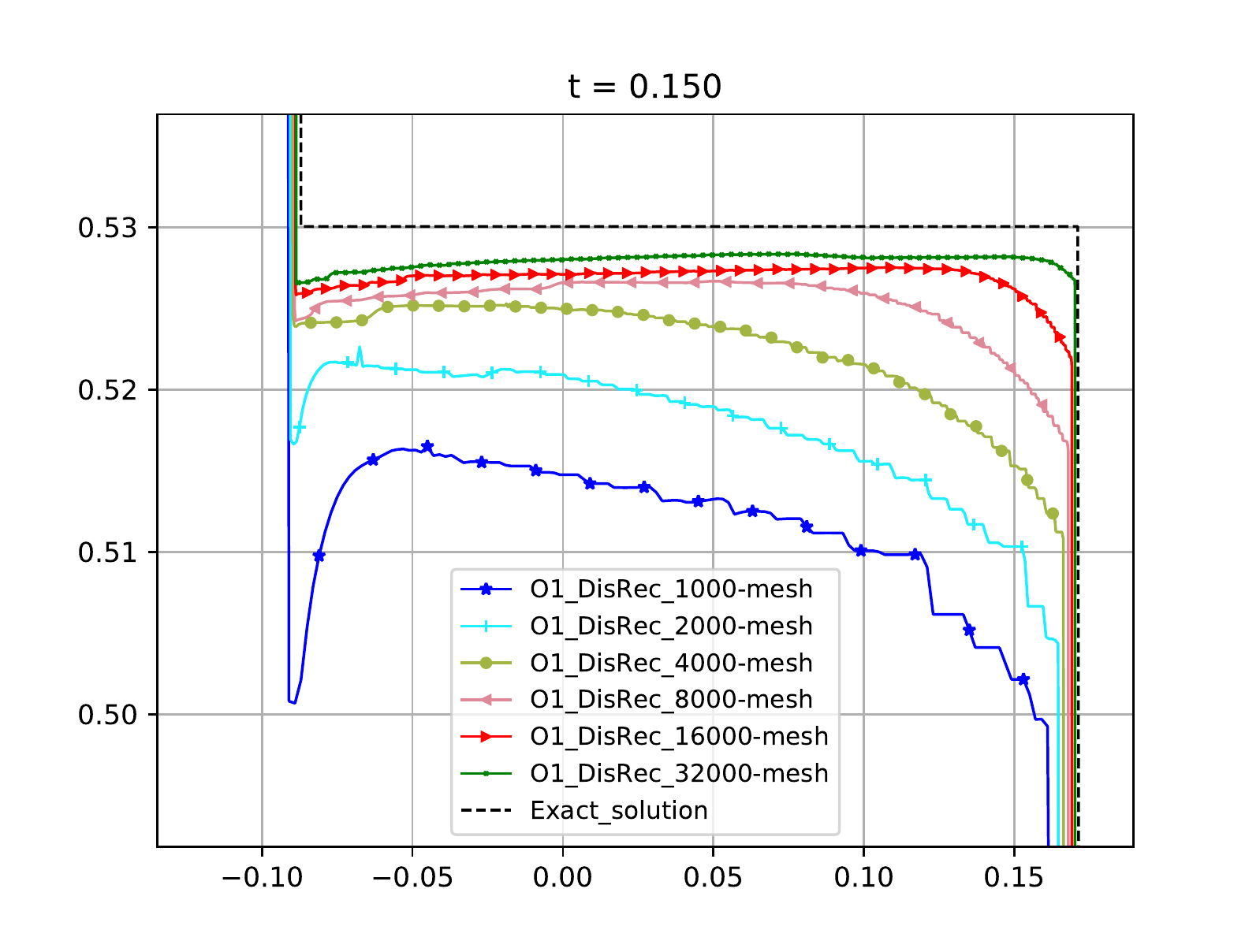}
				\caption{Zoom}
		\end{subfigure}
		\caption{Modified Shallow Water system. Test 2: variable $q$. Left: Numerical solutions obtained with the first-order methods with discontinuous reconstruction based on the Roe matrix at time $t=0.15$ with different cell meshes. Right: zoom}
		\label{fig:1DSimplifiedSW_Test2_ko1_convergence_comparison_q}
	\end{figure}
	
\begin{figure}[h]
		\begin{subfigure}{0.5\textwidth}
			\includegraphics[width=1.1\linewidth]{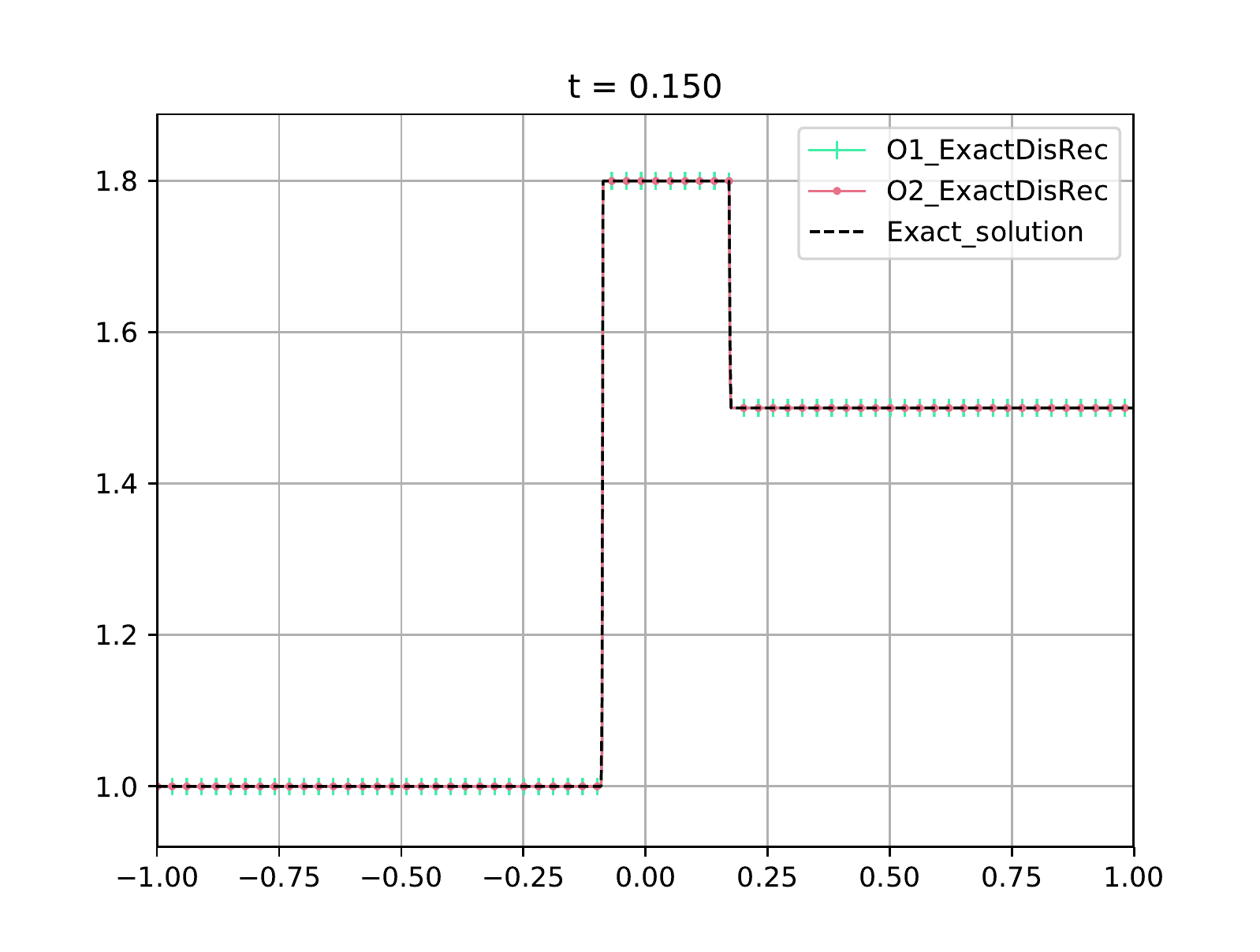}
			\caption{Variable $h$}
		\end{subfigure}
		\begin{subfigure}{0.5\textwidth}
				\includegraphics[width=1.1\linewidth]{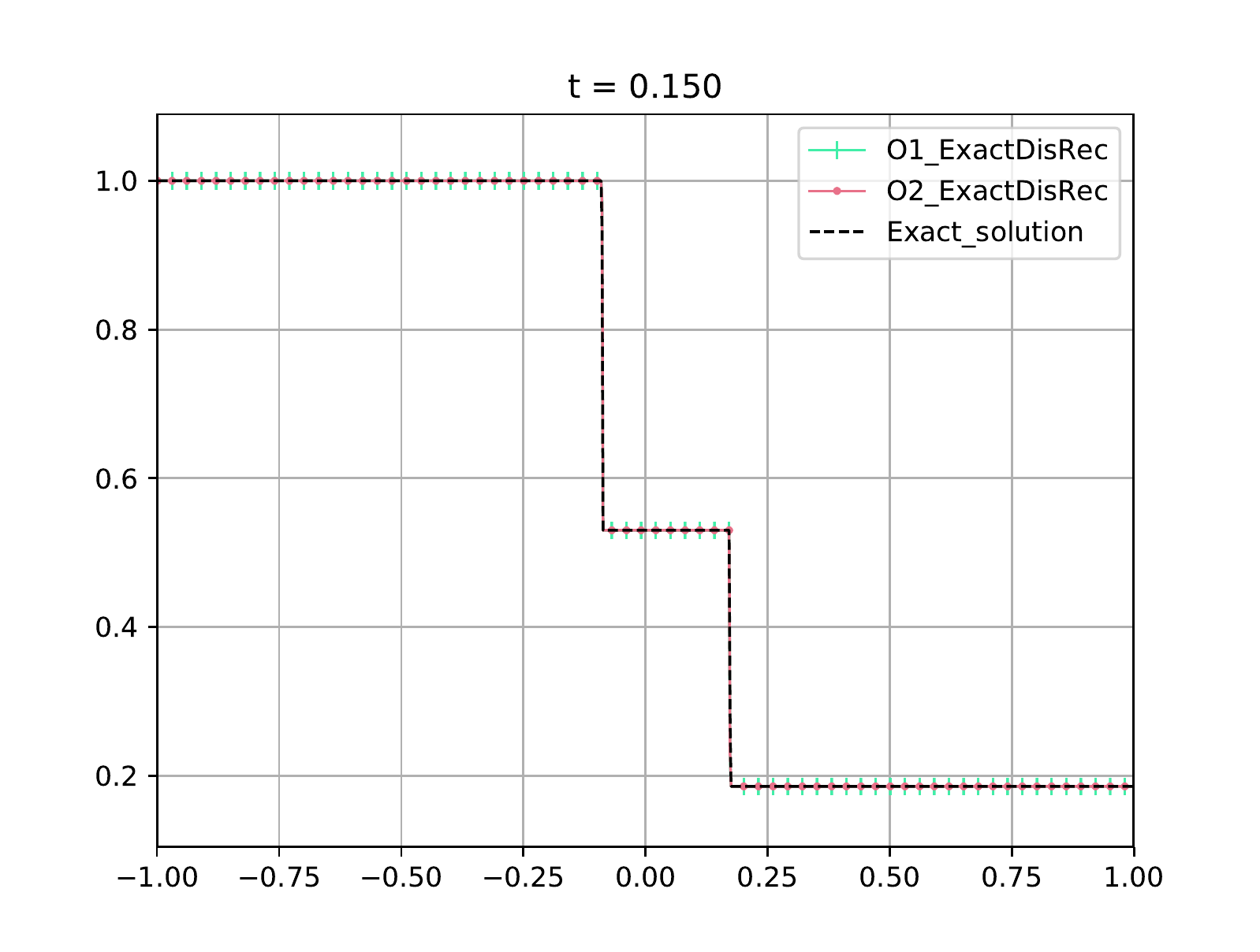}
				\caption{Variable $q$}
		\end{subfigure}
		\caption{Modified Shallow Water system. Test 2: Numerical solutions obtained with the  first and second-order methods with discontinuous reconstruction based on the exact solutions of the Riemann problems at time $t=0.15$ with 1000 cells. Left : variable $h$. Right: variable $q$.}
		\label{fig:1DSimplifiedSW_Test2_ko1_vs_ko2_ExactDis}
	\end{figure}

\subsubsection*{Test 3: right-moving 1-shock + right-moving 2-shock}\label{ssstest3}

Let us consider the following initial condition

\begin{equation}\label{eq:1DS_Test3}
    (h,q)_{0}(x)= \begin{cases}
     (1, 1) & \text{if $x<0$,}  \\
     (5,2.86423084288) & \text{otherwise.}
\end{cases} 
\end{equation}
The solution of the Riemann problem consists of a 1-shock and a 2-shock waves with positive speed and intermediate state $\bu_* = [1.5, 5.96906891076]^T$. Figures \ref{fig:1DSimplifiedSW_Test3_ko1_vs_ko2_Dis_vs_NoDis_h} and \ref{fig:1DSimplifiedSW_Test3_ko1_vs_ko2_Dis_vs_NoDis_q} show the exact solution and the numerical approximations at time $t = 0.06$ obtained with Roe method, its second order extension based on the standard MUSCL-Hancock reconstruction, and the first and second order discontinuous in-cell reconstruction schemes based on the Roe structure using 1000-cell mesh and CFL = 0.5: as in the previous test case, the in-cell discontinuous reconstruction captures the shocks and intermediate state much better than the standard first and second order Roe methods. In Figure \ref{fig:1DSimplifiedSW_Test3_ko1_vs_ko2_ExactDis} the results given by the first and second order in-cell discontinuous schemes based in the exact solution of the Riemann problems are shown:  both of them capture exactly the  exact solution.

\begin{figure}[h]
		\begin{subfigure}{0.5\textwidth}
			\includegraphics[width=1.1\linewidth]{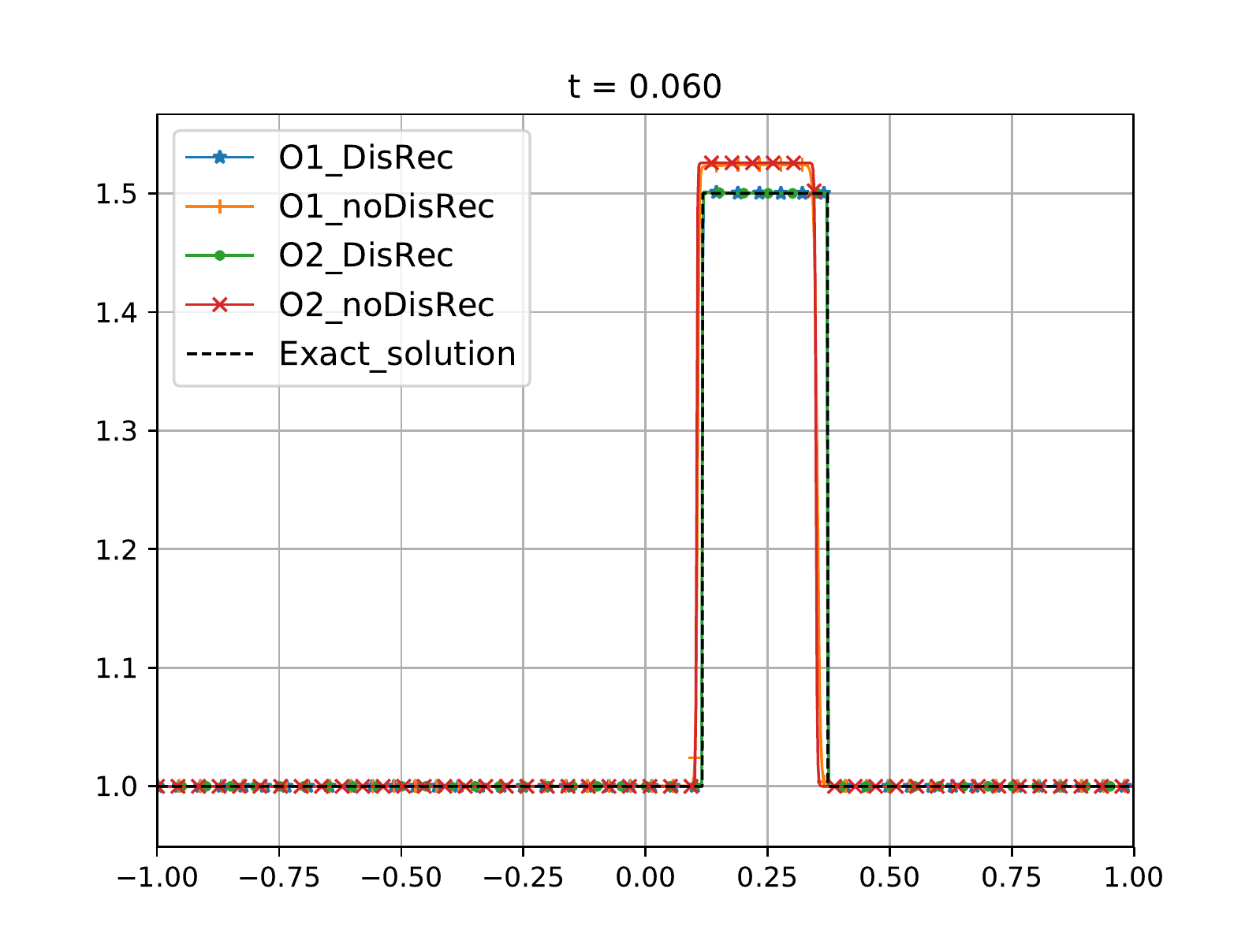}
		\end{subfigure}
		\begin{subfigure}{0.5\textwidth}
				\includegraphics[width=1.1\linewidth]{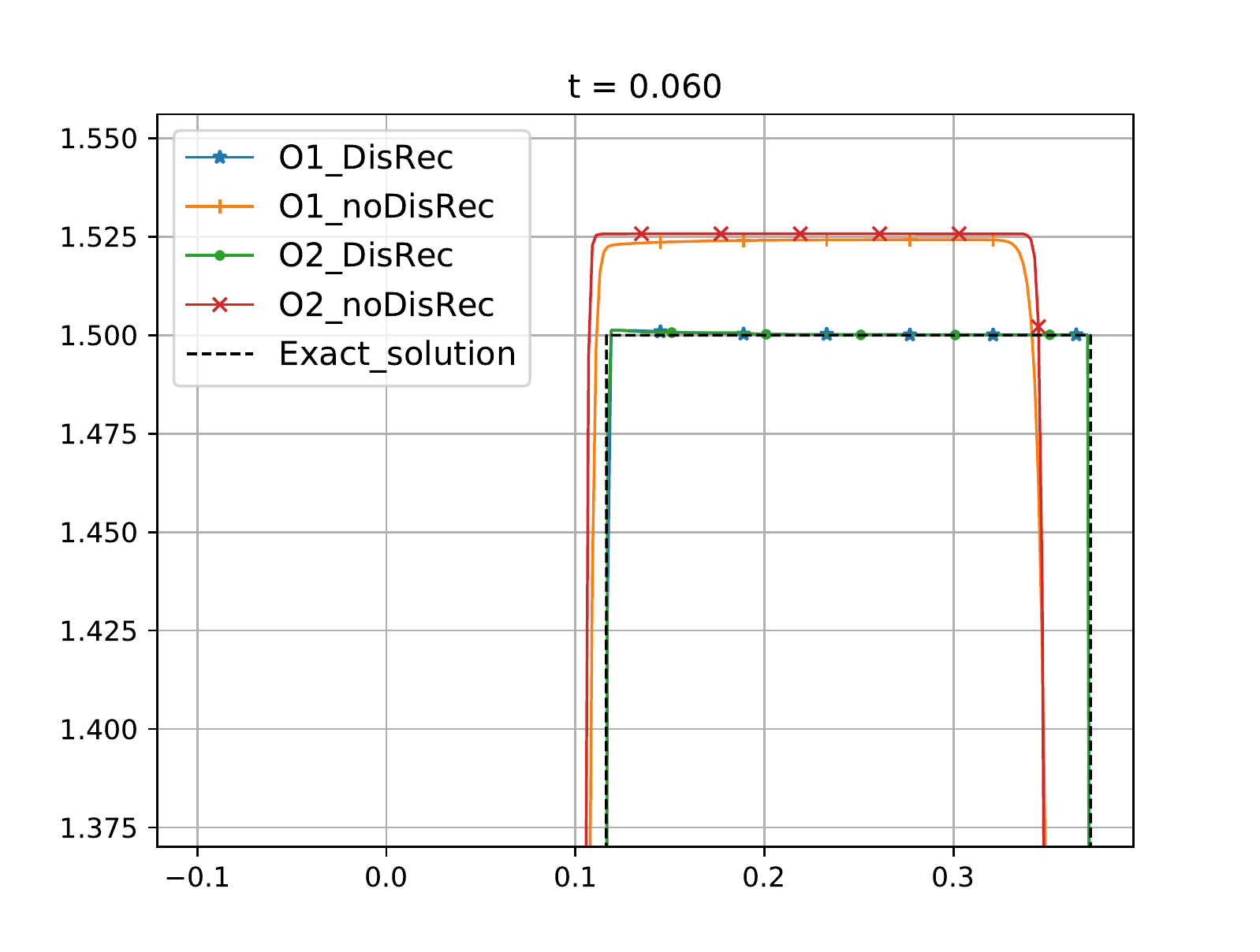}
				\caption{Zoom}
		\end{subfigure}
		\caption{Modified Shallow Water system. Test 3: variable $h$. Left: Numerical solutions obtained with the  first and second-order methods with and without discontinuous reconstruction based on the Roe matrix at time $t=0.06$ with 1000 cells. Right: zoom}
		\label{fig:1DSimplifiedSW_Test3_ko1_vs_ko2_Dis_vs_NoDis_h}
	\end{figure}
	
\begin{figure}[h]
		\begin{subfigure}{0.5\textwidth}
			\includegraphics[width=1.1\linewidth]{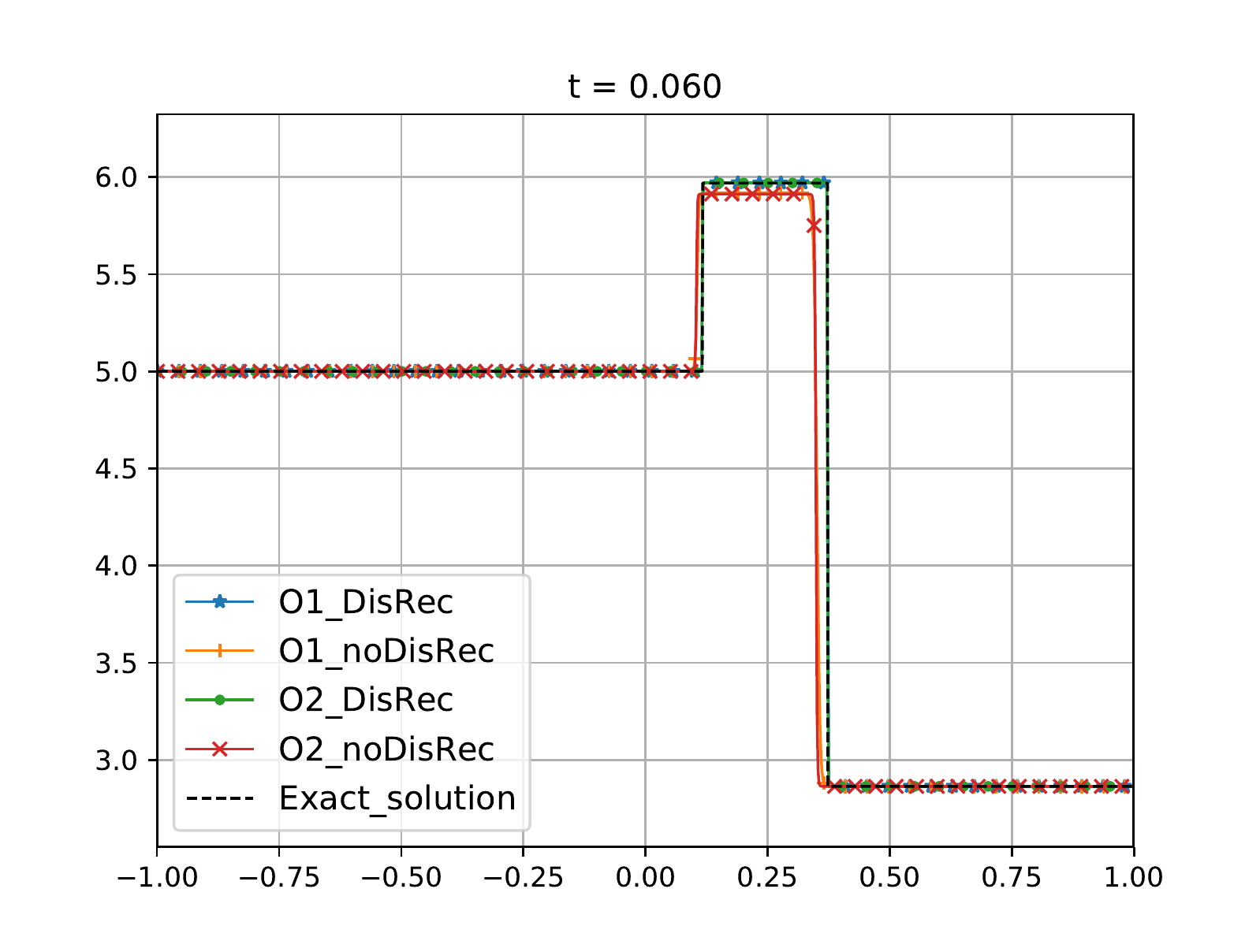}
		\end{subfigure}
		\begin{subfigure}{0.5\textwidth}
				\includegraphics[width=1.1\linewidth]{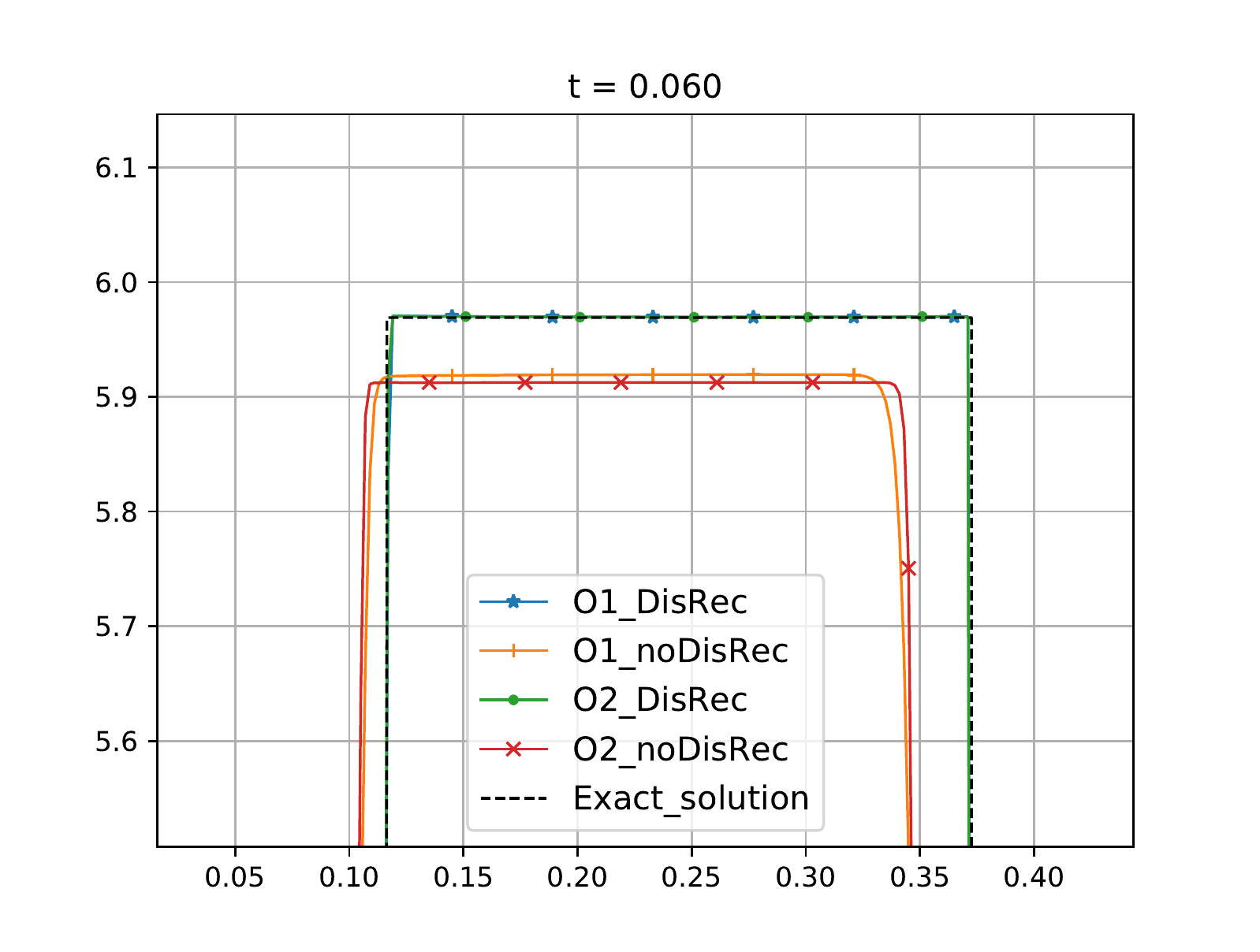}
				\caption{Zoom}
		\end{subfigure}
		\caption{Modified Shallow Water system. Test 3: variable $q$. Left: Numerical solutions obtained with the  first and second-order methods with and without discontinuous reconstruction based on the Roe matrix at time $t=0.06$ with 1000 cells. Right: zoom}
		\label{fig:1DSimplifiedSW_Test3_ko1_vs_ko2_Dis_vs_NoDis_q}
	\end{figure}
	
\begin{figure}[h]
		\begin{subfigure}{0.5\textwidth}
			\includegraphics[width=1.1\linewidth]{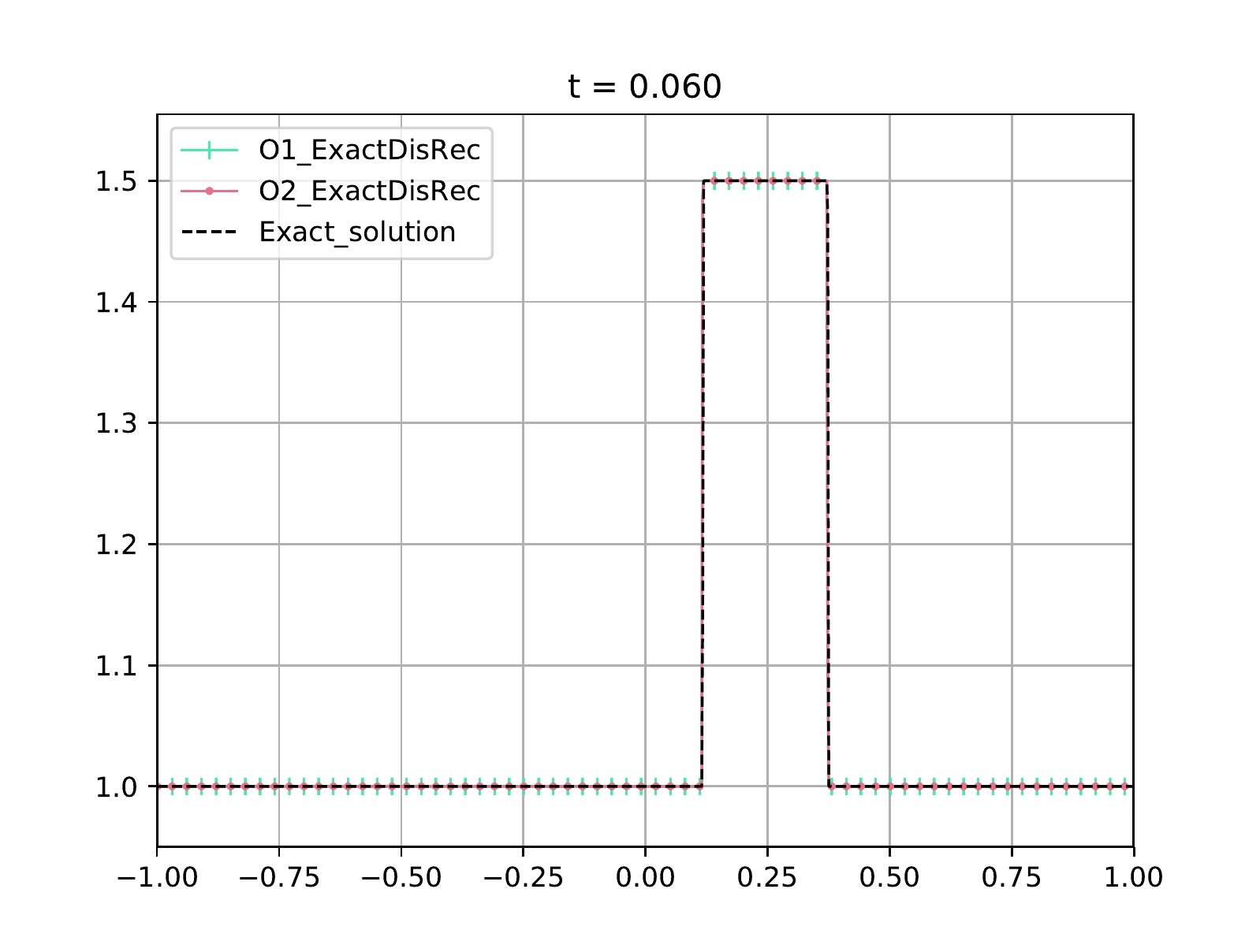}
			\caption{Variable $h$}
		\end{subfigure}
		\begin{subfigure}{0.5\textwidth}
				\includegraphics[width=1.1\linewidth]{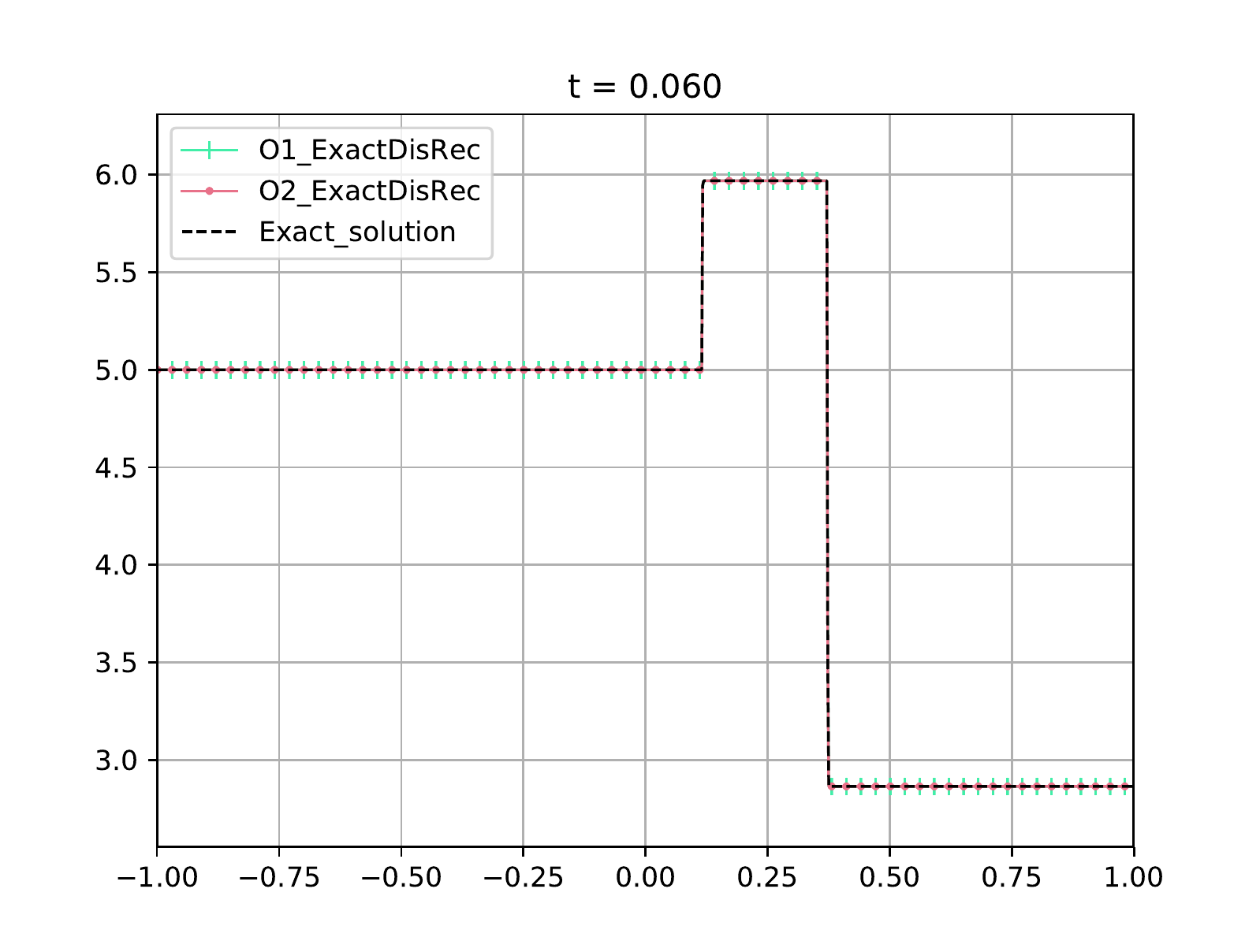}
				\caption{Variable $q$}
		\end{subfigure}
		\caption{Modified Shallow Water system. Test 3: Numerical solutions obtained with the  first and second-order methods with discontinuous reconstruction based on the exact solutions of the Riemann problems at time $t=0.06$ with 1000 cells.  Left: variable $h$. Right: variable $q$.}
		\label{fig:1DSimplifiedSW_Test3_ko1_vs_ko2_ExactDis}
	\end{figure}

\section{Conclusions}
In this paper, an extension to second-order accuracy of the in-cell discontinuous reconstruction methods introduced in \cite{chalons2019path} is presented: it has been compared with the first-order one using several numerical tests. We observe, as expected, an improvement in the smooth parts of the solutions. The isolated shock-capturing property is enunciated, proved and tested. In the presence of more than one shock we have used two different strategies: one based on the linearized Riemann problem when a Roe matrix is available, and another one based on the exact Riemann problems when the solution is explicitly known. We have observed that the strategy based on the Roe matrix can fail when the intermediate states appearing in the solution of the linearized Riemann problem do not coincide with the exact intermediate states appearing in the solution of the exact Riemann problem. The only important part in the in-cell reconstruction procedure is to know the exact intermediate states, so it is not necessary to know the entire structure of the exact Riemann problem. The extension to high-order accuracy can be done through the Taylor expansion and the application of the Cauchy-Kovalevski procedure. The idea behind the exact in-cell reconstruction used for capturing the two shocks for the simplified shallow water model can be extended for models with more waves appearing in their Riemann problems.

\section*{Acknowledgments}
The  research  of CP, MC  and  EPG  was  partially  supported  by  the  Spanish  Government(SG),  the  European  Regional  Development Fund(ERDF), the Regional Government of Andalusia(RGA), and the University of M\'alaga(UMA) through the projects  of  reference  RTI2018-096064-B-C21  (SG-ERDF),  UMA18-Federja-161  (RGA-ERDF-UMA), and P18-RT-3163 (RGA-ERDF). EPG was also financed by the Junior Scientific Visibility Program from the Foundation Math\'ematiques Jacques Hadamard for a stay of three month in the Laboratoire de Math\'ematiques de Versailles (LMV).
TML was supported by the Spanish  Government(SG) through the projects  of  reference RTI2018-096064-B-C22.

\clearpage

%\bibliographystyle{plain}
%\bibliography{references}
\bibliographystyle{abbrv}

\bibliography{ref_corrected}

\end{document}